\documentclass[a4paper,11pt]{amsart}
\usepackage[czech,english]{babel}
\usepackage{hyperref}
\usepackage{enumitem}
\usepackage{mathtools}
\usepackage{tikz}
\usetikzlibrary{babel, cd}
\usepackage{amsfonts}
\usepackage{amsmath}
\usepackage{amssymb, latexsym}
\usepackage{amsthm}
\usepackage[all]{xy}
\usepackage{geometry}

\hypersetup{hidelinks} 
\newtheorem{lemma}{Lemma}[section]
\newtheorem{proposition}[lemma]{Proposition}
\newtheorem{theorem}[lemma]{Theorem}
\newtheorem{corollary}[lemma]{Corollary}

\newtheorem{conjecture}[lemma]{Conjecture}

\theoremstyle{definition}
\newtheorem{example}[lemma]{Example}
\newtheorem{definition}[lemma]{Definition}
\newtheorem{notation}[lemma]{Notation}

\theoremstyle{remark}
\newtheorem{remark}[lemma]{Remark}

\date{}

\def\bal#1\eal{\begin{align}#1\end{align}}
\def\bas#1\eas{\begin{align*}#1\end{align*}}
\def\bit{\begin{itemize}}
\def\eit{\end{itemize}}
\def\bet{\begin{enumerate}}
\def\eet{\end{enumerate}}
\def\ed{\end{document}}

\def\cal{\mathcal}

\def\a{\alpha}
\def\b{\beta}
\def\g{\gamma}
\def\d{\delta}
\def\e{\varepsilon}
\def\f{\varphi}
\def\k{\kappa}
\def\l{\lambda}
\def\Lam{\Lambda}

\def\s{\sigma}

\def\w{\omega}

\def\Om{\Omega}

\def\del{\partial}
\def\adel{\ol{\partial}}
\def\DEL{\Delta}
\def\G{\Gamma}

\def\odel{\overline{\partial}}

\def\bC{{\mathbb C}}
\def\bN{{\mathbb N}}

\def\bR{{\mathbb R}}
\def\bZ{{\mathbb Z}}

\def\C{{\cal C}}
\def\D{{\cal D}}
\def\E{{\cal E}}
\def\F{{\cal F}}
\def\H{{\cal H}}

\def\X{{\cal X}}

\def\fG{{\mathfrak G}}

\def\fh{{\mathfrak h}}

\def\fF{{\mathfrak F}}
\def\fG{{\mathfrak G}}
\def\fX{{\mathfrak X}}
\def\fY{{\mathfrak Y}}

\def\co{\operatorname{co}}

\def\exd{\operatorname{d}}
\def\dim{\operatorname{dim}}

\def\haar{\mathrm{\bf h}}

\def\unit{\mathrm{U}}
\def\counit{\mathrm{C}}

\def\id{\operatorname{id}}
\def\ker{\operatorname{ker}}
\def\Ker{\operatorname{ker}}
\def\proj{\operatorname{proj}}

\def\spn{\operatorname{span}}

\def\vol{\operatorname{vol}}
\def\im{\operatorname{im}}

\newcommand{\floor}[1]{\lfloor #1 \rfloor}
\def\hol{^{(1,0)}}
\def\ahol{^{(0,1)}}

\def\inv{^{-1}}

\def\coby{\, \square_{H}}
\def\oby{\otimes}

\def\wed{\wedge}
\def\sseq{\subseteq}
\def\tl{\triangleleft}

\def\wt{\widetilde}
\def\wh{\widehat}
\def\ol{\overline}

\def\la{\left\langle}
\def\\la{\left\langle}
\def\ra{\right\rangle}
\def\>{\right\rangle}
\def\bs{\backslash}
\def\mto{\mapsto}

\def\cf3{\bC_q[F_3]}

\def\csu2{\bC_q[SU_2]}

\def\cs1{\bC P^{1}}
\def\cccp1{\bC_q[\bC P^{1}]}
\def\ccp2{\bC P^{2}}
\def\cp2{\bC_q[\bC P^{2}]}
\def\cpn{\bC_q[\bC P^{n}]}

\def\qf3{\bC_q[F_3]}
\def\wqf3{\Om^1_q[F_3]}
\def\qsu3{\bC_q[SU_3]}
\def\wqsu3{\Om^1_q[SU_3]}

\def\usl2{\mathcal{U}(\mathfrak{sl}(2))}

\def\ws2{\Om^1_q(S^2)}
\def\wsu2{\Om^1_q(SU_2)}

\def\cs{\bC_q[S^{n,r}]}
\def\hol{^{(1,0)}}
\def\ahol{^{(0,1)}}

\def\az{\ol{z}}

\def\n2{_{(-2)}}
\def\m2{_{(-2)}}
\def\m1{_{(-1)}}
\def\0{_{(0)}}
\def\1{_{(1)}}
\def\2{_{(2)}}
\def\3{_{(3)}}
\def\4{_{(4)}}
\def\5{_{(5)}}
\def\hol{^{(1,0)}}
\def\ahol{^{(0,1)}}

\def\ev{\operatorname {ev}}
\def\coev{\operatorname {coev}}

\def\sgmm{{}^{G}_M \hspace{-.030cm}{\mathrm{mod}}_0}

\def\lgmmm{{}^{G}_M \hspace{-.030cm}{\mathrm{Mod}}_0}
\def\lmhm{{^H\!\!\mathrm{Mod}}_0}

\def\alg{algebra~}
\def\algn{algebra}

\def\hk{Heckenberger--Kolb~}

\def\ncgn{noncommutative geometry}
\def\nc{noncommutative~}

\def\qhs{quantum homogeneous space~}

\def\st{such that~}

\def\wrt{with respect to~}

\def\mpr{maximal prolongation~}

\def\uqsl2{U_q(\frak{sl_2)}}
\def\tuqsl2{\wt{U}_q(\frak{sl}_2)}
\def\tu1sl2{\wt{U}_1(\frak{sl}_2)}

\def\gr{\bC_q[\text{Gr}_{n,r}]}

\def\cgr{\text{Gr}_{n,r}}

\def\sgmm{{}^{G}_M \hspace{-.030cm}{\mathrm{mod}}_0}

\DeclareMathOperator{\lin}{lin}
\DeclareMathOperator{\Lin}{Lin}
\DeclareMathOperator{\Hom}{Hom}
\DeclareMathOperator{\dgproj}{dgproj}

\DeclareMathOperator{\Mod}{Mod}
\DeclareMathOperator{\dgMod}{dg-Mod}

\DeclareMathOperator{\GrMod}{GrMod}

\DeclareMathOperator{\lproj}{lproj}
\DeclareMathOperator{\rproj}{rproj}

\def\wedgeev{\mathbin{\wedge_{{\mathrm {ev}}}}}

\renewcommand{\mod}{\operatorname{mod}}

\newcommand\psM{\prescript{}{M}}

\newcommand\psH{\prescript{H}{}}
\newcommand\psA{\prescript{A}{}}
\newcommand\psAM{\prescript{A}{M}}
\newcommand\psAN{\prescript{A}{N}}
\newcommand\psGM{\prescript{G}{M}}
\newcommand\lvee{\prescript{\vee}{}}
\newcommand\lwedge{\prescript{\wedge}{}}

\newcommand\psD{\prescript{}{\D}}
\newcommand\psAD{\prescript{A}{\D}}
\newcommand\psAE{\prescript{A}{\E}}

\DeclareMathOperator\lhVB{lhVB}
\DeclareMathOperator\rhVB{rhVB}

\DeclareMathOperator\dglproj{dg-lproj}
\DeclareMathOperator\ldgproj{dg-lproj}
\DeclareMathOperator\dgrproj{dg-rproj}
\DeclareMathOperator\rdgproj{dg-rproj}

\DeclareMathOperator\dgzlproj{dg_0-lproj}

\DeclareMathOperator\dgzrproj{dg_0-rproj}

\title[A Kodaira Vanishing Theorem for Noncommutative  K\"ahler Structures]{A Kodaira Vanishing Theorem for Noncommutative  K\"ahler Structures}

\author{R\'{e}amonn \'{O} Buachalla}
\address{Institute for Mathematics \\ Astrophysics and Particle Physcs (IMAPP) \\ Radboud University Nijmegen \\ Heyendaalseweg 135 \\ 6525 AJ Nijmegen \\ The Netherlands}\email{reamonnobuachalla@gmail.com}
\author{Jan \v{S}\v{t}ovi\v{c}ek}
\address{Charles University \\ Faculty of Mathematics and Physics \\ Department of Algebra \\ Sokolovsk\'{a} 83 \\ 186 75 Prague 8 \\ Czech Republic}\email{stovicek@karlin.mff.cuni.cz}
\author{Adam-Christiaan van Roosmalen}
\address{Adam-Christiaan van Roosmalen\\ Hasselt University \\ Agoralaan - gebouw D \\ B-3590 Diepenbeek}
\email{adamchristiaan.vanroosmalen@uhasselt.be}

\begin{document}

\bibliographystyle{amsplain}

\begin{abstract}
Using the framework of noncommutative K\"ahler structures, we generalise to the \nc setting the celebrated vanishing theorem of Kodaira for positive line bundles. The result is established under the assumption that the associated Dirac--Dolbeault operator of the line bundle is diagonalisable, an assumption that is shown to always hold in the quantum homogeneous space case. The general theory is then applied to the covariant K\"ahler structure of the \hk calculus of the quantum Grassmannians allowing us to prove a direct $q$-deformation of the classical Grassmannian Bott--Borel--Weil theorem for positive line bundles.
\end{abstract}

\maketitle

\tableofcontents

\section{Introduction}

The study of cohomological invariants of vector bundles is ubiquitous in differential geometry. Quite often, however, explicit calculations of cohomology groups can prove extremely difficult. One of the distinguishing features of complex, and in particular K\"ahler, geometry is the existence of a large and important family of vanishing theorems for the Dolbeault, or sheaf, cohomology of holomorphic vector bundles \cite{VTBOOK}. Remarkably many of these results extend in some form to the case of positive characteristic, where they have many important applications \cite{DeligneIllusie87,VTBOOK}. Another remarkable fact about this family of vanishing theorems is that they quite often find their most natural presentation in purely global terms.  Consider, for example, the global presentation of the definition of a holomorphic vector bundle given by the Koszul--Malgrange theorem \cite{KoszulMalgrange58}, and its coherent sheaf extension due to Pali \cite{Pali06} (see also \cite{Block10, Polishchuk03, Polishchuk04} for a similar approach using a derived category). Both these facts indicate that holomorphic  vanishing theorems should have a role to play in \ncgn, an area  which studies geometry over \nc rings, and for which local arguments are usually not available.  The goal of this paper is to prove a noncommutative generalisation of the prototypical vanishing theorem, namely the Kodaira vanishing theorem for positive line bundles. This will be done in the framework of differential calculi, noncommutative complex structures \cite{KLVSCP1,BeggsSmith13}, and the recently introduced notions of  \nc K\"ahler structures \cite{MMF3},  holomorphic  structures\cite{KLVSCP1, BeggsSmith13},  and  Chern connections \cite{BM}.

Our motivating family of examples come from the theory of quantum groups, and in particular  quantum group homogeneous spaces. Recall that the compact simply-connected homogeneous K\"ahler manifolds can be identified with the generalised flag manifolds. These spaces admit a direct quantum group deformation. Moreover,  in the cominicsule case (see \textsection\ref{section:Grassmannians}) they admit a direct $q$-deformed de Rham complex, as identified in the remarkable classification of Heckenberger and Kolb \cite{HK}. The calculus of the quantum Grassmannians has recently been shown to admit a covariant  K\"ahler  structure, unique up to real scaler multiple, with the extension to other cominiscule quantum flags conjectured \cite{KOBS}.  Moreover, the positivity of quantum Gramssmannian line bundles carries over directly from the classical setting, meaning that the general machinery of this paper can be directly applied to these examples.

As a first application, we take the proof of a Bott--Borel--Weil theorem for the quantum Grassmannians $\bC_q[\text{\rm Gr}_{n,r}]$. Since the Heckenberger---Kolb calculi have anti-holomorphic sub-complexes dual to a BGG-type exact sequence, it has been conjectured that  $\bC_q[\text{\rm Gr}_{n,r}]$ satisfies a $q$-deformation of the classical Bott--Borel--Weil theorem. The question of the Borel--Weil theorem has recently been resolved after a long series of works. This began with the influential paper of Majid on the noncommutative spin geometry of the Podle\'s sphere (see \cite{Maj}), where the space of holomorphic sections of the positive line bundles was calculated. This result was extended to all of $\cpn$ in a series of papers due to Landi, Khalkhali, Moatadelro, and van Suijlekom \cite{KLVSCP1,KKCP2,KKCPN}. Finally, the  extension to all the Grassmannians appeared recently in a paper due to Mrozinski and the first author (\cite{MMFCM}). The case of  higher line bundle cohomologies has been solved by Landi and D'Andrea for $\cpn$ for the trivial line bundle case in \cite{DAndreaLandi10}, and for all line bundles over $\bC_q[\bC P^{1}]$ and $\bC_q[\bC P^{2}]$ by Khalkhali and Moatadelro (\cite{KKCPN}).  In \textsection\ref{section:Grassmannians} we show that our noncommutative Kodaira vanishing theorem implies vanishing of higher cohomologies for all positive line bundles, implying a direct $q$-deformation of Bott--Borel--Weil for these line bundles.

A second application, which will appear in a separate work, is the construction of  spectral triples for quantum projective space, and more generally, for the cominiscule quantum flag manifolds. In the quantum homogeneous space case, the Dirac operator $\adel + \adel^\dagger$, together with its Hilbert space  of $L^2$-completed anti-holomorphic forms $L^2\left(\Om^{(0,\bullet)}\right)$, provides a natural candidate for a spectral triple. Such objects are of particular interest given the difficult relationship between Connes' definition of a spectral triple and the analytic properties of quantum group differential operators. The principle challenge in verifying Connes' axioms is in demonstrating that the spectrum of $\adel + \adel^{\dagger}$ has suitable growth properties. A major step toward this is confirming that $0$ appears as an eigenvalue only with finite multiplicity. By the noncommutative Hodge theory which follows from the existence of a noncommutative K\"ahler structure, vanishing of cohomologies implies this result directly. Moreover, since the index of any Dolbeault--Dirac operator can be shown to be equal to the Euler characteristic of the calculus, vanishing of higher cohomologies also implies that the operator has non-zero rank, and hence, non-zero associated $K$-homology class.

We now give a general presentation of the approach taken in the paper. Firstly, we follow the approach introduced in \cite{MMF3} which starts with a complex structure and a symplectic form and from these constructs the K\"ahler metric. This avoids the question of how to extend metrics from $1$-forms to higher forms, which can prove quite involved  in the noncommutative setting. Adopting this approach one finds that, with sufficient care, large parts of the proof of the Kodaira vanishing theorem can be carried over from the classical case (for example \cite{HUY}).

There are, however, two major exceptions to this. The first is that the usual local extension of the K\"ahler identities to the Nakano identities needs a reformulation.  We do this by introducing the novel notions of Lefschetz pairs and triples which abstract many of the features common to the twisted and untwisted cases, and allow us to nicely lift the usual global proof of the K\"ahler identities to the twisted case.  The second exception is that the notions of vector bundles and dual holomorphic structures are more subtle in the noncommutative case as left and right module structure no longer coincide.  We use the framework of differential graded modules to express holomorphic structures: a (left) holomorphic vector bundle is a projective left module over the Dolbeault dg algebra; the dual of a left holomorphic vector bundle is then the dual of the dg module with respect to the Dolbeault dg algebra.  We also allow for left holomorphic vector bundles which are dg bimodules of the Dolbeault dg algebra; this condition is related to the concept of bimodule connections (see Remark \ref{remark:BimoduleConnections}) and allows us to define invertible vector bundles (see \ref{definition:InvertibleVectorBundle}).  Under the additional condition that the calculus is factorisable, we show that the vector bundle $\Omega^{(n, 0)}$ of top holomorphic forms is invertible.  This leads to the definition of a Fano structure in \textsection\ref{subsection:Fano} and the vanishing of the $(0,\k)$-cohomologies, for the trivial line bundle, for all $k \geq 1$, see Corollary \ref{corollary:Fano}.

\bigskip

The paper is organised as follows. In \textsection\ref{section:Preliminaries} we recall the definitions  and basic results of noncommutative complex and K\"ahler structures, as well as the basic theory of hermitian and holomorphic bundles, and their associated Chern connections. 

In \textsection\ref{section:HolomorphicDuals} we discuss dualities for holomorphic vector bundles.  In Proposition \ref{proposition:HolShape}, we establish that a holomorphic vector bundle is a dg module over the Dolbeault dga whose underlying graded module is left projective and finitely generated in degree zero.  This allows us to describe duals of holomorphic vector bundles using dualities for dg modules (described in Appendix \ref{Section:Dualsdgas}).  The main results of \textsection\ref{section:HolomorphicDuals} are summarised in Remark \ref{remark:DualsForDGModules}, where we discuss the dual of a holomorphic vector bundle $\fX$ in terms of the dual of the underlying complex vector bundle $\X$.  We continue with a short discussion of holomorphic vector bundles which are equipped with the structure of a dg bimodule over the Dolbeault dg algebra.  This extra structure allows us to consider tensor products of holomorphic vector bundles.  An example of such a tensor product is the twisted de Rham complex of a vector bundle (see Notation \ref{notation:TwistedComplex}).  In 

In \textsection\ref{section:LefschetzPairs} we introduce the novel notions of Lefschetz pairs and Lefschetz triples. These structures, which are graded modules, endowed with a generalised Leibniz map, abstract some key features of symplectic and Hermitian manifolds. Surprisingly, even at this level of generality, one can still discuss Lefschetz decomposition, Hodge maps, and their interactions with the representation theory of $\mathfrak{sl}_2$.

In \textsection\ref{section:HermitianMetrics} we introduce the twisted de Rham complex of a holomorphic vector bundle. We show how to endow such d.g. modules with a Hodge map and a metric $g$ induced from a choice of Hermitian structure for the vector bundle. We carefully discuss positivity for the metric, and show that the adjoint of the Lefschetz map coincides with its Lefschetz pair dual. Finally,  an inner product is formed by composing $g$ with a choice of state on the $*$-\alg of $0$-forms, and it is shown that $\del_\F$ and $\adel_\F$ are adjointable with respect to this new inner product.

In \textsection\ref{section:Hodge} we generalise the results of Hodge theory in \cite{MMF3} to the twisted case, with the very important distinction that the covariance is no longer required and instead we simply require that the associated Dirac operator $\adel_\F + \adel_\F$ is diagonalisable as a linear operator. This gives the vital equivalence between harmonic foams and cohomology classes, central to all our later discussions of cohomology.  We then use this equivalence to establish a direct noncommutative generalisation of Serre duality for holomorphic vector bundles. 

In \textsection\ref{section:Nakano}, we give a new presentation of Chern connections and use this to generalise the Nakano identities to the noncommutative setting. These identities, which form a natural extension of the K\"ahler identities to the twisted setting, are key to our noncommutative generalisation of Kodaira vanishing, and form results of importance in their own right.  (The following theorem is Theorem \ref{theorem:Nakano} in the text.)

\begin{theorem}[Nakano identities]\label{theorem:NakanoIntroduction}
Let $(\Om^{(\bullet, \bullet)}, \odel, \del, \k)$ be a positive definite  K\"ahler structure.  Let $\operatorname{\bf f}$ be a state such that the integral $\int = \operatorname{\bf f} \circ \vol_\k$ is closed.  For an  Hermitian holomorphic vector bundle $(\F, \odel_\F, h)$ with Chern connection $\nabla_\F = \odel_\F + \del_\F$, we have
\begin{align}
[L_\F,\del_\F^{\dagger}] = i \adel_\F, & & [L_\F,\adel_\F^{\dagger}] = - i \del_\F, & &  [\Lambda_{\F},\del_{\F}] = i \adel_{\F}^\dagger, & & {[\Lambda_{\F},\adel_\F] = - i\del^\dagger_{\F}}.
\end{align}
\end{theorem}

We establish these identities using the framework of Lefschetz triples developed in \textsection\ref{section:LefschetzPairs}.

In \textsection\ref{section:Kodaira} we prove the main result of the paper, namely the Kodaira vanishing theorem for positive vector bundles. We begin by introducing an appropriate noncommutative definition of positivity.  Let $(\F, \odel_\F, h)$ be an Hermitian holomorphic vector bundle; we write $\nabla$ for the corresponding Chern connection.  As in the commutative case, one says that $\F$ is positive if $\nabla^2 = i\k$, for a positive definite K\"ahler form $\k$.  In our setting, we require some additional technical conditions on $\k$ (see Definition \ref{definition:PositivelineBundle}). However, these additional conditions are shown to automatically hold for the case of quantum homogeneous spaces. 

The Kodaira vanishing theorem is then established with respect to this definition (see Theorem \ref{theorem:Kodaira} in the text):

\begin{theorem}[Kodaira Vanishing]
Let $(\Om^{(\bullet, \bullet)}, \odel, \del)$ be a $*$-differential calculus with a complex structure.  For any positive vector bundle $(\F,\odel_\F, h)$
\bas
H^{(a,b)}(\F) = 0, & & \text{whenever } a+b > n.
\eas
\end{theorem}

Finally, in \textsection\ref{section:Grassmannians} we treat our motivating family of examples: the quantum Grassmannians. We recall the classification of their  line bundles, their essentially unique differential calculi and K\"ahler structures, as well as positivity of those line bundles indexed by the positive integers. This allows us to apply the framework of the paper to the quantum Grassmannians, and to prove a $q$-deformation of the Bott--Borel--Weil theorem for positive line bundles over the quantum Grassmannians (Theorem \ref{theorem:BorelBottWeil} in the text).

\section{Preliminaries}\label{section:Preliminaries}

Thoughout, we assume that all (Hopf) algebras and categories are $\mathbb{C}$-linear, unless explicitly stated otherwise.  Throughout, we fix a Hopf $*$-algebra $A$ and write $\psA\Mod$ for the category of all left $A$-comodules.  If $M$ is an $A$-comodule algebra, we write $\psAM\Mod$ for the category of left $A$-comodule left $M$-modules (see Appendix \ref{Section:ComoduleDualities} for details).  For notation concerning $A$-comodule differential graded algebras, we refer to Appendix \ref{Section:Dualsdgas}.

As our main applications concern quantum homogeneous spaces, we will often be interested in the case where $A = G$, where $G$ is a compact quantum group algebra (see Appendix \ref{subsection:QuantumHomogeneousSpace} for notation).  However, as our results hold for more general structures (at the cost of introducing some additional conditions which are automaically satisfied for quantum homogeneous spaces), we express our results in terms of a general Hopf $*$-algebra $A$.  We allow $A = \bC$ for a noncovariant case.

For notation used concerning differential $*$-algebras, we refer to Appendix \ref{subsection:DifferentialCalculi}.

\subsection{Complex Structures}

In this subsection we recall the basic definitions and results of complex structures for $A$-covariant differential $*$-algebras. For a more detailed introduction see \cite{MMF2}.

\begin{definition}\label{defnnccs}
Let $M$ be an $A$-comodule $*$-algebra and let $(\Om^{\bullet}, \exd)$ be an $A$-comodule differential $*$-calculus over $M$.  An {\em almost complex structure} for $(\Om^{\bullet}, \exd)$ is an $\bN^2_0$-graded algebra $\bigoplus_{(a,b)\in \bN^2_0} \Om^{(a,b)}$ for $\Om^{\bullet}$ such that, for all $(a,b) \in \bN^2_0$: 
\begin{enumerate}
\item \label{compt-grading}  $\Om^k = \bigoplus_{a+b = k} \Om^{(a,b)}$ in $\psAM\Mod_M$;
\item  \label{star-cond} $(\Om^{(a,b)})^* = \Om^{(b,a)}$.
\end{enumerate}
\end{definition}

Let $\del$ and $\ol{\del}$ be the unique homogeneous operators of order $(1,0)$, and $(0,1)$ respectively, defined by 
\bas
\del|_{\Om^{(a,b)}} \coloneqq \proj_{\Om^{(a+1,b)}} \circ \exd, & & \ol{\del}|_{\Om^{(a,b)}} \coloneqq \proj_{\Om^{(a,b+1)}} \circ \exd,
\eas
where $ \proj_{\Om^{(a+1,b)}}$, and $ \proj_{\Om^{(a,b+1)}}$, are the projections from $\Om^{a+b+1}$ onto $\Om^{(a+1,b)}$, and $\Om^{(a,b+1)}$, respectively.

As observed in \cite[\textsection 3.1]{BS}, the proof of the following lemma carries over directly from the classical setting \cite[\textsection 2.6]{HUY}.
\begin{lemma}\label{intlem} \cite[\textsection 3.1]{BS}
If $\bigoplus_{(a,b)\in \bN^2_0}\Om^{(a,b)}$ is an almost complex structure for a differential $*$-calculus $\Om^{\bullet}$ over an \alg $M$, then the following conditions are equivalent:
\begin{enumerate} 
\item $\exd = \del + \ol{\del}$;
\item $\del^2=0$;
\item $\adel^2 = 0$.
\end{enumerate}
\end{lemma}

\begin{definition}
{An almost complex structure for an $A$-covariant differential $*$-calculus  $\Om^{\bullet}$ that satisfies the conditions in Lemma \ref{intlem} is called \emph{integrable}.  An integrable almost complex structure is also called a \emph{complex structure}.  The corresponding double complex $(\bigoplus_{(a,b)\in \bN^2}\Om^{(a,b)}, \del,\ol{\del})$ is called its {\em Dolbeault double complex}.}
\end{definition}

\begin{proposition}
Let $(\bigoplus_{(a,b)\in \bN^2}\Om^{(a,b)}, \del,\ol{\del})$ be an $A$-covariant differential $*$-calculus with a complex structure.  We have
\bal \label{stardel}
\del(\w^*) = (\adel \w)^*, & &  \adel(\w^*) = (\del \w)^*, & & \text{ for all  } \w \in \Om^\bullet.
\eal
\end{proposition}

\begin{definition}\label{definition:HolomorphicVectorBundle}
Let $(\Om^{\bullet}, \d)$ be an $A$-covariant differential $*$-calculus with a complex structure.  We refer to the category $\psAM\lproj$ as the category of vector bundles over $(\Om^{\bullet}, \d)$.  An object from $\psAM\lproj$ is called a \emph{vector bundle}.

A \emph{$A$-equivariant left holomorphic module} is an $A$-comodule dg $(\Om^{(0,\bullet)}, \adel)$-module $\fF \in \psAD\dgMod$ for which the module action $\Om^{(0,\bullet)} \otimes_M \fF^0 \to \fF$ is a graded left $\Om^{(0,\bullet)}$-module isomorphism.

A \emph{left $A$-equivariant holomorphic vector bundle} is a left $A$-equivariant holomorphic module $\fF$ such that $\fF^0 \in \psAM\lproj$ is a vector bundle.  We define \emph{right holomorphic modules} and \emph{right holomorphic vector bundles} in a similar way.

\emph{Left anti-holomorphic modules} and \emph{left anti-holomorphic vector bundles} are defined similarly, replacing the dg algebra $(\Om^{(0,\bullet)}, \odel)$ by $(\Om^{(\bullet, 0)}, \del)$.
\end{definition}

\begin{remark}
\begin{enumerate}
\item A left holomorphic vector bundle is completely determined by the degree zero part $\F \in \psAM\lproj$ and the linear map $\odel_\F\colon \F \to \Om^{(0,1)} \otimes_M \F$.  In this case, the underlying graded module is $\Om^{(0,\bullet)} \otimes_M \F$ and the differential can be reconstructed by extending $\odel_\F$ to $\Om^{(0,\bullet)} \otimes_M \F$ using the graded Leibniz rule.
\item A morphism $\fF \to \fG$ between holomorphic modules is completely determined by its degree zero component $\fF^0 \to \fG^0$.  Furthermore, a morphism $\phi^0\colon \fF^0 \to \fG^0$ can be extended to a morphism $\phi\colon \fF \to \fG$ if and only if $\odel_\fG \circ \phi^0 = (\id \otimes \phi^0)\circ \odel_\fF$, that is, if and only if the following diagram commutes:
\begin{center}
\begin{tikzcd}
\fF^0 \arrow[r, "\phi^0"] \arrow[d, "\odel_\fF"] & \fG^0 \arrow[d, "\odel_\fG"] \\
\Om^{(0,1)}\otimes_M \fF^0 \arrow[r, "\id \otimes \phi^0"] & \Om^{(0,1)}\otimes_M \fG^0
\end{tikzcd}
\end{center}
\end{enumerate}
\end{remark}

\begin{notation}
As a left holomorphic vector bundle $\fF$ is completely determined by its degree zero part $\F = \fF^0 \in \psAM\lproj$ and the map $\odel_\F\colon \F \to \Om^{(0,1)} \otimes_M \F$, we often write $(\F, \odel_\F)$ for the holomorphic vector bundle $\fF$.  We refer to $\odel_\F\colon \F \to \Om^{(0,1)} \otimes_M \F$ as a \emph{holomorphic structure} on $\F.$  We define anti-holomorhic structures on $\F$ is a similar way.

We write $\psA\lhVB$ for the full subcategory of $\prescript{A}{\Om^{(0,\bullet)}}\dgMod$ spanned by the left holomorphic vector bundles.  The category $\psA\rhVB$ of right holomorphic vector bundles is defined in a similar way.
\end{notation}

\subsection{K\"ahler Structures}\label{subsection:KahlerStructure}

\begin{definition} An {\em $A$-equivariant K\"ahler structure} for a calculus $\Om^{\bullet}$ is a tuple $(\Om^{(\bullet, \bullet)}, \del,\ol{\del}, \k)$ where $\Om^{(\bullet,\bullet)}$ is a complex structure, and $\k$ is an $A$-coinvariant closed central real $(1,1)$-form, called the {\em K\"ahler form}, such that, \wrt the {\em Lefschetz operator}
\bas
L\colon\Om^\bullet \to \Om^\bullet,  & &   \w \mto \k \wed \w,
\eas
isomorphisms are given by
\bal \label{Liso}
L^{n-k}\colon \Om^{k} \to  \Om^{2n-k}, & & \text{ for all } 1 \leq k < n.
\eal
\end{definition} 

\begin{remark}
Since $\k$ is a central real form, $L$ is an $A$-comodule $M$-bimodule $*$-map.  
\end{remark}

\subsection{Complex conjugation for comodule algebras}

{Let $V$ be a complex vector space.  We write $\ol{V}$ for the \emph{complex conjugate vector space}, that is, the vector space $\ol{V}$ has the same underlying set and group structure as $V$ (although we will write $\ol{v}$ for a $v \in V$ if we interpret it as an element in $\ol{V}$), but the action of $\bC$ on $V$ is twisted by the complex conjugation.  Explicitly, for any $\ol{v} \in \ol{V}$ and any $\lambda \in \bC$, we have $\lambda . \ol{v} = \ol{\ol{\lambda} v}$ where $\ol{\lambda}$ is the complex conjugate of $\lambda$.

For an object $\F \in \psAM\Mod$, let $\ol{\F}$ be an object in $\psA\Mod_M$ defined as follows: the structure maps of $\ol{\F}$ are given by
\bas
\ol{\F} \oby M \to \ol{\F}, & & \ol{f} \oby m &\mto \ol{m^*f}, \\
{\ol{\F} \mto A \oby \ol{\F}}, & & \ol{f} &\mto (f\m1)^* \oby \ol{f\0}.
\eas
We call $\ol{\F}$ the {\em conjugate object} of $\F$.

\begin{proposition}\label{proposition:ModuleStar}
Let $M$ be an $A$-comodule $*$-algebra.
\begin{enumerate}
\item If $\F \in \psAM\Mod$, then $\ol{\F} \in \psA\Mod_M$, and
\item if $\F \in \psA\Mod_M$, then $\ol{\F} \in \psAM\Mod$.
\end{enumerate}
\end{proposition}

\begin{proof}
We only verify the first statement, the second statement can be proven in a similar way.  Let $m \in M$ and $f \in \F$.  We have:
\begin{align*}
\Delta_{\ol{\F}}^A(m\ol{f}) &= \Delta_{\ol{\F}}^A(\ol{fm^*}) \\
&= \left((fm^*)\m1\right)^* \oby \ol{fm^*}\0 \\
&= \left((f\m1) (m^*)\m1\right)^* \oby \ol{fm^*}\0 \\
&= m\m1 (f\m1)^* \oby m\0 \ol{f}\0 \\
&= \Delta_M^A(m) \Delta_{\ol{\F}}^A(\ol{f}),
\end{align*}
where we have used that $A$ is a $*$-bialgebra and $\F \in \psAM\Mod$.
\end{proof}

We may extend the correspondence $\F \mapsto \ol{\F}$ to a (conjugate-linear) functor by considering the following action on morphisms: given a morphism $\phi\colon \F \to \mathcal{G}$, we define $\ol{\phi} \colon \ol{\F} \to \ol{\mathcal{G}}$ via $\ol{\phi}(\ol{f}) = \ol{\phi(f)}$.

We now turn our attention to quantum homogeneous spaces.  In this case, we have the following version of Proposition \ref{proposition:ModuleStar}.

\begin{corollary}
The correspondence $\F \mto \ol{\F}$ gives a conjugate-linear functor \[\ol{(-)}\colon \lgmmm \to \lgmmm.\]
\end{corollary}

\begin{proof}
Let $\F \in \lgmmm$.  We verify that $\ol{\F} \in \lgmmm$.  For this, we need to verify that $M^+\ol{\F} = \ol{\F}M^+$.  Therefore, let $m \in M^+$ and $\ol{f} \in \ol{\F}$.  We have
\[ m\ol{f} = \ol{fm^*} = \sum_i \ol{m_i^* f_i} = \ol{f_i} m_i,\]
for some $m_i \in M^+, f_i \in \F$, where we have used that $m \in M^+$ if and only if $m^* \in M^+$.
\end{proof}

We now take a closer look at the interaction between the above functor and the equivalence $\Phi\colon \lgmmm \to \psH\Mod$ from Theorem \ref{theorem:Takeuchi}.

\begin{proposition}\label{proposition:TakeuchiConjugate}
The following diagram essentially commutes 
\[
\xymatrix{
\lgmmm    \ar[d]_{\ol{(-)}}     \ar[rrr]^{\Phi}                & & &  {\prescript{H}{}{\Mod}} \ar[d]^{\ol{(-)}} \\
\lgmmm    \ar[rrr]_{{\Phi}}     & & &   {\prescript{H}{}{\Mod}.}
}
\]
\end{proposition}
\begin{proof}
{
We look for a natural transformation from the composition of the lower branch to the composition of the upper branch, thus $\Phi \ol{(-)} \Rightarrow \ol{\Phi(-)}$.

For any $\F \in \lgmmm$, consider the exact sequence $0 \to M^+ \F \to \F \to \F / (M^+\F) \to 0$ of $\bC$-vector space maps.  As the functor $\ol{(-)}$ is exact, we get the short exact sequence $0 \to \ol{M^+ \F} \to \ol{\F} \to \ol{\F / (M^+\F)} \to 0.$

As such, for any $\F \in \lgmmm$, there is an isomorphism 
\begin{align*}
\ol{\F} / \ol{M^+ \F} \to \ol{\F / (M^+\F)} \colon && \ol{f} + \ol{M^+\F} \mto \ol{f+M^+\F}
\end{align*}
of $\bC$-vector spaces.  One easily checks that these are isomorphisms of left $H$-comodules  and that they are natural in $\F$.  Hence, there is a natural equivalence $\Phi \ol{(-)} \Rightarrow \ol{\Phi(-)}$, as requested.}
\end{proof}

\subsection{Holomorphic Hermitian structures and Chern connections}

\begin{definition}
An {\em $A$-covariant Hermitian structure} for an object $\F \in \psAM\lproj$ is an isomorphism $h_\F\colon\ol{\F} \to {\lvee\F}$ in $\psA\Mod_M$, satisfying 
\begin{align*}
h_\F(\ol{f})(g) = h_\F(\ol{g})(f)^*, && \forall f,g \in \F.
\end{align*}
An \emph{$A$-covariant Hermitian vector bundle} is an $A$-covariant vector bundle $\F \in \lgmmm$ together with an $A$-covariant Hermitian structure.
\end{definition}

\begin{example}\label{example:HermitianStructureOnM}
An $A$-comodule $*$-algebra $M$ has a natural Hermitian structure $h\colon \ol{M} \to \lvee M$ given by $h(\ol{m})(n) = n m^*$. 
\end{example}

\begin{remark}
As all Hermitian structures we consider in this paper are $A$-covariant Hermitian structures for equivariant vector bundles, we will often not mention the equivariance.
\end{remark}

\begin{definition}\label{definition:HermitianConnection}
Let $(\F,h_\F)$ be an Hermitian vector bundle.  For each $n \geq 0$, we consider the map 
\begin{align*}
\mathfrak{h}_\F\colon \bigoplus_{i+j = n}\Om^i \otimes_M \F \otimes_\bR \Om^j \otimes_M \F &\to \Om^{i+j} \\
\omega \otimes f \otimes \nu \otimes g &\mapsto \omega h_\F(\ol{g})(f) \wedge \nu^*.
\end{align*}
 An $A$-covariant connection $\nabla\colon \F \to \Om^1 \otimes_M \F$ is an \emph{Hermitian connection} if
\begin{align*}
\exd \mathfrak{h}(f, {g}) = \mathfrak{h}_\F(\nabla(f), {g}) + \mathfrak{h}_\F(f, {\nabla}({g})) && \forall f,g \in \F.
\end{align*}
\end{definition}

Recall that an \emph{$A$-covariant connection} on an $A$-covariant vector bundle $\F$ is a map $\nabla\colon \F \to \Om^1 \otimes_M \F$ satisfying the Leibniz rule.

\begin{proposition}[\cite{BM}] \label{prop:BM1}
Let $\F \in \psAM\lproj$ endowed with an $A$-covariant holomorphic structure $\adel_{\F}\colon \F \to \Om\ahol \oby_M \F$ and an $A$-covariant Hermitian structure $h\colon \ol{\F} \to \lvee \F$.  There exists a unique $A$-covariant Hermitian connection $\nabla\colon \F \to \Om^1 \oby_M \F$ obeying 
\bas
(\proj_{\Om^{(0,1)}} \oby \id) \circ \nabla_\F = \adel_{\F}
\eas
We call $\nabla_{\F}$ the {\em Chern connection} of the triple $(\F,\adel_\F,h)$.
\end{proposition}

\begin{proof}
The statement is proven in \cite{BM} when $A = \bC$, but the proof is compatible with the $A$-comodule structure.  Alternatively, we will re-establish the existence in Proposition \ref{proposition:ChernViaConjugates}.
\end{proof}

We will also use the following result from \cite{BM}.

\begin{proposition} \label{prop:BM2}
The curvature operator of a Chern connection satisfies 
$
\nabla^2(\F) \sseq \Om^{(1,1)} \oby \F.
$
\end{proposition}

\begin{remark}  We consider now the case of quantum homogeneous spaces.  From Takeuchi's equivalence, and the above discussions (see Proposition \ref{proposition:TakeuchiConjugate}), it is clear that an Hermitian structure on $\F$ is uniquely determined by an $H$-comodule isomorphism
\bas
\ol{\Phi(\F)} \cong  \lvee\Phi(\F).
\eas
Hence, for irreducible objects $\F \in \lgmmm$, there is a unique Hermitian structures, up to scalar multiple (see \cite[Theorem 11.27]{KSLeabh}).
\end{remark}

\section{Holomorphic vector bundles and their duals}\label{section:HolomorphicDuals}

Let $(\bigoplus_{(a,b)\in \bN^2}\Om^{(a,b)}, \del,\ol{\del})$ be a Dolbeault double complex of a differential $*$-calculus $\Om^{\bullet}$ and write $M = \Om^{0}$.  In this paper, we often consider dg modules whose underlying graded module is isomorphic to $\Om^{(0,\bullet)} \otimes_M \F, \Om^{(\bullet,0)} \otimes_M \F,$ or $\Om^{\bullet} \otimes_M \F$, for some $\F \in \psAM\lproj$.  In this section, we introduce a framework that captures these three examples. 

\subsection{\texorpdfstring{The category $\psAD\dgzlproj_\E$}{The category dg0-lproj}}\label{subsection:dg0lproj} Let $\D = (D, \exd)$ and $\E = (E, \exd)$ be $A$-equivariant dg algebras (see Appendix \ref{Section:Dualsdgas}).  We consider the full subcategory $\psAD\dgzlproj_\E$ of $\psAD\dglproj_\E$ consisting of those objects $\fF \in \psAD\dgzlproj_\E$ which are generated, as left $D$-module by finitely many elements in degree zero.  We define the full subcategory $\psAD\dgzrproj_\E$ of $\psAD\dgrproj_\E$ in a similar way.

\begin{remark}
Whether an $\fX \in \psAD\dgMod_\E$ lies in $\psAD\dgzlproj_\E$ is solely a property of the underlying left $D$-module structure, and not of the differential of $\fX$.
\end{remark}

\begin{proposition}\label{proposition:HolShape}
Let $\fX \in \psAD\dgMod_\E$.  The following are equivalent.
\begin{enumerate}
\item\label{enumerate:HolShape1} $\fX \in \psAD\dgzlproj_\E$,
\item\label{enumerate:HolShape2} $\fX^0 \in \prescript{A}{D^0}\lproj_{E^0}$ and the map $\mu\colon D \otimes_{D^0} \fX^0 \to \fX$ induced by the left $D$-action is an isomorphism of $D$-$E^0$-modules,
\item\label{enumerate:HolShape3} as a graded left $D$-module, $\fX$ is a direct summand of $D^{\oplus n}$, for some $n \geq 0$.
\end{enumerate}
\end{proposition}

\begin{proof}
Assume that (\ref{enumerate:HolShape1}) holds; we prove that (\ref{enumerate:HolShape2}) holds.  Consider the canonical graded $D$-$E^0$-module morphism $\mu\colon D \otimes_{D^0} \fX^0 \to \fX$.  As $\fX$ is generated in degree zero, $\mu$ is a surjection.  Since $\fX$ is projective as a left $D$-module, there is a split embedding $\nu\colon \fX \to D \otimes_{D^0} \fX^0$ as left $D$-modules.  As $\mu$ restricts to an isomorphism of left $D^0$-modules on the degree zero part, so does $\nu$.  This shows that the image of $\nu$ contains the degree zero part of $D \otimes_{D^0} \fX^0$.  From this we conclude that $\nu$ is surjective, and hence an isomorphism of left $D$-modules.  This implies that $\mu$ is an isomorphism, as required.

We now show that $\fX^0$ is projective as a left $D^0$-module.  Let $f\colon Y \to \fX^0$ be a epimorphism of left $D^0$-modules.  This induces an epimorphism $\id \otimes_M f\colon D \otimes_{D^0} Y \to D \otimes_{D^0} \fX^0$.  As $D \otimes_{D^0} \fX^0 \cong \fX$ is projective as a left $D$-module, there is a morphism $g\colon D \otimes_{D^0} \fX^0 \to D \otimes_{D^0} Y$ such that $(\id \otimes f) \circ g = \id$, so that $f \circ g^0 = \id$.  This implies that $\fX^0$ is projective as a left $D^0$ module.  We have established that (\ref{enumerate:HolShape2}) holds.

Assume that (\ref{enumerate:HolShape2}) holds; we prove that (\ref{enumerate:HolShape3}) holds.  As $\fX^0 \in \prescript{A}{D^0}\lproj$, we know that $\fX$ is a direct summand of a free $D^0$ module.  Consider a splitting $\fX^0 \to (D^0)^{\oplus n} \to \fX^0$ (for some $n \geq 0$).  Applying the functor $D \otimes_{D^0} -$ yields that the composition
\[\fX \cong D \otimes_{D^0} \fX^0 \to D^{\oplus n} \to D \otimes_{D^0} \fX^0 \cong \fX\]
is an isomorphism of graded left $D$ modules.  This establishes (\ref{enumerate:HolShape3}).

Finally, it is easy to see that (\ref{enumerate:HolShape3}) implies (\ref{enumerate:HolShape1}).  This finishes the proof.
\end{proof}

\begin{remark}\label{remark:BimoduleConnections}
Let $\fX \in \psAD\dgzlproj_\D$.  We write $\mu_l\colon D^1 \otimes_{D^0} \fX^0 \to \fX^1$ and $\mu_r\colon \fX^0 \otimes_{D^0} D^1 \to \fX^1$ for the maps induced by the left and right $D$-action on $\fX$.  As $\mu_l$ is invertible (see Proposition \ref{proposition:HolShape}), there is a $D^0$-bimodule morphism
\[\sigma\colon \fX^0 \otimes_{D^0} D^1 \stackrel{\mu_r}{\longrightarrow} \fX^1 \stackrel{\mu_l^{-1}}{\longrightarrow} D^1 \otimes_{D^0} \fX^0.\]
We also consider the following commutative diagram
\begin{center}
\begin{tikzcd}
{\fX^0} \arrow[r, equal] \arrow[d, "\exd_{\fX,l}"] &  {\fX^0} \arrow[d, "\exd_\fX"] \\
{\D^1 \otimes_{D^0} \fX^0} \arrow[r, "\mu_l"] &  {\fX^1} \arrow[leftarrow, r, "\mu_r"] & {\fX^0 \otimes_{D^0} D^1}
\end{tikzcd}
\end{center}
where $\exd_{\fX,l} = \mu^{-1}_l \circ \exd_\fX$.  For all $x \in \fX^0$ and $e \in D^0$, we have
\begin{align*}
\exd_{\fX,l}(xe)
&= \mu_l^{-1}(\exd_\fX(xe)) = \mu_l^{-1}(\exd_\fX(x)e + x \exd(e)) \\
&= \exd_{\fX,l}(x)e + \mu_l^{-1}(x \exd(e)) = \exd_{\fX,l}(x)e + \mu_l^{-1}\mu_r(x \otimes \exd(e)) \\
&= \exd_{\fX,l}(x)e + \sigma(x \otimes \exd(e)).
\end{align*}
Hence, the map $\sigma\colon \fX^0 \otimes_{D^0} D^1 \to D^1 \otimes_{D^0} \fX^0$ endows $\fX$ with the structure of a bimodule connection.
\end{remark}

\subsection{\texorpdfstring{Dualities for $\psAD\dgzlproj_\E$}{Dualities}}  We now apply the results from Appendix \ref{Section:Dualsdgas} to the full subcategory $\psAD\dgzlproj_\E$ of $\psAD\dgMod_\E$.  We start by considering the restricted contravariant functors $\lwedge(-) = \psD\lin(-,\D)\colon \psAD\dgzlproj_\E \to \psAE\dgMod_\D$ and $(-)^\wedge = \lin_\D(-,\D)\colon \psAE\dgzlproj_\D \to \psAD\dgMod_\E$.

\begin{proposition}\label{proposition:RestrictedDuals}
For any $\fX \in \psAD\dgzlproj_\E$, we have $\lwedge \fX \in \psAE\dgzrproj_\D$.
\end{proposition}

\begin{proof}
As $\fX \in \psAD\ldgproj_{\D}$, we know that $\prescript{}{D}\hom(\fX,D) \in \prescript{A}{\D}\rdgproj_\D$.  We claim that $\lwedge \fX$ is generated in degree zero.  As $\fX \in \psAD\dgzlproj_\E$, we know that $\fX$ is a direct summand of $D^{\oplus n}$ (as a graded left $D$-module), for some $n \leq 0$.  Hence, $\lwedge \fX$ is a direct summand of $\psD\hom(D^{\oplus n}, D) \cong D^{\oplus n}$ (as a graded right $D$-module).  This shows that $\lwedge \fX$ is finitely generated in degree zero.
\end{proof}

\begin{corollary}
The contravariant functors 
\[
\mbox{$(-)^\wedge\colon \psAE\dgzlproj_\D \to \psAD\dgzlproj_\E$ and $\lwedge(-)\colon \psAE\dgzlproj_\D \to \psAD\dgzlproj_\E$}\]
form a duality.
\end{corollary}
\begin{remark}\label{remark:DualsForDGModules}
Let $\fX \in \psAD\dgzlproj_\E$.  We write $\X$ for $\fX^0$.
\begin{enumerate}
\item As $\fX$ is finitely generated, $(\lwedge \fX)^i = \prescript{}{D^0}\hom(\X^, D^i) \cong \lvee \X \otimes_M D^i$, where we used that $\X$ is finitely generated and projective (see Theorem \ref{theorem:TensorWithDual}).  The differential of $\lwedge \fX$ is given by the graded commutator (see \ref{subsection:dgam}):
\begin{align*}
[\exd,-]\colon  \prescript{}{D^0}\hom(\X, D^i) \to \prescript{}{D^0}\hom(\X, D^{i+1}) && \phi \mapsto \exd_\D \circ \phi - \phi \circ \exd_\fX.
\end{align*}

\item As $\lwedge \fX \in \psAE\dgzlproj_\D$ (see Proposition \ref{proposition:RestrictedDuals}), there is a right $D$-module isomorphism $\lvee \X \otimes_{D^0} \D \cong \lwedge \fX$.  Explicitly, this isomorphism is given by
\begin{align*}
\lvee \X \otimes_{D^0} D^i &\to \prescript{}{D^0}\hom(\X, D^i) & \prescript{}{D^0}\hom(\X, D^i) &\to \lvee \X \otimes_{D^0} D^i\\
\phi \otimes \omega &\mapsto \phi(-)\omega, &\phi &\mapsto \sum_i \phi_i \otimes \phi(x_i).
\end{align*}
where $\sum_i \phi_i \otimes x_i \in \lvee \X \otimes \X$ is a dual basis element. 
\item Under the above isomorphism, the evaluation map is written as
\[\wedge_{\ev}\colon (\D \otimes_M \X) \otimes (\lvee \X \otimes_M \D) \to \D\colon (\omega \otimes x \otimes f \otimes \nu) \mapsto \omega f(x) \nu.\]
As the evaluation map is a cochain morphism, we have:
\[\forall \alpha \in D^i \otimes \X, \forall \beta \in \lvee \X \otimes D^j: \odel(\alpha \wedgeev \beta) = \odel_\X(\alpha) \wedgeev \beta + (-1)^i \alpha \wedgeev \odel_{\lvee \X}(\beta).\]
\end{enumerate}
\end{remark}

\subsection{\texorpdfstring{Tensor products in $\psAD\dgzlproj_\D$}{Tensor products}}  We start with an observation that implies that $\psAD\dgzlproj_\D$ is a monoidal subcategory of $\psAD\dgMod_\D$.

\begin{proposition}
Let $\fX \in \psAD\dgzlproj_\E$ and $\fY \in \psAE\dgzlproj_{\E'}$.  We have that $\fX \otimes_\E \fY \in \psAD\dgzlproj_{\E'}$.
\end{proposition}

\begin{proof}
We take the tensor product as dg modules over $\E$, so $\fX \otimes_\E \fY \in \psAD\dgMod_{\E'}$.  By Proposition \ref{proposition:HolShape}, there are $n, m \in \mathbb{N}$ such that there is an epimorphism $D^{\oplus nm} \cong D^{\oplus n} \otimes_E E^{\oplus m} \to D^{\oplus n} \otimes_E Y \to X \otimes_D Y$ of graded left $D$-modules.  As each of these epimorphisms splits, so does the composition.
\end{proof}

\begin{remark}\label{remark:AboutGradedTensors}
Let $\fY \in \psAD\dgzlproj$.  As $\fY$ is generated by $\fY^0$ as left $D$-module, we have, for any $\fX \in \psAD\dgMod_\D$ and any $k \in \mathbb{Z}$, that $(\fX \otimes_D \fY)^k \cong \fX^k \otimes_{D^0} \fY^0$.
\end{remark}

\begin{proposition}\label{proposition:LineBundleTensor}
Let $\fX \in \psAD\dgzlproj_\D$ and assume that $\lwedge\fX \in \psAD\dgzlproj_\D$.  The following are equivalent.
\begin{enumerate}
  \item\label{enumerate:RankOne} The evaluation map $\fX^0 \otimes_{D^0} \lvee (\fX^0) \to D^0$ is a $D^0$-bimodule isomorphism.
	\item\label{enumerate:lveeLineBundle} The evaluation map $\fX \otimes_D (\lwedge \fX) \cong \D$ is a $D$-bimodule isomorphism.
\end{enumerate}
\end{proposition}

\begin{proof}
Assume first that the evaluation map $\fX^0 \otimes_{D^0} \lvee (\fX^0) \to D^0$ is an isomorphism.  In this case, we find, for all $i \in \mathbb{Z}$ and using Remark \ref{remark:AboutGradedTensors} (recall that $\lwedge \fX \in \psAD\dgzlproj_\D$):
\begin{align*}
\left[\fX \otimes_D (\lwedge \fX)\right]^i \cong \fX^i \otimes_{D^0} \psM\Hom(\fX^0, D^0) \cong D^i \otimes_{D^0} \fX^0 \otimes_{D^0} \lvee (\fX^0),
\end{align*}
so that the evaluation $\fX \otimes_D (\lwedge \fX) \to \D$ is an isomorphism.

For the other implication, assume that the evaluation map $\fX \otimes_D (\lwedge \fX) \to \D$ is an isomorphism.  This implies that the degree zero component is an isomorphism as well, i.e.:
\[\fX^0 \otimes_{D^0} \prescript{}{D^0}\hom(\fX^0, D^0) \cong \fX^0 \otimes_{D^0} \lvee (\fX^0) \to D^0.\]
This shows that the evaluation map $\fX^0 \otimes \lvee (\fX^0) \to D^0$ is an isomorphism.
\end{proof}

\begin{definition}\label{definition:InvertibleVectorBundle}
A holomorphic vector bundle $\fX$ is called \emph{invertible} if it satisfies the conditions of Proposition \ref{proposition:LineBundleTensor}.
\end{definition}

\begin{remark}\label{notation:TwistedComplex}
Let $(\Om^{(\bullet,\bullet)}, \del,\ol{\del})$ be an $A$-covariant differential $*$-calculus.  We write $\D = (\Om^{(0,\bullet)}, \odel)$ for the Dolbeault dg algebra.  Note that the Dolbeault dg algebra is a sub-dg-algebra of $(\Om^\bullet, \odel)$ and hence $(\Om^\bullet, \odel) \in \psAD\dgzlproj_\D$ (see Proposition \ref{proposition:HolShape}).  Let $\fF = (\F, \odel_F)$ be a holomorphic vector bundle.  We write $(\Om^\bullet \otimes_M \F, \odel_\F)$ for the dg module $\Om^\bullet \otimes_\Om^{(0,\bullet)} \fF.$  Hence, this setup recovers the usual definition of a twisetd complex.  
\end{remark}

\section{Lefschetz pairs and Lefschetz triples}\label{section:LefschetzPairs}

In this section we introduce the notions of Lefschetz pairs and Lefschetz triples. Our motivation is twofold.  Firstly, we require a unified framework in which to discuss certain important features of K\"ahler structures.  Secondly, we wish to highlight the essentially algebraic nature of many of the constructions of symplectic and Hermitian geometry, hence offering some explanation of why these constructions carry over so smoothly to the noncommutative setting.

One motivating example of a Lefschetz triple is a $*$-differential calculus with a K\"ahler structure (see \textsection\ref{subsection:KLTs} below).  Other examples will be given in Corollary \ref{Corollary:UsefulLefschetzTriples} below.

\subsection{Lefschetz pairs}

We begin with the definition of a Lefschetz pair and its associated primitive elements, and then establish a noncommutative generalisation of Lefschetz decomposition.

\begin{definition} Let $R$ be an $A$-comodule algebra.  A {\em left Lefschetz pair} of {\em total degree} $2n$ over $R$ is a pair \mbox{$(X = \bigoplus_{k=0}^{2n} X^k, L)$} where $X$ is a graded left $A$-comodule left $R$-module and where $L$ is a homogeneous left $A$-comodule $R$-module map of degree $2$, called the {\em Lefschetz operator}, \st the following maps 
\bas
L^{n-k}\colon X^{k} \stackrel{\sim}{\rightarrow} X^{2n- k}, & & \text{  for all }  k = 0, \dots, n-1.
\eas
are isomorphisms.

A \emph{right Lefschetz pair} is defined in a similar way, replacing the left $R$-module structure by a right $R$-module structure.  If we do not specify left or right, then we assume that the Lefschetz pair is a left Lefschetz pair.
\end{definition}

\begin{definition}\label{defn:sympform}
For any (left or right) {Lefschetz pair} $(X = \bigoplus_{k=0}^{2n} X^k, L)$, its space of {\em primitive $k$-elements} is the $R$-module
\bas
P^k \coloneqq \{\a \in X^{k} \mid L^{n-k+1}(\a) = 0\},  \text{ ~ if } k \leq n,& & \text{ and } & & P^k \coloneqq 0, \text{ ~ if } k>n.
\eas
\end{definition}

This setup suffices to formulate and prove a direct generalisation of the classical Lefschetz decomposition.

\begin{proposition}[Lefschetz decomposition]\label{LDecomp}
For any (left or right) Lefschetz pair $(X = \bigoplus_{k=0}^{2n} X^k ,L)$ we have the (left or right) $R$-module decomposition
\bas
X^k \cong \bigoplus_{j \geq 0} L^j(P^{k-2j}).
\eas
\end{proposition}
\begin{proof}
Let us assume that the decomposition holds for some $k \leq n-2$. Consider the factorisation 
\begin{displaymath}
    \xymatrix{
       X^{k}\ar@/^2.3pc/[rrrr]^{L^{n-k}} \ar[rr]_{L}
      &  &  X^{k+2} \ar[rr]_{L^{n-k-1}} &  & X^{2n-k}.
        }
\end{displaymath}
Since $L^{n-k}\colon X^{k} \to X^{2n-k}$ is an isomorphism of $R$-modules, we have the $R$-module decomposition 
\[
X^{k+2} \cong \ker\left(L^{n-k-1}|_{X^{k+2}}\right) \oplus  L(X^{k}) = \ker\left(L^{n-(k+2)+1}|_{X^{k+2}}\right) \oplus   L(X^{k}).
\]
Since we have assumed that $k \leq n-2$, it holds that $\ker\left(L^{n-(k+2)+1}|_{X^{k+2}}\right)  = P^{k+2}$, and so, 
\[
A^{k+2} \cong P^{k+2} \oplus L(A^{k}) =  P^{k+2} \oplus \left(\bigoplus_{j \geq 0} L^{j+1}(P^{k-2j})\right)
=  \bigoplus_{j \geq 0} L^j(P^{k + 2 - 2j}).
\]
Since $X^0 = P^0$, and $X^1 = P^1$, it now follows from an inductive argument that the proposition holds for each space of forms of degree less than or equal to $n$.

For forms of degree greater than $n$, we see that, for $k=0,\dots, n$,
\[
  X^{2n-k} \cong L^{n-k}(X^{k}) \cong L^{n-k}\left(\bigoplus_{j \geq 0} L^j(P^{k-2j})\right) = \bigoplus_{j \geq n-k} L^{j}(P^{2n-k-2j}).
\]
Note next that for $0 \leq j < \floor{\frac{1}{2}(k-n)}$, where $\floor{\frac{1}{2}(k-n)}$ denotes the integer floor of $\frac{1}{2}(k-n)$, we have $2n-k-2j > n$, and so, by definition $P^{2n-k-2j} = 0$. For the case  $\floor{\frac{1}{2}(k-n)} \leq  j < n-k$, we have $j > n-(2n-k-2j)$, and so, $L^j(P^{2n-k-2j}) = 0$. Hence, as required, we have
$
 X^{2n-k} \cong
  \bigoplus_{j \geq 0} L^{j}(P^{2n-k-2j}).
$
\end{proof}

\subsection{The Hodge map}

Building on the Lefschetz decomposition in the previous subsection, we introduce a noncommutative generalisation of the Hodge operator and proceed to establish some of its essential properties.

\begin{definition} \label{HDefn}
The {\em Hodge map} of a (left or right) Lefschetz pair $(X = \bigoplus_{k=0}^{2n} X^k ,L)$ is a left $A$-comodule (left of right) $M$-module morphism $\ast_H\colon X \to X$ uniquely defined by
\begin{align}\label{equation:OldHodge}
\ast_{H}(L^j(\a)) = (-1)^{\frac{k(k+1)}{2}}i^{k}\frac{j!}{n-j-k!}L^{n-j-k}(\a), & & \a \in  P^k.
\end{align}
\end{definition}

\begin{remark}\label{remark:HodgeFreedom}
The definition of the Hodge operator allows for some flexibility.  For example, the definition in \cite[Definition 4.11]{MMF3} is similar to the one in Definition \ref{HDefn}, but uses quantum factorials.  As there is no clear deformation parameter in our setting, we use use the regular factorials.
\end{remark}

As a first consequence of the definition, we establish direct generalisations of two of the basic properties of the classical Hodge map.

\begin{lemma}\label{lem:Hodge} It holds that
\begin{enumerate}
\item\label{enumerate:HodgeInvolutive} $\ast_H^2(\a) = (-1)^k(\alpha)$, \text{ for all } $\a \in X^k$, and so $\ast_H$ is an isomorphism,
\item $\ast_H(X^{k}) = X^{2n-k}$,  
\end{enumerate}
\end{lemma}

\begin{proof}
\begin{enumerate}
\item  By the Lefschetz decomposition, it suffices to prove the result for a form of type $L^j(\a)$, for $\a \in P^{(a,b)} \sseq P^k$, $j \geq 0$. From the definition of $\ast_H$ we have that
\begin{align*}
 &\ast^2_H(L^j(\a))  \\
&= (-1)^{\frac{k(k+1)}{2}}i^{k}\frac{j!}{(n-j-k)!}\ast_H\big(L^{n-j-k}(\a)\big) \\
                                &=  (-1)^{\frac{k(k+1)}{2}}i^{2k} \frac{j!}{(n-j-k)!}(-1)^{\frac{k(k+1)}{2}}\frac{(n-j-k)!}{(n-(n-j-k)-k)!}L^{n-(n-j-k)-k}(\a)\\
&= (-1)^k L^j(\a). 
\end{align*}

\item The inclusion $\ast_H(X^{k}) \subseteq X^{2n-k}$ follows from the definition of $\ast_H$ and the fact that $L$ is a degree $2$ map.  The fact that the inclusion is an identity follows from (\ref{enumerate:HodgeInvolutive}).
\end{enumerate}
\end{proof}

\subsection{The Lefschetz identities}

Motivated by the case of classical symplectic manifolds \cite{Yau}, and in particular Hermitian manifolds, \cite{HUY}, we now introduce two new operators, which when taken together with $L$, will induce the structure of an $\mathfrak{sl}_2$-representation on the Lefschetz pair $(X = \bigoplus_{k=0}^{2n} X^k ,L)$. We write $\proj_{X^k}\colon X \to X^k$ for the canonical projection.  The {\em counting operator} is defined by
\begin{align*}
H\colon X \to X, & & \a \mto  \sum_{k=0}^{2n}(k-n) \proj_{X^k}(\a).
\end{align*}
The {\em dual Lefschetz operator} is defined to be the map
$
\Lambda \coloneqq \ast_H\inv \circ L \circ \ast_H.
$

\begin{lemma}\label{lemma:PowersOfLefschetz}
Let $j,k \geq 0$.  For all primitive $\alpha \in P^k$, we have 
\[\Lambda L^j(\alpha) = j(n-j-k+1)L^{j-1}(\alpha).\]
\end{lemma}

\begin{proof}
We have
\begin{align*}
\Lambda(L^j(\a)) &= L \circ \ast_{H} \inv  \circ L \circ \ast_{H}(L^j(\a))  \\
&= \ast_{H} \inv \circ L\left((-1)^{\frac{k(k+1)}{2}}i^{k}\frac{j!}{(n-j-k)!} L^{n-j-k}(\a) \right)\\
&= (-1)^{\frac{k(k+1)}{2}}i^{k}\frac{j!}{(n-j-k)!} \ast_{H}\inv \circ L^{n-j-k+1}(\a)\\ 
&= (-1)^{\frac{k(k+1)}{2}+k}i^{k}\frac{j!}{(n-j-k)!}  \ast_H  \circ L^{n-j-k+1}(\a)\\
&=  j(n-j-k+1) L^{j-1}(\a),
\end{align*}
as required.
\end{proof}
  
\begin{proposition}[Lefschetz identities]\label{LIDS}
We have the relations
\bas
[H,L] = 2 L, & & [L,\Lambda] = H, & & [H,\Lambda] = - 2  \Lambda.
\eas
\end{proposition}
\begin{proof}
Beginning with the first relation, for $\a \in A^k$,
\begin{align*}
[H,L](\a) = & H L(\a) - LH(\a) \\
                                  = & \big((k+2-n) - (k-n)\big)L(\a)\\
                                  = &\, 2 L(\a).
\end{align*}
The derivation of the third relation is analogous.  

Finally, for the second relation, let $\alpha \in P^k$.  From Lemma \ref{lemma:PowersOfLefschetz}, we obtain (for all $j \geq 0$):
\begin{align*}
L\Lambda(L^j(\a)) &= j(n-j-k+1) L^j(\a), \\
\Lambda L(L^j(\a)) &= (j+1)(n-j-k) L^j(\a).
\end{align*}
Hence,
\begin{align*}
[L,\Lambda]L^j(\a)
 &=  \big(j(n-j-k+1) - (j+1)(n-j-k)\big) L^j(\a)\\
&= (2j+k-n) L^j(\a)  \\
&=  H L^j(\a).
\end{align*}
The statement then follows from the Lefschetz decomposition.
\end{proof}

The following decomposition of $X$ into irreducible $\mathfrak{sl}_2$-representations is now immediate.  Note that by combining the appropriate irreducible representations, the Lefschetz decomposition (Proposition \ref{LDecomp}) can be reproduced.

\begin{corollary} \label{cor:irreps}
Given a Lefschetz pair $(X = \bigoplus_{k=0}^{2n} X^k ,L)$, there is a decomposition of $X$ into irreducible $\mathfrak{sl}_2$ given by:
\bas
X = \bigoplus_{j\geq 0} L^j(\a),  & & \a \in P^{k}, \text{ for } k = 0, \dots, n.
\eas
Put differently, the primitive forms are the lowest weight vectors of $X$.
\end{corollary}

\subsection{Lefschetz triples}

Building on the definition of a Lefschetz pair, we define the notion of a Lefschetz triple. This structure  abstracts some of the essential features of the interaction of a symplectic, or Hermitian manifold, with the exterior derivative of the manifold. 

\begin{definition}
A {\em Lefschetz triple} of {\em degree $2n$} is a triple $(X \cong \bigoplus_{k = 0}^{2n} X^k, L, \exd)$ where $(X = \bigoplus_{k = 0}^{2n} X^k,L)$ is a  Lefschetz pair of  degree $2n$, and $\exd$ is a degree $1$ linear map  \st
$
[L,d] = 0.
$
\end{definition}

We now derive two identities, which hold for any Lefschetz triple, both of which play a crucial role in the extension of the K\"ahler identities to the twisted case.

\begin{proposition} \label{prop:ltids}
For every nonzero $\a \in P^k$, there exist unique forms $\a_0 \in P^{k+1}$ and $\alpha_1 \in P^{k-1}$, such that (for all $j \geq 0$):
\begin{enumerate} 
\item  $\exd\big(L^j(\a)\big) = L^j(\a_0) + L^{j+1}(\a_1)$,
\item $[\Lambda,\exd](L^j(\a)) = -j L^{j-1}(\a_0) + (n-j-k+1) L^j(\a_1)$.
\end{enumerate}
\end{proposition}

\begin{proof}
\begin{enumerate}
\item We prove the identity for the case of $j=0$, the case of higher $j$ following directly from our assumption that $[L,\exd] = 0$. Using the Lefschetz decomposition, $\exd \a \in A^{k+1}$ can  be written as 
\begin{align*}
\exd \a =  \sum_{j \geq 0} L^j(\a_j), & & \a_j \in P^{k+1-2j}.
\end{align*}
Since $L$ commutes with $\exd$ and $L^{n-k+1} (\a)= 0$, we have $
0 =  \sum_{j \geq 0} L^{n - k + j+1}(\a_j). 
$
Moreover, since the Lefschetz decomposition is a direct sum decomposition, 
\begin{align*}
L^{n-k+j+1}(\a_j) = 0, & &  \text{ for all } j \geq 0.
\end{align*}
Now it is only for  $j \leq 1$ that $\a_j$ can be contained in $\ker(L^{n-k+j+1})$. Hence  $\a_j = 0$ for all $j > 2$, and the required identity for $\exd$ follows.  Uniqueness of $\a_0$ is clear.  Uniqueness of $\a_1$ follows from it being a form of degree at most  $n-1$ and $L$ having trivial kernel in the space of such forms. 

\item We begin by finding an explicit description for the action of the left-hand side of the  proposed identity.  For $\a \in P^k$, and $j \geq 0$, We use Lemma \ref{lemma:PowersOfLefschetz} to find:
\bas
\Lambda \exd(L^j(\a)) = & \Lambda (L^j(\exd \a)) =  \Lambda L^j(\a _0+L(\a _1)) =  \Lambda L^j(\a _0) + \Lambda L^{j+1}(\a _1)\\
= & \, j(n-j-k) L^{j-1}(\a _0) +  (j+1)(n-j-k+1) L^j(\a _1), 
\eas
and
\bas
 \exd \Lambda(L^j(\a)) = j(n-j-k+1)\big(L^{j-1}(\a _0) + L^{j}(\a _1)\big).
\eas
Combining these two result gives:
\[ [\Lambda,\exd](L^j(\a))  = -j L^{j-1}(\a _0) + (n-j-k+1) L^j(\a _1), \]
as required.
\end{enumerate}
\end{proof}

\subsection{Lefschetz triples from K\"ahler structures}\label{subsection:KLTs}

A prime example of a Lefschetz triple comes from an $A$-covariant K\"ahler structure $(\Om^{(\bullet,\bullet)}, \k)$ for a calculus $\Om^{\bullet}$ of total degree $2n$ over $M$ (see \ref{subsection:KahlerStructure}).  Here, the Lefschetz map is given by $L = \kappa \wedge -\colon \Om^{(\bullet, \bullet)} \to \Om^{(\bullet, \bullet)}$.

\begin{remark}
The triples $(\Om^{\bullet}, L, \odel), (\Om^{\bullet}, L, \del),$ and $(\Om^{\bullet}, L, \exd)$ are simultaneously left and right Lefschetz triples (i.e. $L$ is an $M$-bimodule morphism).  Moreover, as $\k$ is central, $L$ is an $\Omega^\bullet$-bimodule morphism.
\end{remark}

In general, the Hodge map, as given in Definition \ref{HDefn} does not commute with the $*$-operation.  We consider the following variation on the Hodge map:
\begin{align}\label{equation:NewHodge}
\ast_{H}(L^j(\a)) = (-1)^{\frac{k(k+1)}{2}}i^{a-b}\frac{j!}{n-j-k!}L^{n-j-k}(\a), & & \a \in P^{(a,b)} \sseq P^k.
\end{align}
Note that the power of $i$ now takes into account the bidegree of a primitive form. As is easily checked, the results of Lemma \ref{lem:Hodge}, Proposition \ref{LIDS}, and Corollary \ref{cor:irreps} hold for this new map  (compare with Remark \ref{remark:HodgeFreedom}).  Moreover, this rescaling, together with the reality of $\k$, is easily seen \cite[Lemma 4.12.4]{MMF3} to imply that the Hodge map commutes with the $\ast$-map of the calculus.  Hence, \eqref{equation:NewHodge} can be viewed as the more natural map to consider in this setting.

\begin{definition}\label{definition:VolumeForm}
We write $\vol_\k$ (or just $\vol$ if $\k$ is understood) for the $M$-bimodule isomorphism $*_H|_{\Omega^{2n}}\colon \Omega^{2n} \to \Omega^0$.
\end{definition}

\begin{remark}
A calculus $\Om^{\bullet}$ with a complex and a K\"ahler structure $(\Om^{(\bullet,\bullet)}, \k)$ is always covariantly $*$-orientable.
\end{remark}

\begin{remark}
As $a=b=n$ in Definition \ref{definition:VolumeForm}, the formulas \eqref{equation:OldHodge} and \eqref{equation:NewHodge} coincide.
\end{remark}

\section{Hermitian metrics and forms}\label{section:HermitianMetrics}

In this section, we start with an Hermitian vector bundle $(\F, h)$ over a differential $*$-calculus endowed with K\"ahler structure.  We use this K\"ahler structure to produce a left Lefschetz pair (Lemma \ref{prop:lefadjs}), and subsequently use the Hermitian structure to introduce an Hermitian metric and Hermitian inner product on $\Om^\bullet \otimes_M \F$ (see Definition \ref{definition:HermitianMetric} and Proposition \ref{proposition:HermitianInnerProduct}). 

\subsection{\texorpdfstring{The Lefschetz pair $(\Om^\bullet \oby_M \F, L_{\F})$ and the metric $g_\F$}{Lefschetz pairs and metrics} }
We start by noting that the K\"ahler structure on the calculus $\Om^\bullet$ induces the structure of a left Lefschetz pair on $\Om^\bullet \otimes_M \F$.  The proof of the following proposition is evident.

\begin{proposition} \label{prop:lefadjs}
For any vector bundle $\F \in \psAM\lproj$, we have the following:
\begin{enumerate}
\item 
A Lefschetz pair is given by $\big(\Om^\bullet \oby_M \F, L_\F\big)$, where
\begin{align*}
L_\F\colon \Om^\bullet \otimes_M \F \to \Om^\bullet \otimes_M \F\colon && \omega \otimes f \mapsto L(\omega) \otimes_M f.
\end{align*}

\item Denoting by $\Lambda_\F\colon \Om^\bullet \otimes_M \F \to \Om^\bullet \otimes_M \F$, the dual Lefschetz operator of the left Lefschetz pair $\big(\Om^\bullet \oby_M \F, L_\F\big)$, we have the identity
$
\Lambda_\F = \Lambda \oby_M \id_\F.
$

\item Denoting by $\ast_\F\colon \Om^\bullet \otimes_M \F \to \Om^\bullet \otimes_M \F$ the Hodge operator of the left Lefschetz pair $\big(\Om^\bullet \oby_M \F, L_\F\big)$, we have the identity
$
\ast_\F = \ast_H \oby_M \id_\F.
$

\end{enumerate}
\end{proposition}

\begin{remark}
Starting from a right vector bundle $\F \in \psA\rproj_M$, an analogous argument shows that there is a right Lefschetz pair $\big(\Om^\bullet \oby_M \F, L_\F\big)$.  Moreover, the associated dual Lefschetz and Hodge operators are given by $\id_{\lvee\!\F} \oby \Lambda$ and $\id_{\lvee\!\F} \oby \ast_H$, respectively.
\end{remark}

\begin{definition}\label{definition:TwistedHodge}
For an Hermitian vector bundle $(\F,h)$, we introduce the $\bR$-linear isomorphism 
\begin{align*}
C_h\colon \Om^\bullet \oby_M \F \to \lvee \! \F \oby_M \Om^\bullet, & & \omega \otimes f \mto h(\ol{f}) \otimes \omega^*.
\end{align*}
Moreover we denote
\begin{align*}
\bar{\ast}_\F \coloneqq \ast_\F \circ C_h &\colon \Om^{\bullet} \otimes_M \F \to \lvee\!\F \otimes_M \Om^\bullet\colon &
\omega \otimes f &\mto h(\ol{f}) \otimes \ast_H(\omega^*)  \\
\bar{\ast}_{\lvee \! \F} \coloneqq \ast_{\lvee \! \F} \circ C_h\inv &\colon \lvee\! \F \otimes_M \Om^\bullet \to \Om^{\bullet} \otimes_M \F \colon &\phi \otimes \omega &\mto \ast_H(\omega^*) \otimes \ol{h^{-1}(\phi)}
\end{align*}
\end{definition}

We now collect some properties of the Hodge map.  The proofs are completely analogous to the untwisted case, and so, we omit them.

\begin{lemma}\label{Hodgeprop} For any left Hermitian vector bundle $\F$, we have:
\begin{enumerate}

\item $\ol{\ast}_\F\left(\Om^{(a,b)} \oby_{M} \F\right) = \lvee\! \F \otimes_{M} \Om^{(n-a,n-b)}$, 
\item $\bar{\ast}_{\lvee \F}\left(\lvee\!\F \oby_M \Om^{(a,b)}\right) = \Om^{(n-a,n-b)} \oby_{M}  \F$, 
 \label{astpq}

\item $\bar{\ast}_{\lvee\!\F} \circ {\bar{\ast}_\F}(\alpha) = (-1)^{|\alpha|} \alpha$, for all homogeneous $\alpha \in \Om^{\bullet} \otimes_M \F$,
\item $\bar{\ast}_{\F} \circ \bar{\ast}_{\lvee\!\F} (\beta) = (-1)^{|\beta|} \beta$, for all homogeneous $\beta \in \lvee\!\F \otimes_M \Om^{\bullet}$,

\item $\bar{\ast}_\F$ is an isomorphism of real vector spaces,

\item $C_h(\omega \wedge \alpha) = (-1)^{|\alpha||\omega|} C_h(\alpha) \wedge \omega^*,$ for all homogeneous $\alpha \in \Om^{\bullet} \otimes_M \F$ and $\omega \in \Om^\bullet$.

\end{enumerate}
\end{lemma}

In analogy with the untwisted case we now introduce a metric on the twisted complex.

\begin{definition}\label{definition:HermitianMetric}
Let $\F \in \psAM\lproj$ be a vector bundle and let $h\colon \ol{\F} \to \lvee \! \F$ be an Hermitian structure.  The {\em Hermitian metric} of $(\F, h)$ is the map  
\bas
g_\F\colon (\Om^\bullet \oby_M \F) \oby_{\bR} (\Om^\bullet \oby_M \F) \to M\colon & & (\a,\b) \mto \vol(\a \wedgeev \ol{\ast}_{\F}(\b)).
\eas
\end{definition}

Recall from Example \ref{example:HermitianStructureOnM} that $M$ has a canonical Hermitian structure induced by the $*$-map.  For this special case we will usually denote $g \coloneqq g_M$.

As one would want from any sensible definition of Hermitian metric, $g_\F$ is conjugate symmetric, as the following proposition shows. 

\begin{proposition} \label{prop:metricstarsym}
For any $A$-covariant Hermitian vector bundle $(\F,h)$, we have
\bas
g_\F(\a, \b) = g_\F(\b,\a)^* & & \text{ for all $\alpha, \beta \in \Om^\bullet \otimes_M \F$}.
\eas
\end{proposition}

\begin{proof}
For all $\omega \otimes a, \nu \otimes b \in \Om^\bullet \otimes_M \F$, we have
\bas
g_\F(\omega \otimes a, \nu \otimes b) &= \vol\left(\omega \wed \left(h(\ol{b})(a)\right)  \ast_H(\nu^*)\right) \\ 
&= \vol\left(\omega \wed   \ast_H\left(\left(h(\ol{b})(a)\right) \nu^*\right)\right) \\ 
&= g\left(\w, \nu \left(h(\ol{b})(a)\right)^*\right)\\
&= g\left(\w, \nu \left(h(\ol{a})(b)\right)\right),
\eas
where the last equality follows from the fact that $h$ is Hermitian. Recalling from \cite[Corollary 5.2]{MMF3}  that $g(\omega, \nu m) = g(\nu m, \omega)^*$, for all $m \in M$, we see that 
\[
g_\F(\omega \otimes a, \nu \otimes b) = g\left(\nu\left(h(\ol{a})(b)\right), \w \right)^* = g(\nu \oby b, \w \oby \a)^*,
\]
Hence $g_\F(\a, \b) = g_\F(\b,\a)^*$, for all $\alpha, \beta \in \Om^\bullet \otimes_M \F$, as required.
\end{proof}

\begin{proposition}\label{proposition:LefschetzAdjoint}
The Lefschetz and dual Lefschetz maps are adjoint with respect to $g_\F$, that is, it holds that  
\bas
g_\F\left(L_\F(\a), \b\right) = g_\F\left(\a, \Lambda_\F(\b)\right), & & \text{ for all } \a,\b \in \Om^\bullet \oby_M \F.
\eas
\end{proposition}

\begin{proof}
For any  $\omega \otimes a \in \Om^\bullet \otimes_M \F$ and $b \otimes \nu \in \Om^\bullet \otimes_M \F$, it holds that
\begin{align*}
\langle \omega \otimes a, \Lambda_\F(\nu \otimes b) \rangle_\F
&= \langle \omega \otimes a, \Lambda(\nu) \otimes b \rangle_\F \\
&= \vol \left( \omega \otimes a \wedgeev h(\ol{b}) \otimes \ast_H(\Lambda(\nu)^*) \right) \\
&= \vol \left( \omega \wed h(\ol{b})(a) \ast_H(\Lambda(\nu^*)) \right) \\
&= \vol \left( \omega \wed h(\ol{b})(a) \ast_H \circ \ast_H^{-1} \circ L \circ \ast_H (\nu^*) \right) \\
&= \vol \left( \omega \wed h(\ol{b})(a) L \circ \ast_H (\nu^*) \right).
\end{align*}
Recalling that $L$ is simply wedging by the K\"ahler form $\k$ and that $\k$ is by assumption central, we have that 
\begin{align*}
g_\F\left(\omega \otimes a, \Lambda_\F(\nu \otimes b)\right)  &= \vol \left( \omega \wed h(\ol{b})(a) \k \wed \ast_H (\nu^*) \right)\\
&= \vol \left( \k \wed \omega \wed h(\ol{b})(a) \ast_H (\nu^*) \right) \\
&= \vol \left( L( \omega )\wed h(\ol{b})(a) \ast_H (\nu^*) \right) \\
&= \vol \left( L(\omega) \otimes a \wedgeev h(\ol{b}) \otimes \ast_H(\nu^*) \right) \\
&= g_\F\left(L(\omega) \otimes a, \nu \otimes b\right) \\
&= g_\F\left(L_\F(\omega \otimes a), \nu \otimes b\right).
\end{align*}
\end{proof}


\subsection{States and inner products} \label{section:state sips}

In this subsection we recall the notion of a state for a $*$-\alg and show how states can be used to produce inner products from Hermitian metrics. The ability to produce inner products from Hermitian metrics is a crucial ingredient for Hodge theory. (As discussed in \textsection \ref{section:QHS} below, out motivating example of a state is the haar functional of a compact quantum group algebra.)

\begin{definition} For any $*$-\alg $M$,  we define its {\em cone of positive elements} to be 
\bas
M_{> 0} = \Big\{\sum_i \lambda_i m_i m_i^* \mid m_i \in M, \lambda_i \in \mathbb{R}_{\geq 0}\Big\} \setminus \{0\}.
\eas
\end{definition}

\begin{definition}
An Hermitian structure $h\colon \ol{\F} \to \lvee\!\F$ for a vector bundle $\F \in \psAM\lproj$ is said to be {\em positive definite} if
 \bas
h\big(\ol{f}\big)\big(f\big) \in M_{>0}, & & \text{ for all non-zero }  f \in \F.  
\eas
The associated Hermitian metric $g_\F$ is said to be {\em positive definite} if 
\bas
g_\F(\a,\a) \in M_{> 0}, & & \text{ for all nonzero } \a \in \F \otimes \Om^\bullet.
\eas
\end{definition}

Clearly, positive definiteness of $g_{\F}$ implies positive definiteness of $h$. We now restrict to the special case of  $M$ endowed with its has a canonical Hermitian structure.

\begin{definition}
A K\"ahler structure  is said to be {\em positive definite} if the metric associated canonical Hermitian structure of $M$ is positive definite. 
\end{definition}

The following proposition shows us that for the case of a positive definite K\"ahler structure,  positivity of $g_\F$ for an Hermitian vector bundle $\F$ follows directly from positivity of $h$.

\begin{proposition} \label{prop:metricpositivity}
Let $\left(\Om^{(\bullet,\bullet)}, \k\right)$ be a positive definite K\"ahler structure. For any vector bundle $\F \in \psAM\lproj$ with a positive definite Hermitian structure $h\colon \ol{\F} \to \lvee\!\F$, the associated  metric $g_\F$ is positive definite.
\end{proposition}
\begin{proof}
Without loss of generality, we can write $\a = \w \oby f$. Since by assumption $h(\ol{f})(f) \in M_{>0}$, we have that $h(\ol{f})(f) = \sum_{i=1}^m \l_l m_i m_i^*$, for some $m_i \in M$, and $\l_i \in \bR_{>0}$. It now follows that 
\bas
g_\F(\a,\a^*) &= g_\F\left(\w \oby_M f \oby_\bR \w \oby f \right)\\
                       &= \vol\left( \w \wed h(\ol{f})(f) \ast_H(\w^*)\right)\\
                       &= \sum_{i=1}^m \vol\left( \w \wed \l_i m_i a^*_i \ast_H(\w^*)\right)\\
                       &= \sum_{i=1}^m \l_i  \vol\left( (\w m_i)  \wed \ast_H((\w m_i)^*)\right)\\
                       &= \sum_{i=1}^m \l_i  g\left(\w m_i, \w m_i \right).
\eas
Since by assumption $g$ is positive definite, we have  $g(\w m_i, \w m_i) \in M_{>0}$, for all $i$. This in turn implies that $g_{\F}(\a,\a) \in M_{>0}$.
\end{proof}

\begin{definition}
A {\em state} on a unital $*$-\alg $M$ is a unital linear functional $\operatorname{\bf  f}\colon M \to \bC$ \st $\operatorname{\bf  f}\big(M_{> 0}\big) \sseq \bR_{>0}$. 
\end{definition}

The definition of a state, taken together with Proposition \ref{prop:metricpositivity} above, directly implies the following lemma.

\begin{proposition}\label{proposition:HermitianInnerProduct}
With respect to any choice of state $\operatorname{\bf  f}$, an inner product, which we call the associated \emph{Hermitian inner product}, is given by the pairing
\bas
\langle -,- \rangle_\F\colon {(\Om^\bullet \oby_M \F) \oby_{\bR} (\Om^\bullet \oby_M \F) \to \mathbb{C}}, & & (\a,\b) \mto \operatorname{\bf  f} \circ g_\F(\a,\b).
\eas
\end{proposition}
\begin{proof}
It follows from Proposition \ref{prop:metricstarsym}  that $\langle -,- \rangle_\F$ is a sesquilinear map. Positive definiteness follows from Proposition \ref{prop:metricpositivity} above and the definition of a state.
\end{proof}


\subsection{Codifferentials}

Through this subsection we fix a choice of  state $\operatorname{\bf  f}\colon M \to \bC$, and assume that $(\Om^{(\bullet,\bullet)},\k)$ is a K\"ahler structure with closed associated integral. These assumptions are enough to show that holomorphic, and anti-holomorphic, structure maps for Hermitian vector bundles are adjointable with respect to the inner product induced by $\operatorname{\bf  f}$.

\begin{lemma}\label{lemma:Stokes}
Let $(\F, \odel_\F) \in \psA \lhVB$ be a left holomorphic vector bundle.  For each $\alpha \in \Om^\bullet \otimes_M \F$ and $\beta \in \lvee \! \F \otimes_M \Om^{\bullet}$, we have:
\bas
\int_M \del(\a \wedgeev \b) = 0, & & \int_M \adel(\a \wedgeev \b) = 0
\eas
\end{lemma}
\begin{proof}
We may assume that $\a$ and $\b$ are homogeneous elements; let $\a \in \Om^{(a,b)} \otimes_M \F$ and $\b \in \lvee \F \otimes_M \Om^{(c,d)}$.  Note that $\int_M \omega = 0$ if $\omega \not\in \Om^{(n,n)}$.

Thus, the only interesting case is $(a+c, b+d) = (n,n-1)$.  In this case, $\adel(\a \wedgeev \b) \in \Om^{(n,n-1)}$, so that $\del(\a \wedgeev \b) = 0$.  Hence, $\adel(\a \wedgeev \b) = \exd(\a \wedgeev \b)$.  Using that the integral is closed gives us:
\[\int_M \adel(\a \wedgeev \b) = \int_M \exd(\a \wedgeev \b) = 0.\] 
The case of $\int_M \del(\a \wedgeev \b)$ is analogous. \qedhere
\end{proof}

\begin{proposition}\label{proposition:adelIsAdjointable}
Let $(\F, h)$ be an Hermitian vector bundle.
\begin{enumerate}
\item A holomorphic structure $\adel_{\F}\colon \F \to \Om^{(0,1)} \otimes_M \F$ is adjointable \wrt $\la\cdot,\cdot\ra_\F$, and
\bas
\adel_{\F}^\dagger = -\bar{\ast}_{\lvee \! \F} \circ \adel_{\lvee \F} \circ \bar{\ast}_{\F}.
\eas
\item An anti-holomorphic structure $\del_{\F}\colon \F \to \Om^{(1,0)} \otimes_M \F$ is adjointable \wrt $\langle \cdot,\cdot \rangle_\F$, and
\bas
\del_{\F}^\dagger = -\bar{\ast}_{\lvee \! \F} \circ \del_{\lvee \F} \circ \bar{\ast}_{\F}.
\eas
\end{enumerate}
\end{proposition}

\begin{proof}
We only prove the first statement, as the proof of the second statement is analogous.  Using Remark \ref{remark:DualsForDGModules} and Lemma \ref{lemma:Stokes}, we find
\begin{align*}
\int \adel_\F (\a) \wedgeev \b  &= - \int \adel(\a \wedgeev \b) + (-1)^{|\a|+1} \int  \a \wed_{\F} \adel_{\lvee \F}(\b)\\
&= (-1)^{|\a|+1} \int  \a \wed_{\F} \adel_{\lvee \F}(\b).
\end{align*}
 Adjointability of $\adel_{\F}$, as well as the given presentation, now follows from 
\begin{align*}
 \left\langle \a, - \prescript{}{\F}{\bar{\ast}} \circ \adel_{\lvee\F} \circ \bar{\ast}_\F(\b)\right\rangle_\F 
&=  - \int \a \wedgeev {\bar{\ast}}_\F \big(\prescript{}{\F}{\bar{\ast}} \circ \adel_{\lvee\F} \circ \bar{\ast}_{\F} (\b)\big) \\
&= (-1)^{|\beta|+2} \int \a \wedgeev \big(\adel_{\lvee\F} \circ \bar{\ast}_{\F} (\b)\big) \\
&= - \int \adel_\F(\a) \wedgeev \bar{\ast}_{\F} (\b)\\
&= \left\langle \adel_\F(\a), \b \right\rangle_\F.
\end{align*}
This establishes that $\adel_{\F^{\vee}}^\dagger = -\prescript{}{\lvee\!\F}{\bar{\ast}} \circ \, \adel_{\lvee\! \F} \circ \bar{\ast}_{\F},$ as required.
\end{proof}

Finally, we observe that in the untwisted case, which is to say when $\F = M$, the codifferentials of $\del$ reduces to 
\[
\del^\dagger = - \ast_H \circ \ast \circ \del \ast \circ \ast_H = - \ast_H \circ \adel \circ \ast_H,
\]
as established in \cite[Lemma 5.16]{MMF3}. The case of $\adel^\dagger$ is analogous.

\subsection{Hilbert spaces completions and the quantum homogeneous space case} \label{section:QHS}

In the quantum homogeneous space case $M = G^{\co(H)}$, let $\haar\colon G \to \bC$ be the haar functional corresponding to the cosemisimple Hopf \alg structure of $G$. Since we are by assumption working with compact quantum group algebras, $\haar$ restricts to a state on $M_{> 0}$. Whenever we are working in the quantum homogeneous space setting, our choice of state will always be the Haar functional. The following result is immediate, but important, so we highlight it as a lemma. 

\begin{lemma}
For $(\F,h)$ a Hermitian vector with $\F \in $, and $\left(\Om^{(\bullet,\bullet)}, \k\right)$ a covariant K\"ahler structure, the inner product $\la \cdot,\cdot\ra:\Om^\bullet \oby_M \F \to \bC$ is a left $G$ comodule map \wrt the trivial left $G$-comodule structure on $\bC$.
\end{lemma}

\begin{remark}
For any Hermitian vector bundle $(\F,h)$, the inner product space $\Om^\bullet \oby_M \F$ can be completed to a Hilbert space $L^2(\Om^{\bullet} \oby_M \F)$.  In ongoing work, we show that the operators $L_\F,\Lambda_{\F}, H_\F$ extend to bounded operators on $L^2(\Om^{\bullet} \oby_M \F)$, and hence give a bounded representation of $\mathfrak{sl}_2$ on the Hilbert space.  The operators $\adel_{\F}$ and $\adel_\F^\dagger$ are more badly behaved and, in general, will be unbounded operators.

For quantum homogeneous spaces, one can adapt this approach to produce sufficient conditions for these operators to produce spectral triples, as will be discussed in later work.
\end{remark}

\section{Hodge Decomposition and Serre Duality}\label{section:Hodge}

Let $(\Om^{(\bullet, \bullet)}, \odel, \del, \k)$ be a positive definite K\"ahler structure.  Let $\operatorname{\bf f}$ be a state such that the integral $\int = \operatorname{\bf f} \circ \vol_\k$ is closed.

Under the additional assumption of diagonalisability of the natural Dirac operator, we establish noncommutative generalisations of the fundamental results of Hodge theory for K\"ahler manifolds.  As a consequence, a direct \nc generalisation of Serre duality is then established.  Finally, we show that the assumption of diagonalisability always holds in the quantum homogenous space case.

\subsection{Diagonalisable K\"ahler structures and Hodge decomposition}

In this subsection we introduce direct noncommutative generalisations of the twisted Dirac and Laplace operators of a K\"ahler manifold, and under the assumption of diagonalisability  of the Dirac operator prove a noncommutative generalisation of Hodge decomposition. This then allows to establish an isomorphism between harmonic forms and cohomology classes, a crucial result in all of later discussion of cohomology groups.

\begin{definition}
For a (left or right) holomorphic vector bundle $(\F,\adel_\F)$ with a positive definite Hermitian structure, we define its {\em Dirac operator} and {\em Laplacian} as the self-adjoint operators
\bas
D_{\odel_\F} \coloneqq \adel_\F + \adel_\F^\dagger, & & \DEL_{\odel_\F} \coloneqq D_{\F}^2 = \adel_{\F}^\dagger\adel_{\F} + \adel_{\F} \adel_{\F}^\dagger, 
\eas
respectively.  Moreover, we denote $\H^\bullet_{\adel_\F} \coloneqq \ker\left(\DEL_{\F}\right)$, and call elements of $\H^\bullet_{\adel_\F}$ {\em harmonic} elements.
\end{definition}

We now introduce the notion of diagonalisability for a K\"ahler structure in terms of its Dirac operator. (See Remark \ref{remark:HermitianFormDiagonilisable} below for a discussion of the more general Hermitian structure case.)

\begin{definition}\label{definition:Diagonalisable}
We say that a K\"ahler structure on the holomorphic vector bundle with a positive definite Hermitian structure is {\em diagonalisable} if its Dirac operator $D_{\adel_\F}$ is diagonalisable as a $\bC$-linear operator.
\end{definition}

With the definitions of Dirac and Laplace operators, and diagonalisability in hand, we are now ready to prove our noncommutative generalisation of Hodge decomposition.

\begin{proposition}\label{proposition:HodgeDecomposition}
Let $V = (\oplus_i V^i, \exd)$ be a differential graded vector space.  Assume that $V$ has a non-degenerate Hermitian inner product $\langle - , - \rangle$, and that $\exd$ admits an adjoint $\exd^\dagger\colon V \to V$.  Write $D = \exd + \exd^\dagger$.

\begin{enumerate}
\item\label{enumerate:HodgeProposition1} $\Ker D = \Ker D^2 = \Ker \exd \cap \Ker \exd^\dagger,$
\item\label{enumerate:HodgeProposition2} $\im \exd, \im \exd^\dagger,$ and $\Ker D$ are pairwise perpendicular.
\end{enumerate}

If the linear map $D$ is diagonalisable, then
\begin{enumerate}[resume]
\item\label{enumerate:HodgeProposition3} $\Ker \exd = \im \exd \oplus \Ker D^2,$
\item\label{enumerate:HodgeProposition4} $V = \im \exd \oplus \im \exd^\dagger \oplus \ker(D^2),$
\item\label{enumerate:HodgeProposition5} the canonical map $\Ker D \to H^\bullet_{V} = \Ker (\exd) / \im (\exd)$, mapping an element to its image in the homology, is an isomorphism.
\end{enumerate}
\end{proposition}

\begin{proof}
Let $v, w \in V$.  We have
\[
\langle \exd v, \exd^\dagger w \rangle = \langle \exd^2 v , w \rangle = \langle 0,w\rangle=0,
\]
so that $(\im \exd) \perp (\im \exd^\dagger)$.  Consequently, we find
\[\Ker D = \Ker (\exd + \exd^\dagger) = \Ker \exd \cap \Ker \exd^\dagger.\]
Furthermore, as $D = \exd + \exd^\dagger$ is a self-adjoint operator, we have $\Ker D = \Ker D^2$.  This establishes that (\ref{enumerate:HodgeProposition1}) holds.

To show that (\ref{enumerate:HodgeProposition2}) holds, we need only show that $(\Ker D) \perp (\im \exd + \im \exd^\dagger).$  For this, let $u \in \Ker D$ and $v, w \in V$, we have (using (\ref{enumerate:HodgeProposition1})):
\[ \langle u, \exd(v) + \exd^\dagger(w) \rangle = \langle \exd^\dagger(u), v \rangle + \langle \exd(u), w \rangle = 0.\]
This establishes (\ref{enumerate:HodgeProposition2}).

To continue the proof, we start with the following observation.  Let $b= \sum_i b^i \in V$ be any eigenvector for $D$ and let $\lambda \in \mathbb{C}$ be the eigenvalue.  Let $j$ be the largest number such that $b^j \not= 0$.  We have $\exd(b^j) = 0$, so $b^j \in \Ker \exd$.

We claim that $\lambda b^j + \exd^\dagger(b^j)$ is an eigenvector of $D$ with eigenvalue $\lambda$.  Firstly, as $D^2$ is a degree zero map and $D^2(b) = \lambda^2 b$, we find that $D^2(b^j) = \lambda^2 b^j$.  We have:
\begin{align*}
D\left(\lambda b^j + \exd^\dagger(b^j)\right) &= \left(\exd + \exd^\dagger \right)\left(\lambda b^j + \exd^\dagger(b^j)\right) \\
&= \lambda\exd^\dagger(b^j) + \exd \exd^\dagger (b^j) \\
&= \lambda\exd^\dagger(b^j) + \exd \exd^\dagger (b^j) + \exd^\dagger \exd (b^j)\\
&= \lambda\exd^\dagger(b^j) + D^2 (b^j)\\
&= \lambda\exd^\dagger(b^j) + \lambda^2 b^j\\
&= \lambda \left( \lambda b^j + \exd^\dagger(b^j)\right).
\end{align*} 
Hence, $\lambda b^j + \exd^\dagger(b^j)$ is an eigenvector of $D$ with eigenvalue $\lambda$, and, as such, also an eigenvector of $D^2$ with eigenvalue $\lambda^2$.  Since $D^2$ is a degree 0 map, both $b^j$ and $\exd^\dagger(b^j)$ are eigenvectors of $D^2$ with eigenvalue $\lambda^2$.

We have the following possibilities.  The first possibility is that $\lambda = 0$.  In this case: $b \in \ker \exd$.  Note that $b = b^{j}$ (thus, $b$ is concentrated in degree $j$) also implies that $\lambda = 0$.  The second possibility is where $\lambda \not= 0$.  In this case, $b' \coloneqq b - (b^j + \frac{1}{\lambda}\exd^\dagger(b^j))$ is an eigenvector of $D$ (or possibly zero) with eigenvalue $\lambda$ and $b'$ is supported in strictly fewer degrees than $b$.  Iterating this reduction, we see that $b$ can be written as a sum of elements of $\Ker \exd$ (the elements $b^j$) and elements of $\im \exd^\dagger$ (the elements $\exd^\dagger(b^j)$).

This shows that, if $b \in V$ is an eigenvector for $D$, we have: $b \in \Ker \exd + \im \exd^\dagger$.  As $D\colon V \to V$ is diagonalisable, there is a basis of eigenvectors and hence $V = \Ker (\exd) + \im (\exd^\dagger)$.

For any $v \in \Ker (\exd)$ and $w \in V$, we have
\[\langle v, \exd^\dagger(w) \rangle = \langle \exd(v), w \rangle = 0.\]
This means that $V = \Ker (\exd) \oplus \im (\exd^\dagger)$ is an orthogonal decomposition.  As $\im (\exd^\dagger) \subseteq \Ker (\exd^\dagger)$, we have $V = \Ker (\exd) + \Ker (\exd^\dagger)$.  Combining this with (\ref{enumerate:HodgeProposition1}) yields 
\[ \Ker (\exd^\dagger) = \left( \ker (\exd) \cap \ker (\exd^\dagger)\right )\oplus \im (\exd^\dagger) = \Ker (D^2) \oplus \im (\exd^\dagger),\] establishing (\ref{enumerate:HodgeProposition3}).

Statement (\ref{enumerate:HodgeProposition4}) follows from $V = \Ker (\exd) \oplus \im (\exd^\dagger)$ and (\ref{enumerate:HodgeProposition3}).  Finally, (\ref{enumerate:HodgeProposition5}) follows from $H^\bullet_{V} = \Ker (\exd) / \im (\exd)$ together with (\ref{enumerate:HodgeProposition4}).
\end{proof}

\begin{theorem}[Hodge Decomposition]\label{theorem:HodgeDecomposition}
Let $(\F, \odel_\F)$ be a (left or right) Hermitian holomorphic vector bundle and write $\Om^\bullet_\F$ for $\Om^{\bullet} \otimes_M \F$ (if $\F$ is left holomorphic) or $\Om^{\bullet} \otimes_M \F$ (if $\F$ is right holomorphic).  If the Hermitian structure on $\F$ is positive definite and the K\"ahler structure on $\Om^\bullet_\F$ is diagonalisable, then there is an orthogonal decomposition of $A$-comodules:
\bas
\Om^{\bullet}_\F = {\cal H}_{\ol{\del}_{\F}}^\bullet \oplus \ol{\del}_\F\left(\Om^\bullet_{\F}\right) \oplus  \ol{\del}_\F^\dagger\left(\Om^\bullet_{\F}\right).
\eas
Furthermore, the projection 
$
\H^{(a,b)}_{\adel_\F} \to H^{(a,b)}_{\adel_\F}\colon \a \mto [\a]
$
is an isomorphism.
\end{theorem}

\begin{proof}
The decomposition as vector spaces follows directly from Proposition \ref{proposition:HodgeDecomposition}.  As the embeddings ${\cal H}_{\ol{\del}_{\F}}^\bullet, \ol{\del}_\F\left(\Om^\bullet_{\F}\right),  \ol{\del}_\F^\dagger\left(\Om^\bullet_{\F}\right) \to \Om^{\bullet}_\F$ are all $A$-comodule maps, the direct sum decomposition also respects the $A$-comodule structure.
\end{proof}

\begin{remark}\label{remark:HermitianFormDiagonilisable}
\item Removing the assumption of closure of $\k$ from the definition of a K\"ahler form one arrives at the definition of an Hermitian form;  the definition of diagonalisability (Definition \ref{definition:Diagonalisable}) carries over directly to the Hermitian case.  As closure of the K\"ahler form is not required in the proof of Theorem \ref{theorem:HodgeDecomposition} (it is not required in Proposition \ref{proposition:HodgeDecomposition}), the Hodge decomposition holds in this more general setting.
\end{remark}


\subsection{Serre duality}

We now turn our attention to Serre duality.  We start by using Hodge decomposition (and its implied equivalence between harmonic forms and cohomology classes) to show that the map $\bar{\ast}_\F\colon \Om^\bullet \oby_M \F \to \lvee\!\F \oby_M \Om^\bullet$ induces an isomorphism on homology.

\begin{proposition}\label{proposition:hodgeH}
Under the conditions of Theorem \ref{theorem:HodgeDecomposition}, the map $\bar{\ast}_\F\colon \Om^\bullet \oby_M \F \to {}^\vee\!\F \oby_M \Om^\bullet$ maps harmonic forms to harmonic forms, and hence induces an isomorphism of cohomology groups.
\end{proposition}
\begin{proof}
We prove the fact that $\bar{\ast}_\F$ maps harmonics to harmonics by showing that  $\bar{\ast}_\F$ intertwines the left and right Laplacians. For $\a \in \Om^k \oby_M \F$.  We have, using Lemma \ref{Hodgeprop} and Proposition \ref{proposition:adelIsAdjointable}.  
\begin{align*}
 \DEL_{\adel_{{}^\vee\!\F}} \circ \bar{\ast}_{\F}(\a)
&=  \big(\adel^\dagger_{\lvee\!\F} \circ \adel_{\lvee\!\F} + \adel_{\lvee\!\F}  \circ \adel^\dagger_{\lvee\!\F} \big) \circ \bar{\ast}_{\F}(\a) \\
&= - \Big(\bar{\ast}_{\F} \circ \adel_{\F} \circ \bar{\ast}_{\lvee\!\F} \circ \adel_{\F} + \adel_{\lvee\!\F}  \circ \bar{\ast}_{\F} \circ \adel_{\F} \circ \bar{\ast}_{{}^\vee\!\F} \Big) \circ \bar{\ast}_{\F}(\a)\\
&= - \bar{\ast}_{\F} \circ \adel_{\F} \circ \bar{\ast}_{\lvee\!\F} \circ \adel_{\F} \circ \bar{\ast}_{\F}(\a) - \adel_{\lvee\!\F}  \circ \bar{\ast}_{\F} \circ \adel_{\F} \circ \bar{\ast}_{{}^\vee\!\F} \circ \bar{\ast}_{\F}(\a)\\
&= - \bar{\ast}_{\F} \circ \adel_{\F} \circ \bar{\ast}_{\lvee\!\F} \circ \adel_{\F} \circ \bar{\ast}_{\F}(\a) -(-1)^{k} \adel_{\lvee\!\F}  \circ \bar{\ast}_{\F} \circ \adel_{\F}(\a)\\
&= - \bar{\ast}_{\F} \circ \adel_{\F} \circ \bar{\ast}_{\lvee\!\F} \circ \adel_{\F} \circ \bar{\ast}_{\F}(\a) - \bar{\ast}_{\F} \circ \bar{\ast}_{\lvee\!\F} \circ \adel_{\lvee\!\F}  \circ \bar{\ast}_{\F} \circ \adel_{\F}(\a)\\
&= \bar{\ast}_{\F} \circ \adel_{\F} \circ \adel^\dagger_{\F}(\a) + \bar{\ast}_{\F} \circ \adel^\dagger_{\F} \circ \adel_{\F}(\a)\\
&= \bar{\ast}_{\F}\,  \circ \DEL_{\adel_{\F}},
\end{align*}
as required.
\end{proof}

\begin{corollary}\label{corollary:HodgeMapsCohomology}
Under the conditions of Theorem \ref{theorem:HodgeDecomposition}, the map $\bar{\ast}_\F\colon \Om^\bullet \oby_M \F \to {}^\vee\!\F \oby_M \Om^\bullet$ induces a map on the cohomology $\bar{\ast}_\F\colon H^{(a,b)}_{\adel_\F} \to \bar{\ast}_\F\colon H^{(n-a,n-b)}_{\adel_{\lvee\!\F}}$.
\end{corollary}

With this result in hand we are now in a position to prove a direct noncommutative generalisation of classical Serre duality pairing of cohomology groups. Moreover, under the assumption of finite dimensionality, we get a $\bC$-linear isomorphism of the paired spaces.

\begin{theorem}[Serre duality]
There is a non-degenerate sesquilinear pairing given by
\begin{align*}
H^{(a,b)}_{\odel_\F} \otimes H^{(n-a,n-b)}_{\odel{\lvee\!\F}} \to \mathbb{C} && \big([\alpha],[\beta]\big) \mapsto \int \alpha \wedgeev \beta.
\end{align*}
\end{theorem}

\begin{proof}
We first show that the proposed pairing is well-defined, which is to say, that the definition is independent of the choice of representatives $\alpha$ and $\beta$.  For this, consider 
\bas
\adel_\F (\zeta) \in \Om^{(a,b)} \otimes_M \F, & & \text{ and } & & \adel_\F(\eta) \in {}^\vee \!\F \otimes_M \Om^{(n-a,n-b)}.
\eas
Using Remark \ref{remark:DualsForDGModules} and Lemma \ref{lemma:Stokes}, we find:
\begin{align*}
&\int \big(\alpha + \adel_\F(\zeta)\big) \wedgeev \big(\beta + \adel_{\F}(\eta)\big)\\
&= \int \alpha \wedgeev \beta + \int \adel_\F(\zeta) \wedgeev \beta + \int \alpha \wedgeev \adel_\F(\eta)  + \int \adel_\F(\zeta) \wedgeev \adel_\F(\eta) \\
&= \int \alpha \wedgeev \beta + \int \adel \big(\zeta \wedgeev \beta\big) +(-1)^{|\alpha|} \int \adel\big(\alpha \wedgeev \eta\big)  + \int \adel (\zeta \wedgeev \adel \eta) \\
&= \int \alpha \wedgeev \beta,
\end{align*}
as required.  

Finally, we use Corollary \ref{corollary:HodgeMapsCohomology} above to show that the pairing is non-degenerate. For every nonzero $[\alpha] \in H^{(a,b)}_\F$, where $\a$ is a harmonic choice of representative, we have
\[
\big([\alpha], \bar{\ast}_\F[\alpha]\big) = \int \alpha \wedgeev \bar{\ast}_\F(\alpha) = \langle \alpha, \alpha \rangle_\F \neq 0,
\]
by non-degeneracy of the inner product $\la \cdot,\cdot\ra_\F$.
\end{proof}

\begin{corollary}
If $\Om^\bullet_{\F}$ has finite-dimensional cohomology groups, then 
\bas
H^{(a,b)}_{\adel_\F} \simeq \Big(H^{(n-a,n-b)}_{\adel_{\lvee\!\F}}\Big)^*,
\eas
where $\Big(H^{(n-a,n-b)}_{\adel_{\lvee\!\F}}\Big)^*$ denotes the $\bC$-linear dual of $H^{(n-a,n-b)}_{\adel_\F}$.
\end{corollary}

\subsection{The quantum homogeneous space case}

In this subsection we specialise to the the case of quantum homogeneous spaces and show that any  covariant K\"ahler structure is automatically diagonalisable. We begin by recalling  that, for any $\F \in \sgmm$, a decomposition into finite-dimensional  left $G$-comodules  is given by
\bas
\F \simeq G \coby \Phi(\F) \simeq \Big(\bigoplus_{V\in \wh{G}} \C(V))\Big) \coby \Phi(\F) = \bigoplus_{V\in \wh{G}} \big(\C(V) \coby  \Phi(\F)\big) \eqqcolon \bigoplus_{V\in \wh{G}} \F_V,
\eas 
where summation is over all equivalence classes of left $G$-comodules, and $\C(V)$ is the coalgebra of matrix coefficients of $V$.  We call this the {\em Peter--Weyl decomposition} of $\F$. 
As is well-known \cite[Proposition 11.15]{KSLeabh}, the Peter--Weyl decomposition of $G$ itself, $G \cong \bigoplus_{V \in \wh{G}} C(V)$, is orthogonal \wrt the inner product induced by the haar functional, which is to say, with respect to the inner product
\bas
\la \cdot,\cdot\ra\colon G \oby G \to \bC, & & f \oby g^* \to \haar(fg^*).  
\eas
The following generalisation of \cite[Lemma 5.7]{MMF3} shows that this extends to the de Rham complex of any Hermitian vector bundle.

\begin{proposition}
For  $(\F,h)$ an Hermitian vector bundle, its Peter--Weyl decomposition is orthogonal with respect to the inner product $\la\cdot,\cdot \ra_\F$.
\end{proposition}
\begin{proof}
Let $\a$ and $\b$ be elements of  $\Om^\bullet \oby_M \F$ which are homogeneous with respect to the Peter--Weyl decomposition of $\Om^\bullet \oby_M \F$, and \st each  element is contained in a distinct  summand of the decomposition. Noting that $\ol{g} \coloneqq g\circ (\id \oby C_h\inv)$  is a morphism in $\sgmm$, we see that 
\bas
g_\F(\a,\b) &= \ol{g}_{\F}\big(\a, C_h(\b)\big)\\
                   &= \haar \circ (\id \oby \Phi(\ol{g})) \circ \unit \big(\a \oby_M C_h(\b)\big) \\
                   &= \haar(\a_{(-1)}, \b_{(-1)}^*)  \Phi\big(\ol{g})[\a\0 \oby_M C_h(\b\0)\big]\\
                   &=  0,
\eas
where we have used the easily verifiable identity
\bas
\Delta_G \circ C_h = (* \oby C_h) \circ \Delta_G\colon \Om^\bullet \oby_M \F \to G \oby \Om^\bullet \oby_M \F.
\eas 
\end{proof}

Orthogonality of the decomposition now allows us to prove diagonalisability for covariant K\"ahler structures.

\begin{proposition}\label{cor:EqAdj}
For  $(\F,h)$ an Hermitian vector bundle, every left $G$-comodule map $f:\Om^\bullet \oby_M \F \to \Om^\bullet \oby_M \F$ is adjointable \wrt $\la\cdot,\cdot\ra_\F$.  Moreover, if $f$ is self-adjoint, then it is diagonalisable.
\end{proposition}

\begin{proof}
Since $f$ is a left $G$-comodule map  $f\big((\Om^\bullet \oby_M \F)_V\big) \sseq (\Om^\bullet \oby_M \F)_V$, for all $V \in \wh{G}$. Adjointability of $f$ now follows from finite-dimensionality of $\big(\Om^\bullet \oby_M \F\big)$ and the fact that the Peter--Weyl decomposition is orthogonal \wrt  $\la\cdot,\cdot\ra_\F$. The proof that  $f$ is diagonalisable whenever it is self-adjoint is analogous.
\end{proof}

As an immediate consequence we get the following corollary.

\begin{corollary}
Every quantum homogeneous K\"ahler space is diagonalisable.
\end{corollary}

\newcommand{\mfh}{\mathfrak{h}}

\section{Chern connections and the Nakano identities}\label{section:Nakano}

Let $(\Om^{(\bullet, \bullet)}, \odel, \del, \k)$ be a positive definite K\"ahler structure.  Let $\operatorname{\bf f}$ be a state such that the integral $\int = \operatorname{\bf f} \circ \vol_\k$ is closed.

In this section, we prove the Nakano identities (see Theorem \ref{theorem:Nakano} below), which are a extension of the K\"ahler identities from the bicomplex $\Om^{(\bullet, \bullet)}$ to the bicomplex $\Om^{(\bullet, \bullet)} \otimes_M \F$.  The Nakano identities are key to our noncommutative generalisation of the Kodaira vanishing theorem, and a result of importance in its own right.

To avoid local arguments, we extend the standard global proof of the K\"ahler identities (see \cite{HUY, Weil58}) using the framework of Lefschetz triples.  The map $\del_\F$ that occurs in Theorem \ref{theorem:Nakano} is the unique anti-holomorphic connection such that $\nabla = \del_\F + \odel_\F$ is the Chern connection.  Our first step to proving Theorem \ref{theorem:Nakano} is to establish Proposition \ref{proposition:ChernViaConjugates} below.  The proof does not require all of the standing hypothesis of this section, only the ones that are stated.

\subsection{Chern connections}\label{subsection:Chern} Let $(\Om^{(\bullet, \bullet)}, \odel, \del)$ be an $A$-covariant $*$-differential structure with a complex structure.  Let $(\F, \odel_\F, h)$ be an $A$-covariant Hermitian holomorphic vector bundle.  We know from Proposition \ref{prop:BM1} that there is a unique extension to a Hermitian connection.  In Proposition \ref{proposition:ChernViaConjugates}, we will describe this extension $\odel_\F$ via the map $C_h$ from Definition \ref{definition:TwistedHodge}.  We start with two lemmas.

\begin{lemma}\label{lemma:SwitchInfh}
For all homogeneous $\alpha, \beta \in \Omega^\bullet \otimes_M \F$, we have
\[\mfh(\alpha, \beta) = (-1)^{|\alpha||\beta|} \mfh(\beta, \alpha)^*.\]
\end{lemma}
\begin{proof}
We write $\alpha = \omega \otimes f$ and $\beta = \nu \otimes g$.  The statement follows from:
\begin{align*}
\mfh(\omega \otimes f, \nu \otimes g) &= \omega h(\ol{g})(f) \wedge \nu^* \\
&= (-1)^{|\omega||\nu|}\left(\nu h(\ol{g})(f)^* \wedge \omega^*\right)^* \\
&= (-1)^{|\omega||\nu|}\left(\nu h(\ol{f})(g) \wedge \omega^*\right)^* \\
&= (-1)^{|\omega||\nu|} \mfh(\nu \otimes g, \omega \otimes f)^*.
\end{align*}
\end{proof}

\begin{lemma}\label{lemma:ChernHermitian} For all $f \in \F, \phi \in \lvee \F$, and $\omega \in \Om^\bullet$, we have
\begin{enumerate}
\item $\mfh\left(f, C^{-1}_h(\phi \otimes \omega) \right) = \phi(f)\omega,$ and
\item $\mfh\left(C^{-1}_h(\phi \otimes \omega), f \right) = (\phi(f)\omega)^*.$
\end{enumerate}
\end{lemma}

\begin{proof}
The first statement follows from
\[\mfh\left(f, C^{-1}_h(\phi \otimes \omega) \right) = \mfh\left(f, \omega^* \otimes \ol{h^{-1}(\phi)}\right) = \phi(f)\omega. \]
The second statement follows from the first together with Lemma \ref{lemma:SwitchInfh}.
\end{proof}

\begin{proposition}\label{proposition:ChernViaConjugates}
Let $(\Om^{(\bullet, \bullet)}, \odel, \del)$ be an $A$-covariant $*$-differential structure with a complex structure.  Let $(\F, \odel_\F, h)$ be an $A$-covariant Hermitian holomorphic vector bundle, and consider the degree one map:
\[\del_\F \coloneqq C^{-1}_h \circ \odel_{\lvee\!\F} \circ C_h\colon \Om^{(\bullet, 0)} \otimes_M \F \to \Om^{(\bullet, 0)} \otimes_M \F.\]
The map $\del_\F$ is an anti-holomorpic structure for $\F$ and the map $\nabla \coloneqq \odel_\F + \del_\F$ is the Chern connection for the triple $(\F, \odel_\F, h)$.
\end{proposition}

\begin{proof}
We start by verifying that $\del_\F$ is a holomorphic structure for $\F$.  Firstly, note that since 
\[\del_\F^2 = \left( C^{-1}_h \circ \odel_{\lvee\!\F} \circ C_h \right) \circ \left( C^{-1}_h \circ \odel_{\lvee\!\F} \circ C_h \right)
= C^{-1}_h \circ \odel^2_{\lvee\!\F} \circ C_h = 0,\]
the degree one map $\del_\F$ gives $\Om^{(\bullet,0)} \otimes_M \F$ the structure of a complex.  Secondly, for all homogeneous $\omega \in \Om^{(\bullet,0)}$ and $\alpha \in \Om^{(\bullet, 0)} \otimes_M \F$, we have (using Lemma \ref{Hodgeprop})
\begin{align*}
\del_\F(\omega \wedge \alpha) &= C^{-1}_h \circ \odel_{\lvee\!\F} \circ C_h (\omega \wedge \alpha)\\
&= (-1)^{|\alpha||\omega|} C^{-1}_h \circ \odel_{\lvee\!\F} \circ \left(C_h (\alpha) \wedge \omega^*\right)\\
&= (-1)^{|\alpha||\omega|} C^{-1}_h \left( \odel_{\lvee\!\F} \circ C_h (\alpha) \wedge \omega^* \right) + (-1)^{(|\omega|+1)|\alpha|} C^{-1}_h \left(C_h (\alpha) \wedge \odel(\omega^*)\right) \\
&= (-1)^{|\omega|} \omega \wedge C^{-1}_h \left( \odel_{\lvee\!\F} \circ C_h (\alpha) \right)
+ \odel(\omega^*)^* \wedge C^{-1}_h \circ \left(C_h (\alpha))\right) \\
&= (-1)^{|\omega|} \omega \wedge \del_\F(\alpha) + \del(\omega) \wedge \alpha,
\end{align*}
so that $(\Om^{(\bullet,0)} \otimes_M \F, \del_\F)$ is a left dg module over $(\Om^{(\bullet),0}, \del)$.  Hence, $\del_\F$ gives $\F$ the structure of an anti-holomorphic vector bundle.

We now verify that $\nabla = \odel_\F + \del_\F$ is the Chern connection.  For this, we need to show that $h$ is a Hermitian connection (see Definition \ref{definition:HermitianConnection}).  Let $f,g \in \F$.  We have:
\begin{align*}
\exd \mathfrak{h}(f, g) &= \mathfrak{h}_\F(\nabla(f), g) + \mathfrak{h}_\F(f, {\nabla}{g})) \\
&= \mathfrak{h}_\F(\del_\F(f), g)+ \mathfrak{h}_\F(\odel_\F(f), g) + \mathfrak{h}_\F(f, {\del_\F}(g))) + \mathfrak{h}_\F(f, {\odel_\F}(g))).
\end{align*}
We rewrite each of the terms on the right-hand side using Lemmas \ref{lemma:SwitchInfh} and \ref{lemma:ChernHermitian}:
\begin{align*}
\mathfrak{h}_\F(\del_\F(f), g) &= \fh_\F(C^{-1}_h \circ \odel_{\lvee\!\F} \circ C_h (f),g) \\
&= (\ev_g \otimes \id) \left(\odel_{\lvee\!\F} \circ C_h (f)\right)^* \\
&= (\ev_g \otimes \id) \left(\odel_{\lvee\!\F} h(\ol{f})\right)^* \\
&= \odel(h(\ol{f})(g))^* - \left((\id \otimes h(\ol{f}))(\odel_\F(g)) \right)^*.\\
\mathfrak{h}_\F(\odel_\F(f), g) &= (\id \otimes h)(\ol{g})(\odel_\F(f)) \\
\mathfrak{h}_\F(f, {\del_\F}{g})) &= \fh_\F(f, C^{-1}_h \circ \odel_{\lvee\!\F} \circ C_h (g)) \\
&= (\ev_f \otimes \id) \left(\odel_{\lvee\!\F} \circ C_h (g)\right)^* \\
&= (\ev_f \otimes \id) \left(\odel_{\lvee\!\F} \circ h (\ol{g})\right)^* \\
&= \odel(h(\ol{f})(g)) - (\id \otimes h(\ol{g})) \odel_\F(f) \\
\mathfrak{h}_\F(f, {\odel_\F}{g})) &= \fh_\F(\odel_\F(g), (f))^* \\
&= (\id \otimes h(\ol{f})) \odel_\F(g).
\end{align*}
The required equality now follows easily (using $\exd = \del + \odel$).  This completes the proof.
\end{proof}

\subsection{The Nakano Identities}

We now come to the proof of the Nakano identities.  Classically, this is often done using the a local argument to extend the identities from the untwisted to the twisted case (\cite[\textsection 5.2 ]{HUY}).  In the noncommutative setting, however, such local arguments are not available.  We circumvent this problem using the framework of Lefschetz triples developed in \textsection\ref{section:LefschetzPairs}.

\begin{proposition}\label{Proposition:NIDSI}
For any Hermitian vector bundle $(\F,h)$ with a left holomorphic structure $\odel_\F\colon \F \to \Om^{(0,1)} \otimes_M \F$.  Let $\del\colon \F \to \Om^{(0,1)} \otimes_M \F$ be the anti-holomorphic structure from Proposition \ref{proposition:ChernViaConjugates}.  We have the following relations
\begin{align*}
[L_\F,\del_\F] = 0, &&  [L_\F,\adel_\F] = 0 && [\Lam_\F, \del^\dagger_\F] = 0, &&  [\Lam_\F, \adel_\F^\dagger] = 0. 
\end{align*}
\end{proposition}

\begin{proof}
As $(\Om^{\bullet}, L = \kappa \wedge -, \del)$ is a Lefschetz triple, we have $[\del,L] = 0$.  This yields: 
\bas
[\del_{\F},L_\F](\w \otimes f) &= \big[ \del \oby \id + \id \oby \del_\F, L \oby \id\big](\w \otimes f)\\
&= \big[\del \oby \id, L \oby \id\big](\w \otimes f) + \big[\id \oby \del_\F,L \oby \id\big](\w \otimes f)\\
&= \Big(\big[\del,L](\w)\Big) \oby f\\
&= 0.
\eas
The second identity can be proved in a similar way, using that $[\odel, L_\F] = 0$.  Finally, the third and fourth identities now follow as the adjoints of the first and second identities respectively.
\end{proof}

As an immediate consequence we get the following corollary which gives examples of Lefschetz triples, as promised in \textsection\ref{section:LefschetzPairs}.

\begin{corollary}\label{Corollary:UsefulLefschetzTriples}
For any $\F \in \psAM\lproj$ with a left holomorphic and anti-holomorphic structure, there are Lefschetz triples given by
\bas
(\Om^{\bullet} \oby_M \F, L_\F, \del_\F), & & (\Om^{\bullet} \oby_M \F, L_\F, \adel_\F), & & (\Om^{\bullet} \oby_M \F, L_\F, \exd_\F),
\eas
where $\exd_\F \coloneqq \adel_\F + \del_\F.$
\end{corollary}

\begin{proof}
This follows from Propositions \ref{prop:lefadjs} and \ref{Proposition:NIDSI}.
\end{proof}

We now use the fact that we are dealing with Lefschetz triples to extend the K\"ahler identities to the twisted case.

\begin{theorem}[Nakano identities]\label{theorem:Nakano}
Let $(\Om^{(\bullet, \bullet)}, \odel, \del, \k)$ be a positive definite $A$-covariant K\"ahler structure.  Let $\operatorname{\bf f}$ be an $A$-covariant state such that the integral $\int = \operatorname{\bf f} \circ \vol_\k$ is closed.  For an $A$-covariant Hermitian holomorphic vector bundle $(\F, \odel_\F, h)$ with Chern connection $\nabla_\F = \odel_\F + \del_\F$, we have
\begin{align} \label{KIDS2}
[L_\F,\del_\F^{\dagger}] = i \adel_\F, & & [L_\F,\adel_\F^{\dagger}] = - i \del_\F, & &  [\Lambda_{\F},\del_{\F}] = i \adel_{\F}^\dagger, & & {[\Lambda_{\F},\adel_\F] = - i\del^\dagger_{\F}}
\end{align}
\end{theorem}

\begin{proof}
 We first prove the third identity, and then derive the remaining three as consequences.  Let $\alpha \in P^k$ be a primitive element for the Lefschetz pair $(\Om^{\bullet} \oby_M \F, L_\F)$.  Proposition \ref{prop:ltids} implies that 
\begin{equation} \label{eqn:NIDSLHS}
[\Lambda_{\F},\del_{\F}](L_{\F}^j(\a)) = (n-j-k+1)L^j(\a_1) -jL^{j-1}(\a_0),
\end{equation}
for some $\alpha_1 \in P^{k-1}$ and $\alpha_0 \in P^{k+1}$.
Moving to the right-hand side of the proposed identity, and assuming without loss of generality that $\a = \omega \otimes f \in P^{(a,b)} \oby_M \F \subseteq P^k \oby_M \F$, we see that (using the Hodge map from \eqref{equation:NewHodge})
\bas
   i\adel_{\F}^\dagger \circ L^j_{\F}\big(\a\big) 
&=  - i \, \bar{\ast}_\F \circ  \adel_{\,\lvee\!\F} \circ \, \bar{\ast}_\F\big(L^j(\w) \oby_M f \big) \\
&=  - i \, \bar{\ast}_\F \circ  \adel_{\,^\vee\!\F} \circ \ast_h \Big(\big(\ast_H \circ \, L^j(\w)\big) \oby_M  f \Big) \\
&= (-1)^{\frac{k(k+1)}{2}} i^{a-b-1} \frac{j!}{(n-j-k)!} \ast_\F \circ \, C^{-1}_h \circ \,\adel_{\,^\vee\!\F} \circ C_h \Big(L^{n-j-k}(\w) \oby_M f)\Big)\\
&=  (-1)^{\frac{k(k+1)}{2}} i^{a-b-1} \frac{j!}{(n-j-k)!} \ast_\F \circ \, \del_{\F} \Big(L^{n-j-k}_{\F}(\a)\Big).
\eas
It now follows from Propositions \ref{prop:ltids} and \ref{Proposition:NIDSI}  that 
\bas
   i\adel_{\F}^\dagger \circ L^j_{\F}(\alpha) &= (-1)^{\frac{k(k+1)}{2}} i^{a-b-1} \frac{j!}{(n-j-k)!} \ast_\F \circ \, L^{n-j-k}_{\F} \circ \del_{\F}(\a)\\
& = (-1)^{\frac{k(k+1)}{2}} i^{a-b-1} \frac{j!}{(n-j-k)!} \ast_\F \Big( L^{n-j-k}_{\F}(\a_0) + L^{n-j-k+1}_{\F}(\a_1)\Big).
\eas
Recalling that $\a_0 \in P^{(a+1,b)} \oby_M \F$ and $\a_1 \in P^{(a,b-1)} \oby_M \F$, we see from the definition of the Hodge map (and Proposition \ref{prop:lefadjs}) that  
\bas
\ast_\F\big(L_\F^{n-j-k}(\a_0)\big) = & \, (-1)^{\frac{(k+1)(k+2)}{2}}i^{a-b+1}\frac{(n-j-k)!}{(j-1)!} L_\F^{j-1}(\a_0),\\
\ast_\F\big(L_\F^{n-j-k+1}(\a_1)\big) =  & \, (-1)^{\frac{k(k-1)}{2}}i^{a-b+1}\frac{(n-j-k+1)!}{j!} L_\F^{j}(\a_1).
\eas
This implies that
\bas
   i\adel_{\F}^\dagger \circ L^j_{\F}\big(\alpha \big) =   & - j L_{\F}^{j-1}(\a_0) + (n-j-k+1) L_{\F}^j(\a_1). 
\eas
Comparing this expression with \eqref{eqn:NIDSLHS} confirms the third identity.  The proof of the fourth identity is similar.

The first and the second equality are adjoint to the third and the fourth, respectively; this uses $\Lam_{\F} = L_{\F}^\dagger$ (see Proposition \ref{proposition:LefschetzAdjoint}).
\end{proof}

As a first application of the Nakano identities we establish anti-commutator relations between $\del_{\F}$ and $\adel_{\F}$, and  $\del^\dagger_{\F}$ and  $\ol{\del}_{\F}$. 

\begin{corollary} \label{LaplacianEq}
It holds that
\bet
\item $\del_{\F}\ol{\del}^\dagger_{\F} + \ol{\del}^\dagger_{\F} \del_{\F} = 0$,
\item $\del^\dagger_{\F} \ol{\del}_{\F} + \ol{\del}_{\F}\del^\dagger_{\F} = 0$,
\eet
\end{corollary}
\begin{proof}
\bet
\item The first identity follows from
\bas
 i\big(\del_{\F} \circ  \adel^\dagger_{\F} + \adel^\dagger_{\F} \circ  \del\big)_{\F}  = & \, \del_{\F} \circ [\Lambda_{\F},\del_{\F}] + [\Lambda_{\F},\del_{\F}] \circ \del_{\F} \\
&= \del_{\F} \circ \Lambda_{\F} \circ  \del_{\F} - \del^2_{\F} \circ  \Lambda_{\F} + \Lambda_{\F} \circ  \del^2_{\F} - \del_{\F} \circ  \Lambda_{\F} \circ  \del_{\F} \\
&= 0.
\eas

\item The second identity is the operator adjoint of the first identity. \qedhere
\eet
\end{proof}

We finish this section with a second application of the Nakano identities. In particular, we generalise the equality of the untwisted Laplacians $\DEL_{\del} = \DEL_{\adel}$ to the twisted case \cite[Corollary 7.6]{MMF3}. Note that this generalised identity depends on the curvature of the Chern connection and only gives an equality when the Chern connection is flat.

\begin{corollary}[Akizuki--Nakano identity]
For $\nabla = \del_\F + \odel_\F$ the Chern connection associated to $\adel_\F$ and $h$, we have the identity
\bas
\DEL_{\del_{\F}} = \DEL_{\adel_{\F}} + [\Lambda_\F, i \nabla^2].
\eas
\end{corollary}
\begin{proof}
It follows from the Nakano identities that
\begin{align*}
-i \DEL_{\del_\F} = &  -i  \big(\del_\F \circ \del^\dagger_\F + \del^\dagger_\F \circ \del_\F\big) \\ 
&= \del_\F \circ [\Lambda_\F,\ol{\del}_\F] + [\Lambda_\F,\ol{\del}_\F] \circ \del_\F \\
&= \del_\F \circ \Lambda_\F \circ \ol{\del}_\F - \del_\F \circ \ol{\del}_\F  \circ \Lambda_\F + \Lambda_\F \circ \ol{\del}_\F \circ \del_\F - \ol{\del}_\F \circ \Lambda_\F \circ \del_\F.
\end{align*}
 Proposition \ref{prop:BM2} implies that $\nabla^2 = \del_\F \circ \adel_\F + \adel_\F \circ \del_\F$, and so,
\begin{align*}
 \begin{split}
 -i \DEL_{\del_\F} &= \del_\F \circ \Lambda_\F \circ \ol{\del}_\F - \Lambda_\F \circ \del_\F \circ \adel_\F + \ol{\del}_\F \circ \del_\F  \circ \Lambda_\F - \ol{\del}_\F \circ \Lambda_\F \circ \del_\F \\
 & \qquad + \Lambda_\F \circ \nabla^2  - \nabla^2 \circ \Lambda_\F
\end{split} \\
&= - [\Lambda_\F,\del_\F] \circ \ol{\del}_\F - \ol{\del}_\F \circ [\Lambda_\F,\del_\F] + [\Lambda_\F,\nabla^2] .
\end{align*}
Another application of the Nakano identities now implies that
\begin{align*}
 -i \DEL_{\del_\F}  =  & -i\ol{\del}^\dagger_\F \circ \ol{\del}_\F - i\ol{\del}_\F \circ \ol{\del}^\dagger_\F  + [\Lambda_\F,\nabla^2] \\
 &= -i\big(\ol{\del}^\dagger_\F \circ \ol{\del}_\F + \ol{\del}_\F \circ \ol{\del}^\dagger_\F\big)  + [\Lambda_\F,\nabla^2] \\
 &= -i\DEL_{\ol{\del}_\F} +  [\Lambda_\F,\nabla^2].
\end{align*} 
The required identity follows immediately.
\end{proof}

\section{The Kodaira vanishing theorem and noncommutative Fano structures}\label{section:Kodaira}

Let $(\Om^{(\bullet, \bullet)}, \odel, \del, \k)$ be a positive definite K\"ahler structure.  Let $\operatorname{\bf f}$ be a state such that the integral $\int = \operatorname{\bf f} \circ \vol_\k$ is closed.

In this section we give the definition a positive vector bundle.  We proceed by proving a generalisation of the Kodaira vanishing theorem for such vector bundles. We finish by introducing the notion of a Fano structure for a quantum homogeneous space, and prove an additional vanishing result for such structures.

\subsection{The Kodaira vanishing theorem}

We begin with a positivity result for a general holomorphic Hermitian vector bundle, which is to say, we make no assumption yet that the vector bundles is a line bundle.

\begin{lemma} \label{lem:posforf}
Let $(\F, \odel_\F)$ be an $A$-covariant left holomorphic Hermitian vector bundle, and denote by $\nabla = \del_\F + \odel_\F$ its Chern connection (see Proposition \ref{prop:BM1}). If $D_{\odel_\F} = \odel_\F + \odel_\F^\dagger$ is diagonalisable, then for any $\a \in \H^{(a,b)}_{\adel_\F}$, we have that 
\bas
 \big\langle i\,\nabla^2 \circ \Lambda_{\F}(\a),\a\big\rangle  \leq 0, & & \big\langle \Lambda_{\F} \circ i\,\nabla^2(\a),\a\big\rangle  \geq 0.
\eas
\end{lemma}

\begin{proof}
Using Proposition \ref{proposition:adelIsAdjointable}, we have
\bas
\big\langle i\nabla^2  \circ  \Lambda_{\F}(\a),\a\big\rangle = & \, i \big\langle \del_{\F} \circ \adel_{\F} \circ  \Lambda_{\F}(\a),\a\big\rangle + i \big\langle \adel_{\F} \circ  \del_{\F} \circ  \Lambda_{\F}(\a),\a\big\rangle \\
= & \, \big\langle \adel_\F \circ \Lambda_{\F} (\a),-i \del^\dagger_{\F}(\a)\big\rangle + i \big\langle \del_{\F} \circ \Lambda_{\F}(\a),\adel^\dagger_{\F}(\a)\big\rangle.
\eas
Since by assumption $\a$ is $\adel_\F$-harmonic, we have $\adel^\dagger_{\F}(\a) = 0$, and so, 
\bas
\big\langle i\nabla^2  \circ  \Lambda_{\F}(\a),\a\big\rangle = & \, \big\langle \adel_{\F} \circ  \Lambda_{\F}(\a), -i\del^\dagger_\F(\a)\big\rangle
=  \,  \big\langle\adel_{\F} \circ \Lambda_{\F}(\a),[\Lambda_{\F},\adel_{\F}](\a)\big\rangle,
\eas
where the second equality uses the Nakano identities. Harmonicity of $\a$ means that $[\Lambda_{\F},\adel_{\F}](\a) = - \adel_{\F} \circ \Lambda_{\F}(\alpha)$, and so: 
\bas
\big\langle  i\nabla^2  \circ  \Lambda_{\F}(\a),\a\big\rangle = & \, - \big\langle \adel_{\F} \circ \Lambda_{\F}(\a),\adel_{\F} \circ \Lambda_{\F}(\a)\big\rangle
=  -\|\adel_{\F} \circ \Lambda_{\F}(\a)\|^2 \leq 0.
\eas
Similarly, the second identity follows from
\bas
\big\langle\Lambda_{\F} \circ i\nabla^2 (\a),\a\big\rangle = & \, i \big\langle\Lambda_{\F} \circ \adel_{\F} \circ  \del_{\F}(\a),\a\big\rangle +  i \big\langle \Lambda_{\F} \circ \del_{\F} \circ  \adel_{\F}(\a),\a\big\rangle\\
= & \, i\big\langle[\Lambda_{\F},\adel_{\F}] \circ  \del_{\F}(\a),\a\big\rangle + i\big\langle\adel_{\F} \circ  \Lambda_{\F} \circ  \del_{\F}(\a),\a\big\rangle\\
= & \,  i\big\langle-i \del_{\F}^\dagger \circ  \del_{\F}(\a),\a\big\rangle + i\big\langle\Lambda_\F \circ  \del_{\F}(\a), \adel^\dagger_{\F}(\a)\big\rangle\\
= & \,  \big\langle\del_{\F}(\a),\del_\F(\a)\big\rangle\\
= & \, \|\del_{\F}(\a)\|^2 \geq 0.
\eas
\end{proof}

\begin{definition}\label{definition:PositivelineBundle}  Let $M$ be an algebra, and $(\Om^\bullet, \exd)$ a differential calculus over $M$ admitting a complex structure $(\Om^{(\bullet,\bullet)}, \odel, \del)$.

A holomorphic vector bundle $\fF = (\F, \odel_\F)$ is called {\em positive} if it admits an Hermitian structure $h\colon \ol{\F} \to \lvee\! \F$ such that the curvature form of the associated Chern connection $\nabla_\F$ satisfies
\bas
 L_{\F} = i \nabla_\F^2 = i \adel_\F \circ \del_\F + i \del_\F \circ \adel_\F,
\eas
where $L_\F$ is the Lefschetz operator for some positive definite K\"ahler form $\k \in \Om^{(1,1)}$ for which there is a state $\operatorname{\bf  f}$ such that $\int = \operatorname{\bf  f} \circ \vol_\k$ is closed and for which the twisted Dirac operator $D_{\odel_\F}$ is diagonalisable.
\end{definition}

Combining this definition of positivity with the positivity results presented in Lemma \ref{lem:posforf} above, we can prove the following noncommutative generalisation of the celebrated Kodaira vanishing theorem \cite{Kodaira53}.

\begin{theorem}[Kodaira Vanishing]\label{theorem:Kodaira}
Let $(\Om^{(\bullet, \bullet)}, \odel, \del)$ be an $A$-covariant $*$-differential calculus with a complex structure.  

For any positive vector bundle $(\F,\odel_\F, h)$
\bas
H^{(a,b)}(\F) = 0, & & \text{whenever } a+b > n.
\eas
\end{theorem}

\begin{proof}
Let $L_\F$ be the Lefschetz operator associated to the Chern connection of $(\F,h)$. Lemma \ref{lem:posforf}, together with the Lefschetz identities established in Lemma \ref{LIDS}, implies that
\begin{align*}
0 \leq \left\langle [\Lambda_{\F},L_{\F}](\a),\a\right\rangle = -\left\langle H_{\F}(\a),\a \right\rangle = \big(n-(a+b)\big) \|\a\|^2, & & \text{ for all } \a \in \H^{(a,b)}(\F).
\end{align*}
Thus when $a+b > n$, we must have that $\a = 0$, implying that there are no non-zero harmonic forms in $\H^{(a,b)}(\F)$. The required result now follows from the equivalence between harmonic forms and cohomology classes given in Theorem \ref{theorem:HodgeDecomposition}.
\end{proof}


\subsection{Noncommutative Fano structures}\label{subsection:Fano}

In this section, we consider a noncommutative version of a Fano manifold.  We start with the definition of a factorisable almost complex structure.  Throughout, we write $\D = (\Om^{(0,\bullet)}, \odel)$ for the Dolbeault dga.

\begin{definition}\label{definition:Factorisable}
An almost complex structure for an $A$-covariant differential $*$-calculus $\Om^{\bullet}$ over a $*$-algebra $M$ is called {\em factorisable} if the $M$-bimodule isomorphisms
\begin{align*}  \label{wedge-cond}
\wed\colon \Om^{(a,0)} \oby_M \Om^{(0,b)} \to \Om^{(a,b)},  & & \text{ and } & & \wed\colon \Om^{(0,b)} \oby_M \Om^{(a,0)} \to \Om^{(a,b)}
\end{align*}
are isomorphisms, for all $a,b$.
\end{definition}

\begin{remark}
The maps in Definition \ref{definition:Factorisable} are automatically $A$-comodule maps.
\end{remark}

\begin{proposition}\label{proposition:CanonicalBundleIsInvertible}
If the complex structure $(\Om^{(\bullet, \bullet)}, \odel, \del)$ is factorisable, then the left holomorphic vector bundle $(\Om^{(n,\bullet)}, \adel)$ is invertible (that is, it satisfies the conditions of Proposition \ref{proposition:LineBundleTensor}).
\end{proposition}

\begin{proof}
Since the calculus is orientable and factorisable, we have that 
\begin{align*}
\vol(- \wedge -) \colon \Om^{(n,0)} \otimes_M \Om^{(0,n)} &\to M \\
\omega \otimes \nu &\mapsto \vol(\omega \wedge \nu) 
\end{align*}
is an $M$-bimodule isomorphism.  Write $\sum_i \omega_i \otimes \nu_i = \vol(- \wedge -)^{-1}(1)$.  It is now clear that $\{\omega_i\}_i$ generates $\Om^{(n,0)}$ as a $M$-left module and that $\sum_i \omega_i \otimes \{\vol(- \wedge \nu_i)\}_i \in \Om^{(n,0)} \otimes_M \lvee \Om^{(n,0)}$ is a dual basis element.  It follows from \cite[Proposition II.2.6.12]{BourbakiAlgebraI13} that $\Om^{(n,0)}$ is a finitely generated projective left $M$-module, thus: $\Om^{(n,0)} \in \psAM\lproj_M$.  Similarly, we show that $\Om^{(0,n)} \in \psAM\lproj_M$.

As both $\Om^{(0,n)} \otimes_M -$ and $\lvee \Om^{(n,0)} \otimes_M -$ are left adjoint to $\Om^{(n,0)} \otimes_M -$, we have that $\Om^{(0,n)} \cong \lvee \Om^{(n,0)}$ in $\psAM\Mod_M$.  From this follows that $(\Om^{(n,\bullet)}, \odel)$ is an invertible vector bundle.
\end{proof}

\begin{proposition}\label{proposition:Fano}
{Let $\fF = (\mathcal{F},\odel_\F)$ be a $A$-covariant holomorphic vector bundle.  If $\lwedge \Om^{(n,\bullet)} \otimes_\D \fF$ is a positive vector bundle, then $H^{(0,k)}(\fF) = 0$, for all $k > 0$.}
\end{proposition}

\begin{proof}
If $\lwedge \Om^{(n,\bullet)} \otimes_D \fF$ is positive, then it follows from the Kodaira vanishing theorem that $H^{(n,k)}(\lwedge \Om^{(n,\bullet)} \otimes_M \mathcal{F}) = 0$, for all $k > 0$.  Using Proposition \ref{proposition:CanonicalBundleIsInvertible}, we find:
\[0 = H^{(n,k)}(\lwedge \Om^{(n,\bullet)} \otimes_\D \fF) = H^{(0,k)}(\Om^{(n,\bullet)} \otimes_\D \lwedge \Om^{(n,\bullet)} \otimes_\D \fF) = H^{(0,k)}(\fF),\]
for all $k > 0$, as required.
\end{proof}

Proposition \ref{proposition:Fano} motivates the following generalisation of a Fano structure for a complex manifold. 

\begin{definition}\label{definition:Fano}
Let $\big(\Om^{(\bullet,\bullet)}, \k\big)$ be a K\"ahler structure for an $A$-covariant differential $*$-calculus $\Om^\bullet$ over an \alg $M$ of total dimension $2n$.  We say that the K\"ahler structure is a {\em Fano structure} the complex structure is factorisable and $(\Om^{(n,0)}, \odel)$ is a positive line bundle.
\end{definition}

\begin{corollary}\label{corollary:Fano}
For a Fano structure $\big(\Om^{(\bullet,\bullet)}, \k\big)$, we have $H^{(0,k)}(\Om^{(0,\bullet)}, \odel) = 0$, for all $k > 0$.
\end{corollary}


\subsection{The quantum homogeneous K\"ahler space case}

Verifying positivity is easier in the covariant case, as the following lemma shows. Moreover, the notion of positivity is more natural.

\begin{lemma}
{For $\E$ a covariant holomorphic Hermitian line bundle, and  $\nabla$  its associated Chern connection,  there exists a uniquely defined left $G$-coinvariant $(1,1)$-form $\w_h$ such that the curvature operator $\nabla^2:\E \to \Om^2 \oby_M \E$  acts as }
\bas
\nabla^2(e) = \w_h \oby e, & & \text{ for all } e \in \E.
\eas
{We call $\w_h$ the {\em curvature form} of $\nabla$.}
\end{lemma}

\begin{proof}
Since the square of any connection is a left $M$-module map, and by assumption of covariance of the holomorphic structure we have that $\nabla$ is a left $G$-comodule map, we must have that  $\nabla^2$ is  a morphism in $\sgmm$. The image of $\nabla^2$ under the functor $\Phi$ is a left $H$-comodule map 
\bas
\Phi(\nabla^2)\colon \Phi(\E) \to \Phi\big(\Om^{(1,1)}\big) \oby \Phi\big(\E\big).
\eas
Since  $\Phi(\E)$ is one-dimensional, we must have that  $\Phi(\nabla^2)(e) = [\w] \oby_M e$, for some $\w \in \Om^{(1,1)}$. One-dimensionality of $\Phi(\E)$ also implies that $^\vee\Phi(\E)$  acts as an inverse object to $\Phi(\E)$ in $^H \mathrm{mod}_0$. For any $v \in {}^\vee \Phi(\E)$, we have
\bas
  \DEL_L\big( \nabla^2(e) \oby v \big) =  \DEL_L\big( [\w] \oby e \oby v \big) = [\w \m1] \oby [\w \0] \oby e \oby v .
\eas
Moreover,  
\bas
(\nabla^ 2 \oby \id) \circ \DEL_L(e \oby v)  =   (\nabla^ 2 \oby \id) (1 \oby e \oby v)  =   1 \oby [\w] \oby e \oby v.
\eas
That fact that $\Phi(\nabla^2)$ is a right $H$-comodule map  implies that  $\DEL_L[\w] = 1 \oby [\w]$. Considering the  image of $\Phi(\nabla^2)$ under $\Psi$ one now sees that $\nabla$ can be presented in the required form. 
\end{proof}

\section{Positive bundles for the quantum Grassmannians}\label{section:Grassmannians}

In this section we recall the definition of the quantum Grassmannians, and their associated \hk calculus with its unique covariant K\"ahler structure. We describe their line bundles (in particular their positive line bundles) and the fact that the canonical complex structure of each calculus is Fano. We finish by applying the framework of the paper to these examples, and as a consequence prove Bott's extension of the  Borel--Weil theorem for the quantum Grassmannians (see Theorem \ref{theorem:BorelBottWeil}).

\subsubsection{The Quantised Enveloping Algebra $U_q(\mathfrak{sl}_n)$}

Recall that the {\em Cartan matrix} of $\mathfrak{sl}_n$  is the matrix  $a_{ij} = 2\d_{ij} -\d_{i+1,j} - \d_{i,j+1}$, where 
  $i,j=1,\dots,n-1$. The {\em quantised enveloping \algn}  $U_q(\mathfrak{sl}_n)$ is the  \nc  \alg  generated by the elements   $E_i, F_i, K_i,K_i\inv$, for $ i=1, \ldots, n-1$,  subject to  the relations
\bas
& K_iK_j = K_jK_i, & ~ K_iK_i\inv = K_i\inv K_i = 1,  ~~~~~ \\
& K_iE_jK_i\inv =  q^{a_{ij}} E_j,  & K_iF_jK_i\inv = q^{-a_{ij}}F_j, ~~~\,~~~~  \\
&  E_iF_j - F_jE_i  = \d_{ij}\frac{K_i - K\inv_{i}}{(q-q\inv)},
\eas
along with the quantum Serre relations which we omit (see \cite[\textsection 6.1.2]{KSLeabh} for details).

A Hopf \alg structure is defined  on $U_q(\mathfrak{sl}_n)$ by setting
\bas
\DEL(K^{\pm 1}_i) = K^{\pm 1}_i \oby K^{\pm 1}_i, ~~  \DEL(E_i) = E_i \oby K_i + 1 \oby E_i, ~~~ \DEL(F_i) = F_i \oby 1 + K_i\inv  \oby F_i~~~~\\
 S(E_i) =  - E_iK_i\inv,    ~~ S(F_i) =  -K_iF_i, ~~~~  S(K_i) = K_i\inv,
 ~~ ~~\e(E_i) = \e(F_i) = 0, ~~ \e(K_i) = 1.     
\eas

We denote by $\bC_q[SU_n]$ the finitary Hopf dual of $U_q(\mathfrak{sl}_n)$. As is well known, $\bC_q[SU_n]$ admits a unique $*$-map giving it the structure of a compact quantum group algebra \cite[\textsection 11.5]{KSLeabh}. Moreover, as an \alg $\bC_q[SU_n]$ is generated by the set of matrix coefficients $\{u^i_j\}_{i,j =1}^n$  of the fundamental representation  of $U_q(\mathfrak{sl}_2)$.

\subsection{The quantum Grassmannians}

We now define the quantum Grassmannians  $\gr$ in terms of the quantised enveloping algebra  $U_q(\mathfrak{sl}_n)$. Consider the sub\alg $U_q(\mathfrak{l}_r) \sseq U_q(\mathfrak{sl}_n)$ generated by the elements
\bas
\{ E_i, F_i, K_j \mid  i,j = 1, \ldots, n-1; i \neq r\}.
\eas
In the following definition, we denote the finitary Hopf duals of $U_q(\mathfrak{sl}_n)$, and $U_q(\mathfrak{l}_r)$, by $\bC_q[SU_n]$, and  $U_q(\mathfrak{l}_r)^\circ$ respectively. 

\begin{definition}
Let $\pi_{n,r}\colon\bC_q[SU_n] \to U_q(\mathfrak{l}_r)^\circ$ be the Hopf \alg map dual to the inclusion $U_q(\mathfrak{l}_r) \hookrightarrow U_q(\mathfrak{sl}_n)$, and denote $\bC_q[L_r] \coloneqq \pi_{n,r}\left(\bC_q[SU_n]\right)$.  The {quantum Grassmannian} $\bC_q[\text{Gr}_{n,r}]$ is the quantum homogeneous space associated to the Hopf algebra  projection  $\pi_{n,r}\colon \bC_q[SU_n] \to \bC_q[L_r]$.
\end{definition}

As is easily seen, we can alternatively describe $\gr$ in the following way
\bas
\gr = \{ g \in \bC_q[SU_n] \mid g\tl l = \e(l) g, \text{ for all } l \in U_q(\mathfrak{l}_r)\}.
\eas
Moreover,
$\bC_q[L_S]$ is a compact quantum groups algebra. The quantum Grassmannians admit a generating composed of of quantum minors. Recall that for $I \coloneqq \{i_1, \dots, i_k\}$ and $J:=\{j_1, \dots, j_k\}$ two subsets of $\{1,\dots, n\}$,  the associated {\em quantum  minor} $z^I_J$ is the element 
\bas
z^I_J \coloneqq \sum_{\s \in S_k} (-q)^{\ell(\s)}u^{\s(i_1)}_{j_1} \cdots u^{\s(i_k)}_{j_k} =  \sum_{\s \in S_k} (-q)^{\ell(\s)} u^{i_1}_{\s(j_1)} \cdots u^{i_k}_{\s(j_k)}.
\eas
When $J = \{1, \dots, r\}$, we will always denote $z^I := z^I_J$, and when $J = \{r+1, \dots, n\}$, we will denote  $\ol{z}^I := \ol{z}^I_J$. As shown in  \cite[Theorem 2.5]{stok}, and  \cite[Proposition 3.2]{HK}, a set of generators for the \alg $\gr$ is given by the  elements 
\bas
\{ z^{IJ} \coloneqq z^I\ol{z}^J \,|\, |I|=r, |J| = n-r \}.
\eas

We next describe the line bundles over $\gr$. Denote by  $U_q(\mathfrak{k_r}) \sseq U_q(\mathfrak{sl}_n)$ the  sub\alg of $U_q(\mathfrak{sl}_n)$ generated by the set of elements
\bas
\{ E_i, F_i, K_i \mid i = 1, \ldots, n-1; i \neq r\}.
\eas
Analogous to the definition of the quantum Grassmannians, this defines a quantum homogeneous space which we denote by $\bC_q[S^{n,r}]$. It follows easily from the description of  the generators of $\gr$ that \alg  $\bC_q[S^{n,r}]$ is generated by the elements 
\bas
\{ z^{I}, \,\ol{z}^{J} \coloneqq z^I\az^J \mid |I|=r, |J| = n-r \},
\eas
see \cite[Proposition 3.9]{MMFCM} for details.  By construction, $\bC_q[S^{n,r}]$ admits a right action by the sub\alg of $U_q(\mathfrak{sl}_n)$ generated by the elements $K_r^{\pm}$. As was shown in \cite[\textsection 3.3]{MMFCM}, this induces a (strong) $\bZ$-grading on $\bC_q[S^{n,r}]$ uniquely defined by 
\begin{align*}
\text{\rm deg}(z^I) = 1, & & \text{\rm deg}(\ol{z}^J) = -1, & & \text{ for all } I,J \sseq \{1, \dots, r\}, {\text{\rm ~ with \, }} \, |I| = r, |J| = n-r. 
\end{align*}
The homogeneous subspaces of this grading are exactly the equivariant line bundles $\E_i$, for $i\in \bZ$,  over $\gr$.  In particular $\E_0 = \gr$. Each line bundle $\E_i$ has a real one-dimensional space of Hermitian structures, which we fix here for one and for all.

\begin{example}
For the special case of $r=1$, we see that $\bC_q[SU_1] \simeq \bC$, and $\pi_{n,r}$ reduces to $\a_1$. Hence the associated quantum homogeneous space is {\em quantum projective space}  $\cpn$  as introduced in \cite{Mey}. The situation for $r = n-1$ is analogous

For the case of $r=1$, the \alg $\bC_q[S^{n,r}]$ reduces to the odd dimensional quantum spheres, and $\E_i$ the line bundles over $\cpn$. Finally, we note that when  $n=1$ we recover the quantum Hopf fibration over the standard Podle\'s sphere.
\end{example}

\subsection{The \hk calculus}

A brief presentation of the Heckenberger--Kolb calculus of  the quantum Grassmannians is given, starting  with the classification of first-order differential calculi over $\bC_q[Gr_{n,k}]$. 

\subsubsection{First-Order Differential Calculi}

We begin by recalling some details about first-order differential calculi necessary for our presentation of the \hk classification below. A {\em first-order differential calculus} over an algebra $M$ is a pair $(\Om^1,\exd)$, where $\Omega^1$ is an $M$-$M$-bimodule and $\exd\colon M \to \Omega^1$ is a linear map for which the {\em Leibniz rule}, 
$
\exd(ab)=a(\exd b)+(\exd a)b, \text{ for } a,b,\in A,
$
holds and for which $\Om^1 = \spn_{\bC}\{a\exd (b) \mid a,b \in M\}$.  The notions of differential map,  and left-covariance when the calculus is defined over a \qhs $M$, have obvious first-order analogues, for details see \cite[\textsection 2.4]{MMF2}. The {\em direct sum} of two first-order differential calculi $(\Om^1,\exd_{\Om})$ and $(\G^1,\exd_{\G})$ is the first-order calculus $(\Om^1 \oplus \Gamma^1, \exd_\Om + \exd_\G)$.  Finally, we say that  a left-covariant first-order calculus over $M$ is {\em irreducible} if it does not possess any non-trivial quotients by a left-covariant $M$-bimodule. 

We say that a differential calculus $(\G^\bullet,\exd_\G)$ {\em extends} a first-order calculus $(\Om^1,\exd_{\Om})$ if there exists a bimodule isomorphism $\f\colon\Om^1 \to \G^1$ \st $\exd_\G = \f \circ \exd_{\Om}$. It can be shown  \cite[\textsection 2.5]{MMF2} that any first-order calculus admits an extension $\Om^\bullet$ which is maximal  in the sense that there exists a unique differential map from $\Om^\bullet$ onto any other extension of $\Om^1$. We call this extension the {\em maximal prolongation} of the first-order calculus.

\subsubsection{The \hk classification}

We begin with the classification of covariant first-order calculi over the quantum Grassmannians, and then describe how the classification naturally leads to a total differential calculus over $\gr$.

\begin{theorem} \cite[\textsection 2]{HK} \label{HKClass}
There exist exactly two non-isomorphic irreducible left-covariant first-order differential calculi of finite dimension over $\bC_q[Gr_{n,k}]$.
\end{theorem}

The direct sum of $\Om\hol$ and $\Om\ahol$ is a $*$-calculus which we call the {\em Heckenberger--Kolb calculus} of $\bC_q[Gr_{n,k}]$ and which we denote  by $\Om^1_q(\cgr)$.

\begin{theorem}[\cite{HKdR}] \label{thm:CalcBasis} Denoting by $\Om^\bullet_q(\cgr) = \bigoplus_{k \in \bN_0} \Om^k_q(\cgr)$ the \mpr of $\Om^1_q(\cgr)$, each $\Om^k_q(\cgr)$ has classical dimension, which is to say, 
\bas
\dim\big(\Om^k_q(\cgr)\big) = \binom{2r(n-r)}{k}, & & k = 0, \dots, 2r(n-r),
\eas 
and $\Om^k_q(\cgr) = 0$, for $k > 2r(n-r)$. Moreover, the decomposition of $\Om^1_q(\cgr)$ into irreducible sub-calculi induces  a pair of opposite factorisable complex structures for  the total calculus of $\Om^1(\cgr)$.
\end{theorem}

The top holomorphic and anti-holomorphic forms of the complex structure are line bundles and satisfy the identities
\bas
\Om^{(0,n)} = \E_{r(n-r)+1}, & & \Om^{(n,0)} = \E_{-r(n-r)+1}.
\eas
Finally,we recall the following classification of covariant K\"ahler  structures for the calculus.

\begin{theorem}\label{thm:KOBS}
Up to real scalar multiple there exists a unique  $\bC_q(SU_n)$-coinvariant form $\k \in \Om^{(1,1)}$.  Moreover, for all but a finite set $\F$ of values of $q$, the pair $(\Om^{(\bullet,\bullet)}, \k)$ is a covariant K\"ahler structure for $\Om^\bullet(\cgr)$.
\end{theorem}

\begin{corollary}
The integral associated to $(\Om^{(\bullet,\bullet)}, \k)$ and haar state of $\bC_q[SU_n]$ is closed.
\end{corollary}
\begin{proof}
The case of $\cpn$ has been been dealt with in \cite[Lemma 3.4.4]{MMF3}, so we will restrict to the case of $r \neq 1,n-1$. Since $\Omega\hol$ and $\Omega\ahol$ are irreducible calculi, and by Theorem \ref{thm:CalcBasis} each has dimension greater than $1$, neither can contain a copy of the trivial comodule.  Closure of the integral  now follows from \cite[Lemma 3.3]{MMF3}.
\end{proof}

Finally, we come to the question of positive definiteness, and the implied quantum generalisation of Bott's extension of the Borel--Weil theorem.

\begin{proposition} \cite{KOBS}
For $\left(\Om^{(\bullet,\bullet)}, \k\right)$ the K\"ahler structure for $\Om^\bullet_q(\text{Gr}_{n,r})$ identified in Theorem \ref{thm:KOBS}, there exists an interval $1 \in I_{n,r} \sseq \bR$, such that if $q \in I_{n,r}$, and $q \notin \mathcal{F}$, then $\left(\Om^{(\bullet,\bullet)}, \k\right)$ is  positive definite.  
\end{proposition}

\begin{proposition} \cite{KOBS}
For $i > 0$, the line bundle $\E_i$ is positive, and hence the complex structure $\Om^{(\bullet,\bullet)}$ is of Fano type (see Definition \ref{definition:Fano}). 
\end{proposition}

The main result of this section, the Bott--Borel--Weil theorem for the quantum Grassmannians,  now follows from Proposition \ref{proposition:Fano} and \cite[Theorem 6.8]{MMFCM}. In the statement of the theorem we denote by $V(k)$ the irreducible left $\bC_q[SU_n]$-comodule defined by 
\bas
V(k,r) \coloneqq \operatorname{Span}_{\bC}\left\{ z^{I_1} \cdots z^{I_k} \mid I_j \sseq \{1,\dots, r\}, I_j| = r; j =1, \dots, k\right\}.
\eas  

\begin{theorem}[Bott--Borel--Weil]\label{theorem:BorelBottWeil}
For all line bundles $\E_k$, with $k \geq 0$,  over the quantum Grassmannian $\gr$, it holds that 
\bas
H^{(0,0)}(\E_{k}) = V(r,k) , & & ~~  H^{(0,i)}(\E_k) = 0, ~~ \text{ for all ~ } i >0.
\eas
\end{theorem}

\subsection{Quantum flag manifolds of cominiscule type}

Let  $\mathfrak{g}$ be a complex semisimple Lie \alg of rank $r$ and $U_q(\mathfrak{g})$ the corresponding Drinfeld--Jimbo quantised enveloping algebra \cite[\textsection 6.1]{HK}. Consider the following generalisation of the definition of the quantum Grassmannians. For $S$ a subset of simple roots,  denote by $\pi_S:\bC_q[G] \to \bC_q[L_S]$  the surjective Hopf \alg map dual to the inclusion $U_q(\mathfrak{l}_S) \hookrightarrow U_q(\mathfrak{g})$, where
\bas
U_q(\mathfrak{l}_S) \coloneqq \la K_i, E_j, F_j \mid i = 1, \ldots, r; j \in S \ra.
\eas 
The quantum homogeneous space of this coaction is called the {\em quantum flag manifold} corresponding to $S$, and is denoted by $\bC_q[G/L_S]$. (See \cite{HK,HKdR} for a more detailed presentation of this definition.)

If $S = \{1, \ldots, r\}\bs \a_i$ where $\a_i$ appears only once in the longest root, then we say that the quantum flag manifold is {\em irreducible}. It follows that each quantum Grassmannian $\cgr$ is an irreducible quantum flag manifold. 

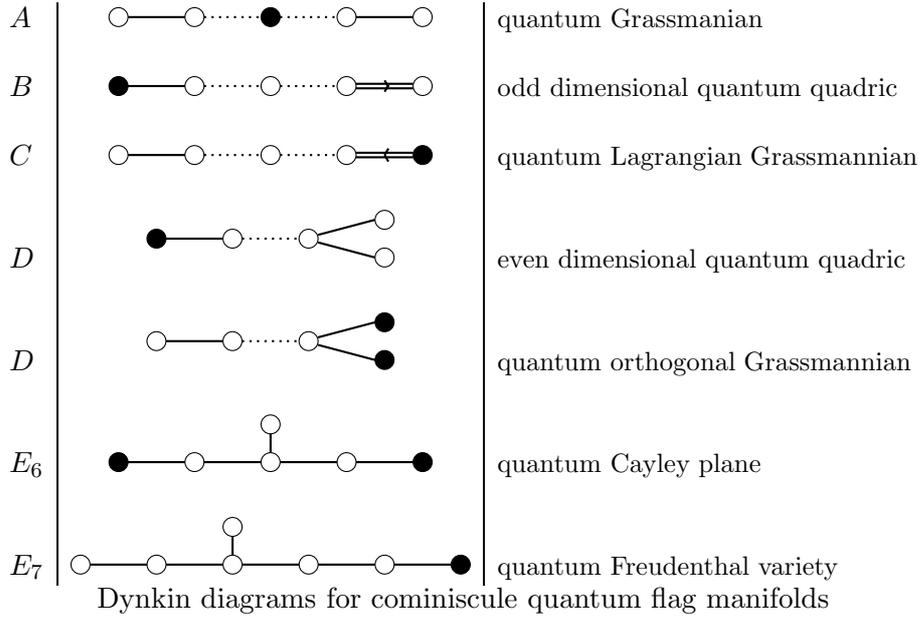
\begin{center}
\begin{tabular}{ l |c | l }
$A$&
\begin{tikzpicture}[scale=.5]
\draw
(0,0) circle [radius=.25] 
(8,0) circle [radius=.25] 
(2,0)  circle [radius=.25]  
(6,0) circle [radius=.25] ; 

\draw[fill=black]
(4,0) circle  [radius=.25] ;

\draw[thick,dotted]
(2.25,0) -- (3.75,0)
(4.25,0) -- (5.75,0);

\draw[thick]
(.25,0) -- (1.75,0)
(6.25,0) -- (7.75,0);
\end{tikzpicture} & \small quantum Grassmanian \\
 & \\
$B$& 
\begin{tikzpicture}[scale=.5]
\draw
(4,0) circle [radius=.25] 
(2,0) circle [radius=.25] 
(6,0)  circle [radius=.25]  
(8,0) circle [radius=.25] ; 
\draw[fill=black]
(0,0) circle [radius=.25];

\draw[thick]
(.25,0) -- (1.75,0);

\draw[thick,dotted]
(2.25,0) -- (3.75,0)
(4.25,0) -- (5.75,0);

\draw[thick] 
(6.25,-.06) --++ (1.5,0)
(6.25,+.06) --++ (1.5,0);                      

\draw[thick]
(7,0.15) --++ (-60:.2)
(7,-.15) --++ (60:.2);
\end{tikzpicture} & \small odd dimensional quantum  quadric\\ 
 & \\ 

$C$& 
\begin{tikzpicture}[scale=.5]
\draw
(0,0) circle [radius=.25] 
(2,0) circle [radius=.25] 
(4,0)  circle [radius=.25]  
(6,0) circle [radius=.25] ; 
\draw[fill=black]
(8,0) circle [radius=.25];

\draw[thick]
(.25,0) -- (1.75,0);

\draw[thick,dotted]
(2.25,0) -- (3.75,0)
(4.25,0) -- (5.75,0);

\draw[thick] 
(6.25,-.06) --++ (1.5,0)
(6.25,+.06) --++ (1.5,0);                      

\draw[thick]
(7,0) --++ (60:.2)
(7,0) --++ (-60:.2);
\end{tikzpicture} &\small  quantum Lagrangian Grassmannian \\

 & \\ 

$D$& 
\begin{tikzpicture}[scale=.5]

\draw[fill=black]
(0,0) circle [radius=.25] ;

\draw
(2,0) circle [radius=.25] 
(4,0)  circle [radius=.25]  
(6,.5) circle [radius=.25] 
(6,-.5) circle [radius=.25];

\draw[thick]
(.25,0) -- (1.75,0)
(4.25,0.1) -- (5.75,.5)
(4.25,-0.1) -- (5.75,-.5);

\draw[thick,dotted]
(2.25,0) -- (3.75,0);
\end{tikzpicture} &\small  even dimensional quantum quadric \\ 
 & \\ 

$D$ & 
\begin{tikzpicture}[scale=.5]
\draw
(0,0) circle [radius=.25] 
(2,0) circle [radius=.25] 
(4,0)  circle [radius=.25] ;

\draw[fill=black] 
(6,.5) circle [radius=.25] 
(6,-.5) circle [radius=.25];

\draw[thick]
(.25,0) -- (1.75,0)
(4.25,0.1) -- (5.75,.5)
(4.25,-0.1) -- (5.75,-.5);

\draw[thick,dotted]
(2.25,0) -- (3.75,0);
\end{tikzpicture} &\small  quantum orthogonal Grassmannian \\
 & \\ 

$E_6$& \begin{tikzpicture}[scale=.5]
\draw
(2,0) circle [radius=.25] 
(4,0) circle [radius=.25] 
(4,1) circle [radius=.25]
(6,0)  circle [radius=.25] ;

\draw[fill=black] 
(0,0) circle [radius=.25] 
(8,0) circle [radius=.25];

\draw[thick]
(.25,0) -- (1.75,0)
(2.25,0) -- (3.75,0)
(4.25,0) -- (5.75,0)
(6.25,0) -- (7.75,0)
(4,.25) -- (4, .75);
\end{tikzpicture}

 &\small  quantum Cayley plane \\
 & \\ 
$E_7$& 
\begin{tikzpicture}[scale=.5]
\draw
(0,0) circle [radius=.25] 
(2,0) circle [radius=.25] 
(4,0) circle [radius=.25] 
(4,1) circle [radius=.25]
(6,0)  circle [radius=.25] 
(8,0) circle [radius=.25];

\draw[fill=black] 
(10,0) circle [radius=.25];

\draw[thick]
(.25,0) -- (1.75,0)
(2.25,0) -- (3.75,0)
(4.25,0) -- (5.75,0)
(6.25,0) -- (7.75,0)
(8.25, 0) -- (9.75,0)
(4,.25) -- (4, .75);
\end{tikzpicture} &\small   quantum Freudenthal variety\\
\end{tabular}
{\mbox{Dynkin diagrams for cominiscule quantum flag manifolds}}
\end{center}

The classification of first-order calculi over the quantum Grassmannians extends to this more general setting.

\begin{theorem} \cite[\textsection 2]{HK}.
There exist exactly two non-isomorphic irreducible left covariant first-order differential calculi of finite dimension over $\bC_q[G/L_S]$. 
\end{theorem}

The maximal prolongation of the direct sum of these two calculi is shown in cite{HKdR}  to have a unique covariant complex structure $\Om^{(\bullet,\bullet)}$. Using an elementary  representation theoretic argument, it can be shown that the each  $\Om^{(1,1)}$ contains a left-coinvariant  form $\k$ that is unique up to scalar multiple.

\begin{conjecture}
For every irreducible quantum flag manifold $\bC_q[G_0/L_0]$,  the pair \linebreak $(\Om^{(\bullet,\bullet)},\k)$ is a covariant K\"ahler structure for the \hk calculus, and  the associated complex structure is of Fano type.
\end{conjecture}

\appendix

\section{Dualities for comodule algebras}\label{Section:ComoduleDualities}

\renewcommand{\hol}{\operatorname{hol}}
\renewcommand{\ev}{\operatorname{ev}}
\renewcommand{\coev}{\operatorname{coev}}

\subsection{Comodule algebras}

Let $A$ and $B$ be Hopf algebras.  We write $\psA\Mod^B$ for the category of $A$-$B$-bicomodules.  Let $V \in \psA\Mod^B$, and let $\Delta^A\colon V \to A \otimes V$ and $\Delta^B\colon V \to V \otimes B$ be the left $A$-coaction and right $B$-coaction, respectively.  We use the Sweedler notation for the coaction:
\begin{align*}
\Delta^A(v) &= v_{(-1)} \otimes v_{(0)}, \\
\Delta^B(v) &= v_{(0)} \otimes v_{(1)}, \\
(\Delta^A \otimes 1) \circ \Delta^B(v) = (1 \otimes \Delta^B) \circ \Delta^A(v) &= v\m1 \otimes v\0 \otimes v\1.
\end{align*}

An \emph{$A$-$B$-bicomodule algebra} is an algebra object in the monoidal category $\psA\Mod^B$.  Explicitly, $M$ is an $A$-$B$-bicomodule algebra if $M$ is both a $\bC$-algebra and an $A$-$B$-bicomodule satisfying the following compatibility conditions
\begin{align*}
\Delta^A(m.n) &= m\m1 n\m1 \otimes m\0 n\0, & \Delta^A(1_M) = 1_A \otimes 1_M, \\
\Delta^B(m.n) &= m\0 n\0 \otimes m\1 n\1, & \Delta^B(1_M) = 1_M \otimes 1_B.
\end{align*}

Let $M$ and $N$ be $A$-$B$-comodule algebras.  We write $\psAM\Mod^B_N$ for the Eilenberg-Moore category of the monad $M \otimes_\bC (-) \otimes_\bC N\colon \psA\Mod^B\to \psA\Mod^B$.  The objects are given by pairs $(X, \mu\colon M \otimes X \otimes N \to X)$ where $X \in \psA\Mod^B$ and
\begin{align*}
\Delta^A(\mu(m \otimes x \otimes n)) &= m_{(-1)} x_{(-1)} n_{(-1)} \otimes \mu(m_{(0)} \otimes x_{(0)} \otimes n_{(0)}), \\
\Delta^B(\mu(m \otimes x \otimes n)) &= m_{(0)} x_{(0)} n_{(0)} \otimes \mu(m_{(1)} \otimes x_{(1)} \otimes n_{(1)}),
\end{align*}
satisfying the usual unital and associativity conditions.  Thus, explicitly, $\mu\colon M \otimes X \otimes N \to X$ is an $A$-$B$-bicomodule map which makes $X$ into an $M$-$N$-bimodule.  A morphism $f\colon X \to Y$ in $\psAM\Mod^B_N$ is a $\bC$-linear map $X \to Y$ which is both an $M$-$N$-bimodule map and an $A$-$B$-bicomodule map.

\begin{notation}\label{notation:LabeledMods}
\begin{enumerate}
\item We adopt the convention that, if any of the Hopf algebras $A$ or $B$, or any of the $A$-$B$-comodule algebras $M$ or $N$ are $\bC$, we do not include them in the notation $\psAM\Mod^B_N$.  In particular, we write $\Mod$ for the category of $\bC$-vector spaces.
\item We write $\psAM\lproj^B_N$ or $\psAM\rproj^B_N$ for the full subcategories of $\psAM\Mod^B_N$ given by those objects which are finitely generated and projective as left $M$-modules or as right $N$-modules, respectively (thus, forgetting about the bimodule structure and the $A$-$B$-bicomodule structure). 
\item We write $\psA\Hom(X,Y), \psM\Hom(X,Y), \Hom^B(X,Y)$, and $\Hom_N(X,Y)$ for the $\bC$-linear maps which respect the left $A$-comodule structure, the left $M$-module structure, the right $B$-comodule structure, and the right $N$-module structure, respectively.  We also consider the obvious extension of this notation, using multiple subscripts and superscripts, if we consider morphisms preserving multiple structures.
\end{enumerate}
\end{notation}

\subsection{Motivation - dual modules and dual bases}\label{subsection:DualModules}

Let $M$ be a complex algebra.  We consider the contravariant functors
\begin{align*}
\lvee(-) = \psM\Hom(-,M)\colon \psM\Mod \to \Mod_M, && (-)^\vee = \Hom_M(-,M)\colon \Mod_M \to \psM\Mod.
\end{align*}
If $X \in \psM\Mod$, then we call $\lvee X$ the \emph{dual module} of $X$.  We recall the following standard proposition.

\begin{proposition}\label{proposition:MainPropositionDuals}
\begin{enumerate}
\item The contravariant functors $\lvee(-)\colon \psM\lproj \to \rproj_M$ and $(-)^\vee\colon {}\rproj_M \to \prescript{}{M}\lproj$ are quasi-inverse.
\item For any finitely generated projective $P \in \psM\lproj$, the functor $(P \otimes_\bC -) \colon \Mod \to \psM\Mod$ is left adjoint to $(\lvee P \otimes_M -) \colon \psM\Mod \to \Mod$.
\end{enumerate}
\end{proposition}

\begin{proof}
The first statement follows from \cite[II.2.7]{BourbakiAlgebraI13}.  For the second statement, we use that $\psM\Hom(P,-)$ is right adjoint to $P \otimes -$.  It follows from \cite[II.4.2]{BourbakiAlgebraI13} (see also \cite[Theorem B.15]{JensenLenzing89}) that $\psM\Hom(P,-) \cong \lvee P \otimes_M -$.
\end{proof}

\begin{remark}
In particular, for any $P \in \psM\lproj$, there are isomorphisms $\psM\Hom(P \otimes X, Y) \cong \Hom(X,\lvee P \otimes_M Y)$, natural in $X \in \Mod$ and $Y \in \psM\Mod$.
\end{remark}
  
The unit $\eta$ and counit $\epsilon$ of the adjunction $(P \otimes_M -) \dashv (\lvee P \otimes_\bC -)$ are induced by maps
\begin{align*}
\epsilon_M\colon P \otimes_\bC {\lvee P} &\to M & \eta_\bC \colon \bC & \to \lvee P \otimes_M P \\
p \otimes f &\mapsto f(p)  & 1 &\mapsto \sum_i f_i \otimes p_i.
\end{align*}
The triangle identities for adjoint functors show that:
\begin{align*}
\forall p \in P\colon p = \sum_i f_i(p) p_i, && \forall f \in {^\vee P}\colon f = \sum_i f(p_i) f_i.
\end{align*}
We call $\eta_\bC(1) = \sum_i f_i \otimes p_i \in \lvee P \otimes_\bC P$ a \emph{dual basis element}.  Note that $\{p_i\}_i \subseteq P$ generates $P$.

\begin{remark}
When $M$ is a comodule algebra over a Hopf algebra $A$, there is an $A$-equivariant version of Proposition \ref{proposition:MainPropositionDuals}.  We refer to \cite{Ulbrich90} for more information.
\end{remark}

The following result is standard (see for example \cite[II.4.2]{BourbakiAlgebraI13}).  We provide a proof for the benefit of the reader.

\begin{proposition}\label{proposition:TensorCommutesWithHom}
Let $M$ be an algebra and let $P \in {\psM\lproj}$.  The linear map
\begin{align*}
\varphi\colon V \otimes {\psM\Hom(P,X)} \otimes_\bC W &\to {\psM\Hom(P,V \otimes X \otimes_\bC W)} \\
v \otimes f \otimes w &\mapsto (p \mapsto v \otimes f(p) \otimes w)
\end{align*}
(natural in $X \in {\psM\Mod}$ and $V, W \in \Mod$) is invertible.  Here, the left $M$-module structure of $V \otimes X \otimes_\bC W$ is given by $m(v \otimes x \otimes w) = v \otimes mx \otimes w$, for all $m \in M, v \in V, w \in W,$ and $x \in X.$
\end{proposition}

\begin{proof}
The inverse map is given by the following sequence of natural isomorphisms:
\begin{align*}
\psM\Hom(P,W \otimes X \otimes_\bC W) &\cong \lvee P \otimes_M (W \otimes_\bC X \otimes_\bC V) \\
& \cong W \otimes_\bC (\lvee P \otimes_M X) \otimes_\bC V \cong W \otimes_\bC \psM\Hom(P,X) \otimes_\bC V.
\end{align*}
\end{proof}

\subsection{Duality for comodule algebras}

The aim of this section is to obtain results similar to that of \textsection\ref{subsection:DualModules}, starting from an $A$-$B$-bicomodule algebra $M$, and allowing for a right $N$-action.  The analog of Proposition \ref{proposition:MainPropositionDuals} will be Theorem \ref{theorem:TensorWithDual} below.

Our first step is, given $X \in \psAM\Mod^B_N$, to construct a functor $\psM\hom(X,-)\colon \psAM\Mod^B_N \to \psAN\Mod^B_N$ which is right adjoint to $X \otimes_N -\colon \psAN\Mod^B_N \to \psAM\Mod^B_N$.  This will be done in Theorem \ref{theorem:AdjointWithhom}.

\subsubsection{An enrichment}

Let $X,Y \in \psAM\Mod_N^B$.  We start by considering the vector space $\psM\Hom(X,Y)$ of all left $M$-module morphisms (disregarding, for the moment, the $A$-$B$-bicomodule structure and the right $N$-module structure) and a first approximation of an $A$-$B$-bicomodule structure.

\begin{lemma}\label{lemma:DefinitionDelta}
Let $A,B,X,Y$ as above.  We consider the maps
\begin{align*}
(\Delta^A)'\colon {\psM\Hom(X,Y)} &\to {\psM\Hom(X,A \otimes_\bC Y)} \\
f &\mapsto \left(x \mapsto S(x_{(-1)}) [f(x_{(0)})]_{(-1)} \otimes [f(x_{(0)})]_{(0)} \right)\\
(\Delta^B)'\colon {\psM\Hom(X,Y)} &\to {\psM\Hom(X,Y \otimes_\bC B)} \\
f &\mapsto \left(x \mapsto [f(x_{[0]})]_{[0]} \otimes  S^{-1}(x_{[1]}) [f(x_{[0]})]_{[1]} \right)
\end{align*}
where the $M$-$N$-bimodule structure on $A \otimes Y$ and on $Y \otimes B$ is given by $m(a \otimes y)n = a \otimes myn$ and $m(y \otimes b)n = myn \otimes b$, for all $a \in A, b \in B, y \in Y, m \in M$, and $n \in N$.

Moreover, the following diagram commutes:
\begin{equation}\label{equation:DiagramForDelta}
\begin{tikzcd}
{\psM\Hom(X,Y)} \arrow[d, "(\Delta^B)'"] \arrow[r, "(\Delta^A)'"] & {\psM\Hom(X,A \otimes Y)} \arrow[d, "(\Delta^B)'"] \\
\psM\Hom(X,Y \otimes B) \ar[r, "(\Delta^A)'"] & {\psM\Hom(X,A \otimes Y \otimes_\bC B)}
\end{tikzcd}
\end{equation}
\end{lemma}

\begin{proof}
One needs to verify that $(\Delta^A)'(f)$ is indeed a morphism of left $M$-modules.  We have
\begin{eqnarray*}
mx &\mapsto& S((mx)_{( -1 )}) [f((mx)_{( 0 )})]_{( -1 )} \otimes [f((mx)_{( 0 )})]_{( 0 )} \\
&=& S(m_{( -1 )} x_{( -1 )}) [f(m_{( 0 )}x_{( 0 )})]_{( -1 )} \otimes [f(m_{( 0 )}x_{( 0 )})]_{( 0 )} \\
&=& S(x_{( -1 )}) S(m_{( -2 )}) m_{( -1 )} [f(x_{( 0 )})]_{( -1 )} \otimes m_{( 0 )} [f(x_{( 0 )})]_{( 0 )} \\
&=& S(x_{( -1 )}) \epsilon_A(m_{( -1 ))}) [f(x_{( 0 ))})]_{( -1 )} \otimes m_{( 0 )} [f(x_{( 0 )})]_{( 0 )} \\
&=& S(x_{( -1 )})  [f(x_{( 0 ))})]_{( -1 )} \otimes m [f(x_{( 0 )})]_{( 0 )} \\
&=& m\left[ S(x_{( -1 )})  [f(x_{( 0 )})]_{( -1 )} \otimes [f(x_{( 0 )})]_{( 0 )} \right]
\end{eqnarray*}
as required.  Hence, the morphism $(\Delta^A)'\colon \psM\Hom(X,Y) \to {\psM\Hom(X,A \otimes Y)}$ is well-defined.  Similarly, one checks that $(\Delta^B)'\colon {\psM\Hom(X,Y)} \to {\psM\Hom(X,Y \otimes_\bC B)}$ is well-defined.  Finally, the commutativity of the diagram follows from the fact that $X$ and $Y$ are $A$-$B$-bicomodules.
\end{proof}

\begin{definition}\label{definition:SmallHom}
We write $\Delta'\colon \psM\Hom(X,Y) \to \psM\Hom(X,A \otimes Y \otimes_\bC B)$ for the map given by \eqref{equation:DiagramForDelta} in Lemma \ref{lemma:DefinitionDelta}.  We recall the morphism $\varphi\colon A \otimes \psM\Hom(X,Y) \otimes B \to \psM\Hom(X,A \otimes_\bC Y \otimes B)$ from Propoosition \ref{proposition:TensorCommutesWithHom}.  We define $\psM\hom(X,Y)$ as the pullback of the diagram
\begin{center}
\begin{tikzcd}
{\psM\Hom(X,Y)} \arrow[r, "\Delta'"] & {\psM\Hom(X,A \otimes Y \otimes B)} \\
\psM\hom(X,Y) \arrow[u, dashrightarrow] \ar[r, "\Delta",dashrightarrow] & A \otimes {\psM\Hom(X,Y)} \otimes B \arrow[u,"\varphi"]
\end{tikzcd}
\end{center}
thus, $\psM\hom(X,Y) \coloneqq (\Delta')^{-1}(A \otimes {\psM\Hom(X,Y)} \otimes B) \subseteq \psM\Hom(X,Y)$.
\end{definition}

\begin{remark}\label{remark:Abouthom}
\begin{enumerate}
\item If $X$ is a finitely generated projective left $M$-module (such that $\varphi$ is invertible), then $\psM\hom(X,Y) = \psM\Hom(X,Y)$ and $\Delta = \varphi^{-1} \circ \Delta'$.
\item For any $f \in \psM\hom(X,Y)$, we have $\Delta^A(f) = f\m1 \otimes f\0$ such that, for all $x \in X:$
\[f\m1 \otimes f\0(x) = S(x_{(-1)}) [f(x_{(0)})]_{(-1)} \otimes [f(x_{(0)})]_{(0)}.\]
Similarly, we have $\Delta^B(f) = f\0 \otimes f\1$ such that, for all $x \in X:$
\[ f\0(x) \otimes f\1 = [f(x_{(0)})]_{(0)} \otimes  S^{-1}(x_{(1)}) [f(x_{(0)})]_{(1)}. \] 
\end{enumerate}
\end{remark}

\begin{proposition}\label{proposition:AllStructuresOnHom}
For all $X,Y \in \psAM\Mod^B_N$, we have $\psM\hom(X,Y) \in \psAN\Mod^B_N$.  The $A$-$B$-bicomodule structure is given by $\Delta\colon {\psM\hom(X,Y)} \to A \otimes {\psM\Hom(X,Y)} \otimes B$.
\end{proposition}

\begin{proof}
It is shown in \cite[Lemma 2.2]{Ulbrich90} that $\psM\hom(X,Y)$ is a left $A$-comodule.  Showing that $\psM\hom(X,Y)$ is an $A$-$B$-bicomodule is similar.  To show that $\psM\hom(X,Y) \in \prescript{A}{N}\Mod^B_N$, we need to show that
\begin{align*}
\Delta^A(nfn') &= n_{(-1)} f_{(-1)} n'_{(-1)} \otimes n_{(0)} f_{(0)} n'_{(0)}, \\
\Delta^B(nfn') &= n\0 f\0 n'\0\otimes n\1 f\1 n'\1,
\end{align*}
{for all $f \in \psM\hom(X,Y)$ and $n,n' \in N.$}  We prove the former; the later statement is proved similarly.  For all $x \in X$, we have
\begin{align*}
&(1 \otimes \ev_x) (n_{(-1)} f_{(-1)} n'_{(-1)} \otimes n_{(0)} f_{(0)} n'_{(0)}) \\
&= n_{(-1)} f_{(-1)} n'_{(-1)} \otimes (n_{(0)} f_{(0)} n'_{(0)})(x) \\
&= n_{(-1)} f_{(-1)} n'_{(-1)} \otimes f_{(0)}(xn_{(0)})n'_{(0)} \\
&= n_{(-2)} S_A(x_{(-1)}n_{(-1)}) [f(x_{(0)}n_{(0)})]_{(-1)} n'_{(-1)} \otimes [f(x_{(0)}n_{(0)})]_{(0)}n'_{(0)} \\
&= \epsilon_A(n_{(-1)}) S_A(x_{(-1)}) [f(x_{(0)}n_{(0)})]_{(-1)} n'_{(-1)} \otimes [f(x_{(0)}n_{(0)})]_{(0)}n'_{(0)} \\
&= {\left(\epsilon_A(n_{(-1)}) S_A(x_{(-1)}) \otimes 1\right) \Delta^A\left[f(x_{(0)}n_{(0)})n'\right]} \\
&= {\left(S_A(x_{(-1)}) \otimes 1\right) \Delta^A\left[f(x_{(0)}n)n'\right]} \\
&= S_A(x_{(-1)}) [f(x_{(0)}n)]_{(-1)} n'_{(-1)} \otimes [f(x_{(0)}n)]_{(0)}n'_{(0)} \\
&= S_A(x_{(-1)}) [f(x_{(0)}n)n']_{(-1)} \otimes [f(x_{(0)}n)n']_{(0)} \\
&= S_A(x_{(-1)}) [(nfn')(x_{(0)})]_{(-1)} \otimes [(nfn')(x_{(0)})]_{(0)} \\
&= (1 \otimes \ev_x) ((nfn')_{(-1)} \otimes (nfn')_{(0)}),
\end{align*}
as required.
\end{proof}

\begin{theorem}\label{theorem:AdjointWithhom}
For any $X \in \psAM\Mod^B_N$, the functor $(X \otimes_N -)\colon \psAN\Mod^B_N \to \psAM\Mod^B_N$ is left adjoint to $\psM\hom(X,-)\colon \psAM\Mod^B_N \to \psAN\Mod^B_N.$
\end{theorem}

\begin{proof}
To show that the functor $\psM\hom(X,-)\colon \psAM\Mod^B_N \to \psAN\Mod^B_N$ is right adjoint to the functor $(X \otimes_N -)\colon {\psAN\Mod^B_N} \to {\psAM\Mod^B_N}$, we show that there is a natural isomorphism
\begin{align*}
\phi\colon \psAM\Hom^B_N(X \otimes_N Y, Z) &\to \psAN\Hom^B_N(Y, \psM\hom(X,Z)) \\
f & \mapsto (\phi(f)\colon y \mapsto f(- \otimes y))
\end{align*}
We need to verify that $f(- \otimes y) \in \psM\hom(X,Z) \subseteq \psM\Hom(X,Z)$.  For this, it suffices to show that $f(-\otimes y)_{(-1)} \otimes f(-\otimes y)_{(0)} = y_{(-1)} \otimes f(-\otimes y_{(0)})$ and $f(-\otimes y)_{(0)} \otimes f(-\otimes y)_{(1)} = f(-\otimes y_{(0)}) \otimes y_{(1)}$, for all $y \in Y$.

For any $x \in X$, we have
\begin{align*}
(1 \otimes \ev_x)f(-\otimes y)_{(-1)} \otimes f(-\otimes y)_{(0)} &= f(-\otimes y)_{(-1)} \otimes f(-\otimes y)_{(0)}(x) \\
&= S(x_{(-1)}) [f(x_{(0)} \otimes y)]_{(-1)} \otimes [f(x_{(0)} \otimes y)]_{(0)} \\
&= S(x_{(-2)}) x_{(-1)}y_{(-1)} \otimes f(x_{(0)} \otimes y_{(0)}) \\
&= \epsilon(x_{(-1)}) y_{(-1)} \otimes f(x_{(0)} \otimes y_{(0)}) \\
&= y_{(-1)} \otimes f(x \otimes y_{(0)}), \\
&= (1 \otimes \ev_x)(y_{(-1)} \otimes f(- \otimes y_{(0)}))
\end{align*}
where we have used that $f$ is an $A$-comodule morphism.  {Similarly, we show that $f(-\otimes y)_{(0)} \otimes f(-\otimes y)_{(1)} = f(-\otimes y_{(0)}) \otimes y_{(1)}$ so that $\phi(f) \in \psM\hom(X,Z)$.  These equalities also establish that $\phi$ is an $A$-$B$-bicomodule morphism.  Furthermore, $\phi$ is clearly an $N$-bimodule morphism.  It is standard to check that $\phi$ is an isomorphism.}
\end{proof}

\begin{corollary}
There exists a category $\psAM{\underline{\Mod}}^B_N$, enriched over $\psAN\Mod^B_N$, whose underlying ($\bC$-linear) category is canonically isomorphic to $\psAM\Mod^B_N$.  The $\psAN\Mod^B_N$-valued hom is given by $\psM\hom(-,-)$.
\end{corollary}

\begin{proof}
This follows from Theorem \ref{theorem:AdjointWithhom} and \cite[\textsection6]{JanelidzeKelly01}.
\end{proof}

\begin{remark}\label{remark:CoactionForLeftModules}
Similar results hold for the functor $\hom_M(X,-)\colon \psAN\Mod^B_M \to \psAN\Mod^B_N$.  Here, the $A$-$B$-bicomodule structure of $\hom_M(X,Y)$ is given by
\begin{align*}
\Delta'\colon {\Hom_M(X,Y)} &\to {\Hom_M(X,A \otimes_\bC Y \otimes_\bC B)} \\
f &\mapsto \left(x \mapsto [f(x_{( 0 )})]_{( -1 )} S^{-1}(x_{( -1 )}) \otimes [f(x_{( 0 )})]_{( 0 )} \otimes [f(x_{( 0 )})]_{( 1 )} S(x_{( 1 )}) \right).
\end{align*}
\end{remark}

\subsubsection{The duality ${}^\vee(-) = \psM\hom(-,M)$}

We will now assume that $M \in \psAM\Mod^B_N$.  This will be the case for the applications we have in mind.

\begin{proposition}
For any $X \in \psAM\lproj^B_N$, we have $\psM\hom(X,M) \in \psAN\rproj_M^B$.
\end{proposition}

\begin{proof}
It is standard that $\psM\hom(X,M) \in \rproj_M$.  The rest follows from Proposition \ref{proposition:AllStructuresOnHom}.
\end{proof}

\begin{definition}\label{definition:Vee}
We write $\lvee (-)$ for the contravariant functor $\psM\hom(-,M)\colon \psAM\lproj^B_N \to \psAN\rproj_M^B$, and $(-)^\vee$ for the contravariant functor $\hom_M(-,M)\colon \psAN\rproj_M^B \to \psAM\lproj_N^B$.
\end{definition}

\begin{theorem}\label{theorem:TensorWithDual}\label{theorem:ComoduleAlgebraDuality}
\begin{enumerate}
\item\label{enumerate:Duality1} The contravariant functors $(-)^\vee\colon \psAN\rproj^B_M \to \psAM\lproj^B_N$ and $\lvee(-)\colon \psAM\lproj^B_N \to \psAN\rproj^B_M$ form a duality.
\item\label{enumerate:Duality2} For any $X \in \psAM\lproj^B_N$, the functor $(X \otimes_N -)\colon \psAN\Mod^B_N \to \psAM\Mod^B_N$ is left adjoint to $(\lvee X \otimes_M -)\colon \psAM\Mod^B_N \to \psAN\Mod^B_N$.
\item\label{enumerate:Duality3} For any $Y \in \psAN\rproj^B_M$, the functor $(Y \otimes_M -)\colon \psAM\Mod^B_N \to \psAN\Mod^B_N$ is left adjoint to $(Y^\vee \otimes_N -)\colon \psAN\Mod^B_N \to \psAM\Mod^B_N$.
\end{enumerate}
Moreover: for any $X \in \psAM\lproj^B_N$, we have $(\lvee X \otimes_M -) \cong {\psM\hom(X,-)} \cong {\psM\Hom(X,-)}$. 
\end{theorem}

\begin{proof}
For the first statement, we start by considering $X \in \psAM\lproj^B_N$ and $Y \in \psAN\rproj^B_M$.  It is well-known that $X \to (\lvee X)^\vee\colon x \mapsto \ev_x$ and $Y \to \lvee{}(Y^\vee)\colon y \mapsto \ev_y$ are isomorphisms in $\psM\lproj$ and $\rproj_M,$ respectively.  We check that these morphisms are morphisms (and hence isomorphisms) in $\psAM\lproj^B_N$ and $\psAN\rproj^B_M$, respectively.  Let $x \in X$ and let $\phi \in \lvee X$.  We have
\begin{align*}
(\id \otimes \ev_\phi) \Delta^A(\ev_x) &= (\id \otimes \ev_\phi) \left((\ev_x)\m1 \otimes (\ev_x)\0 \right) \\
&= (\ev_x)\m1 \otimes (\ev_x)\0(\phi) \\
&= S(\phi\m1)[\ev_x(\phi\0)]\m1 \otimes [\ev_x(\phi\0)]\0 \\
&= S(\phi\m1)[\phi\0(x)]\m1 \otimes [\phi\0(x)]\0.
\end{align*}
It follows from Remark \ref{remark:CoactionForLeftModules} that $\phi\m1 \otimes \phi\0(x) = [\phi(x\0)]\m1 S^{-1}(x\m1) \otimes [\phi(x\0)]\0$, so that
\begin{align*}
(\id \otimes \ev_\phi) \Delta^A(\ev_x) &= S\left([\phi(x\0)]_{(-2)} S^{-1}(x_{(-1)})\right) [\phi(x\0)]_{(-1)} \otimes [\phi(x\0)]\0 \\
&= (x_{(-1)}) S\left([\phi(x\0)]_{(-2)}\right) [\phi(x\0)]_{(-1)} \otimes [\phi(x\0)]\0 \\
&= (x_{(-1)}) \e\left([\phi(x\0)]_{(-1)}\right) \otimes [\phi(x\0)]\0 \\
&= (x_{(-1)}) \otimes \phi(x\0) \\
&= (\id \otimes \ev_\phi) ((x_{(-1)}) \otimes \ev_{x\0}).
\end{align*}
This shows that $X \to (\lvee X)^\vee\colon x \mapsto \ev_x$ is a left $A$-comodule map.  The other coactions are verified in a similar way.  Moreover, these isomorphisms are natural in $X$ and $Y$, respectively.  Hence, there are natural equivalences $1 \to \lvee{}((-)^\vee)$ and $1 \to (\lvee{}(-))^\vee$, finishing the proof of (\ref{enumerate:Duality1}).

We continue by showing the second statement.  As $X \in \psM\lproj,$ the map
\begin{align*}
\varphi\colon \psM\Hom(X,M) \otimes_N Y \to \psM\Hom(X,Y) && \phi \otimes y \mapsto \phi(-)y 
\end{align*}
is an isomorphism (\cite[II.4.2]{BourbakiAlgebraI13}).  Recall from Remark \ref{remark:Abouthom} that $\psM\Hom(X,-) = \psM\hom(X,-)$.  To verify that $\varphi$ is an $A$-comodule morphism, consider (for any $x \in X$)
\begin{align*}
(1 \otimes \ev_x) \left((\phi(-)y)\m1 \otimes (\phi(-)y)\0\right)
(\phi(-)y)\m1 \otimes (\phi(-)y)\0(x) \\
&= S(x\m1) [\phi(x\0)y]\m1 \otimes [\phi(x\0)y]\0 \\
&= S(x\m1) \phi(x\0)\m1 y\m1 \otimes \phi(x\0)\0 y\0\\
&= \phi\m1 y\m1 \otimes \phi\0(x) y\0 \\
&= (1 \otimes \ev_x) \left(\phi\m1 y\m1 \otimes \phi\0 y\0 \right), \\
\end{align*}
so that $(\phi(-)y)\m1 \otimes (\phi(-)y)\0 = \phi\m1 y\m1 \otimes \phi\0 y\0$, and
\[
(\phi\otimes y)\m1 \otimes (\phi\otimes y)\0 = \phi\m1 y\m1 \otimes \phi\0 \otimes y\0.
\]
Showing that $\varphi$ is a $B$-comodule morphism is similar.  This shows that $\lvee X \otimes_N - \cong \psM\hom(X,-)$.  It then follows from Theorem \ref{theorem:AdjointWithhom} that $X \otimes_N -$ is left adjoint to $(\lvee X \otimes_M -)$.

The last statement is follows from the previous statements.
\end{proof}

\begin{example}\label{example:Contragredient} We consider the case where $B = M = N = \bC$.  Here, the category $\psA\lproj$ is the category of finite-dimensional left $A$-comodules.  Let $X \in \psA\lproj$, then $\lvee X = \hom(X,\bC) = \Hom(X,\bC)$ has an $A$-comodule stucture given by $\Delta(f) = f_{(-1)} \otimes f_{(0)}$ where
\[ f_{(-1)} \otimes f_{(0)}(x) = S(x_{-1}) [f(x_{(0)})]_{(-1)} \otimes [f(x_{(0)})]_{(0)} =  S(x_{(-1)}) \otimes [f(x_{(0)})]\]
where we have used that $\Delta(1) = 1 \otimes 1$ in $\bC$.  Hence, $\lvee X$ is the usual contragredient comodule.
\end{example}

\section{Dualities for comodule differential graded modules}\label{Section:Dualsdgas}

This section is preliminary in nature.  Let $A$ be a Hopf algebra.  We give the definition of a $A$-comodule dg algebra $\D$ and apply the results from \ref{Section:ComoduleDualities} to discuss the dual of a finitely generated projective $\D$-module.  For this, we use that the category of cochain complexes is monoidally equivalent to the category of comodules over the Pareigis Hopf algebra.

Our main result is Theorem \ref{theorem:DualitiesForDg} below.  Remark \ref{remark:DualsForDGModules} collects most properties in this section that we will use in this paper.

\subsection{Differential graded algebras and modules}\label{subsection:dgam}

Throughout, we fix a Hopf algebra $A$.  A \emph{graded $A$-comodule} is an $A$-comodule $V$ together with $A$-subcomodules $V^i \subseteq V$ (for all $i \in \bZ$) such that $V = \oplus_{i \in \bZ} V^i$.  An element $v \in V$ is called \emph{homogeneous of degree $i \in \bZ$} if $v \in V^i$.  For a homogeneous $v \in V$, we also write $|v|$ for the degree of $v$.

A linear map $f\colon V \to W$ between graded $A$-comodules $V$ and $W$ is said to be \emph{homogeneous of degree $j \in \bZ$} if $f(V^i) \subseteq W^{i+j}$, for all $i \in \bZ$.  The space of all linear maps of degree $j$ will be denoted by $\Lin^j(V,W)$.  We write $\Lin(V,W)$ for the graded vector space $\bigoplus_{j \in \bZ} \Lin^j(V,W)$.

Note that we do not require the morphisms in $\Lin(V,W)$ to be $A$-comodule morphisms.  Furthermore, if there are only finitely many $i \in \bZ$ for which $V^i \not= 0$ (so that $V$ is bounded in degree), then $\Lin(V,W)$ is the space of all $\bC$-linear maps from $V$ to $W$.

We write $\psA\Lin^{i}(V,W)$ and $\psA\Lin(V,W)$ for the subspaces of $\Lin^{i}(V,W)$ and $\Lin(V,W)$ consisting of $A$-comodule morphisms, respectively.

The category $\psA\GrMod$ is the category whose objects are graded $A$-comodules and whose morphisms are homogeneous $A$-comodule maps of degree 0 (thus, elements of $\psA\Lin^0(-,-)$).  The category $\psA\GrMod$ has a monoidal structure, given by
\[(V \otimes W)^n = \bigoplus_{i+j=n} V^{i} \otimes_\bC W^j,\]
where $V$ and $W$ are graded $A$-comodules; here, $V^{i} \otimes_\bC W^j$ is the usual tensor product of $A$-comodules.  

A \emph{differential graded $A$-comodule} over the Hopf algebra $A$ is a pair $(V, \exd)$ where $V$ is a graded $A$-comodule and $\exd \in \psA\Lin^1(V,V)$ is an $A$-comodule morphism of degree 1 satisfying $\exd^2 = 0$.  If $(V, \exd_V)$ and $(W, \exd_W)$ are differential graded $A$-comodules, we give the space $\Lin(V,W) = \bigoplus_j \Lin^j(V,W)$ the structure of a differential graded vector space by defining
\[\exd_{\Lin(V,W)}(f) = [\exd,f] \coloneqq \exd_W \circ f - (-1)^{|f|} f \circ \exd_V,\]
for homogeneous $f \in \Lin(V,W)$.  The differential $\exd_{\Lin(V,W)} = [\exd,-]$ is called the \emph{graded commutator}.

A \emph{cochain $A$-comodule morphism} $f\colon (V, \exd_V) \to (W, \exd_W)$ is an $A$-comodule morphism $f\colon V \to W$ of degree 0 such that $[\exd,f] = 0$.  Explicitly, it is an $A$-comodule morphism $f\colon A \to B$ satisfying $f(V^i) \subseteq W^i$ (for all $i \in \bZ$) and $\exd_W \circ f = f \circ \exd_V$.

The category $\psA\dgMod$ of differential graded $A$-comodules and $A$-comodule cochain morphisms is a monoidal category.  The monoidal structure is given by 
\begin{align*}
(V\otimes W)^n &\cong \bigoplus_{i+j = n} V^i \otimes W^j\\
\exd_{V \otimes W} (v \otimes w) &= \exd_V (v) \otimes v^j + (-1)^{|v|} v \otimes \exd_W(w)
\end{align*}
where $v, w$ are homogeneous.

An \emph{$A$-comodule dg algebra} is an algebra object in the category $\psA\dgMod$.  Explicitly, an $A$-comodule dg algebra is a differential graded $A$-comodule $\D = (D, \exd)$ together with cochain $A$-comodule morphisms $\mu\colon D \otimes D \to D$ and $\eta\colon I \to D$ such that $(D,\mu,\eta)$ is a graded algebra and $\exd\colon D \to D$ is a degree 1 map satisfying the graded Leibniz rule:
\[ \exd(ab) = \exd(a)b + (-1)^{|a|} a \exd(b),\]
for all $a,b \in D$ (where $a$ is homogeneous).

A \emph{left $\D$-dg module} is an object in the Eilenberg-Moore category of the monad $(\D \otimes -)$.  Specifically, a left $\D$-dg module is a cochain complex $(X, \exd_X)$ such that $X$ is a graded left $D$-module and for all homogeneous $a \in D$, we have $\exd_X(ax) = \exd(a)x + (-1)^{|a|}a \exd_X (x)$.  Right dg modules are defined in an analogous way.

We write $\psAD\dgproj$ and $\psA\dgproj_\D$ for the full subcategories of $\psAD\dgMod$ and $\psA\dgMod_\D$ consisting of those dg modules that are finitely generated projective as $D$ modules (thus, disregarding the differential graded structure).

\subsection{Equivariant differential graded algebras as comodule algebras}

We now consider the Pareigis Hopf algebra (see \cite{Pareigis81}) $B = \bC \langle x,y^\pm \rangle / (xy+yx,x^2)$ with
\begin{align*}
\Delta(x) &= 1 \otimes x + x \otimes y^{-1} & \epsilon(x) &= 0 \\
\Delta(y) &= y \otimes y & \epsilon(y) &= 1 \\
S(x) &= -xy & S(y) &= y^{-1}.
\end{align*}
Let $B'$ be the Hopf algebra $\bC[y^\pm]$ with the usual structure maps (for all $i \in \mathbb{Z}$, $y^i$ is grouplike; thus, $B'$ is isomorphic to the group algebra of $\bZ,+$).  There is a surjective Hopf algebra map $\varpi\colon B \to B'$ given by $\varpi(x) = 0$ and $\varpi(y) = y$.

The following result is well-known (see \cite[\textsection 3.2]{BrzezinskiWisbauer03} for the first statement and \cite{EstradaGarciaTorrecillas07, Pareigis81} for the second statement).

\begin{theorem}\label{theorem:Pareigis}
\begin{enumerate}
\item There is an isomorphism of monoidal categories $F\colon \GrMod \to \Mod^{B'}$.
\item There is an isomorphism of monoidal categories $F\colon \dgMod \to \Mod^{B}$.
\end{enumerate}
\end{theorem}

The first isomorphism $F\colon \Mod^{B'} \to \GrMod$ is given by mapping a graded vector space $V = \oplus_i V^i$ to a $B'$-comodule $(V,\Delta^{B'}_V)$ with $\Delta^{B'}_V(v) = v \otimes y^{|v|}$, for homogeneous $v \in V$.  Conversely, given a $B'$-comodule, we consider the grading on $V$ given by $V^i = \{v \in V\mid \Delta_V = v \otimes y^i\}$.

The second isomorphism $F\colon \dgMod \to \Mod^{B}$ maps a differential graded vector space $(V,\exd_V)$ to the $B$-comodule $(V, \Delta^B_V)$ with $\Delta^B_V(v) = v \otimes y^{|v|} + \exd v \otimes y^{|v|+1}x$, for homogeneous $v \in V$.  The inverse isomorphism can be described as follows.  Starting from a $B$-comodule $(V \Delta^B_V)$, we give $V$ the structure of a graded vector space via $V^i = \{v \in V\mid (1 \otimes \varpi) \Delta_V = v \otimes y^i\}$.  The map $\exd_V\colon V \to V$ is then defined via $\exd_V (v) \otimes y^{|v|+1}x = \Delta_V(v) - v \otimes y^{|v|}$, for homogeneous $v \in V$.

\begin{remark}
\begin{enumerate}
\item The grading of a $B$-comodule $V$ is obtained by restricting the $B$-coaction $\Delta^B_V$ to a $B'$-coaction $\Delta^{B'}_V = (1 \otimes \varpi)\circ \Delta_V$ and invoking the first statement of Theorem \ref{theorem:Pareigis}.
\item To see that $\exd_V\colon V \to V$ is a well-defined map of degree 1 with $\exd^2 = 0$, one needs that $\Delta_V(v) = v \otimes y^{|v|} + v' \otimes y^{|v|+1}x$ and $\Delta_V(v') = v' \otimes y^{i+1}$, for all homogeneous $v \in V$.  This follows from the coassociativity of the coaction (see \cite[Proof of Theorem 18]{Pareigis81} for details).
\item The Pareigis Hopf algebra $B$ can be obtained via a bosonisation process (see \cite{Majid97}), namely as $\bC[Y^\pm] {>\!\!\!\triangleleft\kern-.33em\cdot\, } \bC[x]/(x^2)$ where the $\bC[Y^\pm]$-coaction of $\bC[x]/(x^2)$ is given by $\Delta(x) = y \otimes x$.
\end{enumerate}
\end{remark}

We have the following extension of Theorem \ref{theorem:Pareigis}.

\begin{proposition}\label{proposition:PareigisEquivariant}
There is an isomorphism of monoidal categories $\psA F \colon \psA\dgMod \to \psA\Mod^{B}$.
\end{proposition}

\begin{proof}
Let $(V, \exd) \in \psA\dgMod^B$.  As the functor $F\colon \dgMod \stackrel{\sim}{\rightarrow} \Mod^B$ is the identity on the underlying vector spaces, we may endow $F(V,\exd)$ with the same $A$-comodule structure as $V$.  To show that this induces a functor $\psA F\colon \psA\dgMod {\rightarrow} \psA\Mod^B$, we need to show this $A$-coaction turns $F(V, \exd)$ into an $A$-$B$-comodule. 

As we have for a homogeneous $v \in V$ of degree $i$:
\begin{align*}
(\Delta^A \otimes 1) \circ \Delta^B(v) &= v_{(-1)} \otimes v_{(0)} \otimes y^i + (\exd v)_{(-1)} \otimes (\exd v)_{(0)} \otimes y^{i+1}x, \\
(\Delta^B \otimes 1) \circ \Delta_A(v) &= v_{(-1)} \otimes \Delta^B(v_{(0)}),
\end{align*}
we see that $F(V, \exd)$ is an $A$-$B$-bicomodule if and only if, for every homogeneous $v \in V^i \subseteq V$ we have that $v_{(0)} \in V^i$ and $(\exd v)_{(-1)} \otimes (\exd v)_{(0)} = v_{(-1)} \otimes \exd (v_{(0)})$.  This means that $\psA F\colon \psA\dgMod {\rightarrow} \psA\Mod^B$ is indeed well-defined.

For the other direction, we consider the functor $F^{-1}\colon \psA\Mod^B \stackrel{\sim}{\rightarrow} \psA\dgMod$ from Theorem \ref{theorem:Pareigis}.  Again, $F^{-1}(V)$ inherits an $A$-comodule structure directly from the $A$-comodule structure of $V$.  Similar as above, one checks that $F^{-1}(V)$ is an $A$-$B$-comodule; hence, the functor $(\psA F)^{-1}\colon \psA\Mod^B \stackrel{\sim}{\rightarrow} \psA\dgMod$ is well-defined.

That $\psA F$ and $\psA F^{-1}$ are isomorphisms and monoidal functors follows directly from the fact that $F$ and $F^{-1}$ are isomorphisms and monoidal functors, respectively.
\end{proof}

\begin{corollary}\label{corollary:PareigisEquivariant}
Let $\D$ and $\E$ be algebra objects in $\psA\dgMod$ (thus, $\D$ and $\E$ are $A$-comodule dg algebras).  The monoidal isomorphism $\psA F$ from Proposition \ref{proposition:PareigisEquivariant} induces isomorphisms
\begin{enumerate}
\item $\psAD\dgMod_\E \stackrel{\sim}{\rightarrow} \prescript{A}{F(\D)}\Mod^B_{F(\E)}$,
\item $\psAD\dglproj_\E \stackrel{\sim}{\rightarrow} \prescript{A}{F(\D)}\lproj^B_F(\E)$,
\item $\psAD\dgrproj_{\E} \stackrel{\sim}{\rightarrow} \prescript{A}{F(\D)}\rproj^B_{F(\E)}$.
\end{enumerate}
\end{corollary}

\begin{notation}
Recall that $\psA\Mod^B$ is a closed monoidal category.  We can use Proposition \ref{proposition:PareigisEquivariant} to transfer this property to $\psA\dgMod.$  We write $\lin(V,-)\colon \psA\dgMod \to \psA\dgMod$ for the right adjoint of $(V \otimes -)\colon \psA\dgMod \to \psA\dgMod$.
\end{notation}

Explicitly: $\lin(V,-) = \psA F^{-1}\big(\hom (\psA F(V), \psA F(-)) \big).$

\begin{corollary}
	 If $V \in \psAD\dglproj$, then $\psD\lin(V,-) = \psD\Lin(V,-)$ as differential graded vector spaces.
\end{corollary}

\subsection{Dualities for equivariant differential graded modules}

We can now directly apply the results of Appendix \ref{Section:ComoduleDualities} to this setting. 

\begin{notation}
We write $\lwedge(-)$ for the contravariant functor $\psD\lin(-,\D)\colon \psAD\dgMod \to \psA\dgMod_\D$ and $(-)^\wedge$ for the contravariant functor $\lin_\D(-,\D)\colon \psA\dgMod_\D \to \psAD\dgMod$.
\end{notation}

\begin{theorem}\label{theorem:DualitiesForDg}
Let $X \in \psAD\dglproj_\E$.
\begin{enumerate}
\item The functor $(X \otimes_\E -)\colon \psAE\dgMod \to \psAD\dgMod$ iss left adjoint to $(\lwedge X \otimes_\D -)\colon \psAD\dgMod \to \psAE\dgMod$.
\item The contravariant functors 
\[\mbox{$\lwedge (-) \colon \psAD\dglproj_\E \to \psAE\dgrproj_\D$ and $(-)^\wedge \colon \psAE\dgrproj_\D \to \psAD\dglproj_\E$}\]
form a duality.
\end{enumerate}
\end{theorem}

\section{Differential graded calculi}

In this appendix, we recall some basic definitions about equivariant calculi that we use in this paper.  Although the results and the setting are well-known, we repeat some definitions here to establish notation.

\subsection{\texorpdfstring{Differential $*$-calculi}{Differential *-calculi}}\label{subsection:DifferentialCalculi}

Let $A$ be a Hopf $*$-algebra.  An \emph{$A$-comodule dg $*$-algebra} is a pair $(\D^\bullet, *)$ where $\D = (D,\exd)$ is an $A$-comodule dg algebra and $(-)^*\colon D \to D$ is an involutive conjugate-linear map satisfying
\begin{align*}
(\exd e)^* &= \exd (e^*), \\
(ef)* &= (-1)^{|e| |f|}f^* e^*, \\
\Delta^A(e^*) &= e\m1^* \otimes e\0^*.
\end{align*}
An \emph{$A$-comodule $*$-algebra} is an {$A$-comodule dg $*$-algebra}, concentrated in degree 0.

\begin{definition}
An {\em $A$-covariant $*$-differential calculus} over an $A$-comodule $*$-algebra $M$ is an $A$-comodule differential graded $*$-algebra $(\Om^\bullet,\exd)$ \st $\Om^0=M$ (as $A$-comodule $*$-algebras) and $(\Om^{\bullet}, \exd)$ is generated by $\Om^0$ as a dg algebra.
\end{definition}

We use $\wedge$ to denote the multiplication between elements of a differential $*$-calculus when both elements are {homogeneous of} degree greater than $0$.  We call an element of a differential calculus a {\em form}.  A form $\omega$ is called {\em closed} if $\exd \omega = 0$.  {A {\em differential map} between two $A$-covariant differential $*$-calculi $(\Om^\bullet,\exd_{\Om})$ and $(\G^\bullet,\exd_{\G})$, defined over the same algebra $M$, is an $A$-comodule $M$-bimodule map $\f\colon \Om^\bullet \to \Gamma^\bullet$ of degree 0 such that $\f \circ \exd_{\Om} = \exd_{\Gamma} \circ \f$.}

We say that a differential $*$-calculus has {\em total dimension} $n$ if $\Om^k = 0$, for all $k>n$, and $\Om^n \neq 0$.  If, in addition, there exists an $A$-comodule $M$-bimodule isomorphism $\vol\colon \Om^n \simeq M$ which commutes with the $*$-map, then we say that $\Om^{\bullet}$ is a {\em (covariantly) $*$-orientable calculus}.  We call a choice of such an isomorphism an {\em orientation}.

\subsection{Quantum homogeneous spaces}\label{subsection:QuantumHomogeneousSpace}

A compact quantum group algebras $G$ is a Hopf $*$-algebra which is cosemisimple, and of which the Haar map $\haar$ satisfies $\haar(g^*g) > 0$, for all $g \in G$.  (We recall that compact quantum group algebras can always be completed to a compact quantum group in the sense of Woronowicz, and that this gives us an equivalence between compact quantum groups and compact quantum group algebras.)

For a right $G$-comodule $V$ with coaction $\DEL^G$, we say that an element $v \in V$ is {\em coinvariant} if $\DEL^G(v) = v \oby 1$. We denote the subspace of all coinvariant elements by $V^G$, and call it the {\em coinvariant subspace} of the coaction. We use analogous conventions for left comodules.

\begin{definition}
Let $G$ and $H$ be compact quantum group algebras.  A {\em homogeneous} right $H$-coaction on $G$ is a coaction of the form $(\id \oby \pi)  \DEL$, where $\pi\colon G \to H$ is a surjective Hopf algebra map. A {\em quantum homogeneous space} $M\coloneqq G^H$ is the coinvariant subspace of such a coaction.
\end{definition}

In this paper we will {\em always} use the symbols $G,H$ and $\pi$ in this sense. As is easily seen, $M$ is a sub\alg  of $G$.  Moreover, if $\pi$ is a Hopf $*$-algebra map, then $M$ is a $*$-subalgebra of $G$.

We now define the abelian categories $\lgmmm$ and $\psH\Mod_0$.  The category $\lgmmm$ is the full subcategory of $\psGM\Mod_M$ consisting of those $M$-bimodules $\F$ satisfying $\F M^+ \sseq M^+ \F$.  The category $\psH\Mod_0$ is the full subcategory of $\psH\Mod_M$ consisting of those objects $V$ whose left $M$-action is given by $v\cdot m = v \e(m)$.  Note that $\psH\Mod_0$ is equivalent under the obvious forgetful functor  to $\psH\!\Mod$.

If $\F \in \lgmmm$, then $\F/(M^+\F)$ becomes an object in $\psH\mod_0$ with the left $H$-coaction 
\bal \label{comodstruc0}
\DEL^H[f] = \pi(f\m1) \oby [f\0], & & f \in \F,
\eal
where $[f]$ denotes the coset of $f$ in  $\F/(M^+\F)$. We define a functor
\[\Phi\colon \lgmmm \to \psH\Mod\]
as follows:  $\Phi(\F) \coloneqq  \F/(M^+\F)$, and if $g\colon \F \to \F$ is a morphism in $\lgmmm$, then $\Phi(g)\colon \Phi(\F) \to \Phi(\F)$ is the map to which $g$ descends on $\Phi(\F)$.

If $V \in \lmhm$, then the {\em cotensor product} of $G$ and $V$, defined by
\bas
G \coby V \coloneqq \ker(\DEL^H \oby \id - \id \oby \DEL^H\colon G\oby V \to G \oby H \oby V),
\eas  
becomes an object in $\lgmmm$ by defining an $M$-bimodule structure
\begin{align} \label{rightmaction}
m \Big(\sum_i g^i \oby v^i\Big)  = \sum_i m g^i \oby v^i, & & \Big(\sum_i g^i \oby v^i\Big) m = \sum_i g^i m \oby v^i,
\end{align}
and a left $G$-coaction 
\begin{align*}
\DEL^G\Big(\sum_i g^i \oby v^i\Big) = \sum_i g^i\1 \oby g^i\2 \oby v^i.
\end{align*}
We define a functor $\Psi\colon\psH\Mod_0 \to \lgmmm$ by:
$
\Psi(V) \coloneqq G \coby V,
$
and if $\g$ is a morphism in $\psH\Mod_0$, then $\Psi(\g) \coloneqq \id \oby \g$. 

The following theorem is \cite[Theorem 1]{Tak}.

\begin{theorem}[Takeuchi's equivalence]\label{theorem:Takeuchi}
The monoidal functors $\Phi\colon \lgmmm \to \lmhm$ and $\Psi\colon\lmhm \to \lgmmm$ are quasi-inverses.  There are natural transformations given by
\begin{align}
\counit_V\colon\Phi \circ  \Psi(V) \to V, && \Big[\sum_i g^i \oby v^i\Big] \mto \sum_{i} \e(g^i)v^i \label{counit},\\
\unit_\E\colon \E \to \Psi \circ \Phi(\E), && e \mto e\m1 \oby [e\0]. \label{unit}
\end{align}
\end{theorem}

\begin{corollary}
Takeuchi's equivalence restricts to an equivalence of categories between $\sgmm$ and $\psH\mod_0$, where $\sgmm$ is the full subcategory of $\lgmmm$ consisting of finitely generated left $M$-modules, and $\psH\mod_0$ is the full subcategory of $\psH\Mod_0$ consisting of finite-dimensional comodules.
\end{corollary}

For any $\F \in \psGM\mod_0$, we refer to $\dim_\bC \Phi(\F)$ as the \emph{rank} of $\F$.

In this paper, all differential $*$-calculi over a quantum homogeneous space will be assumed to satisfy $\Om^\bullet M^+ \sseq M^+\Om^\bullet$ and to be of finite rank (that is, $\Om^\bullet \in \psGM\mod_0$). This implies that {there is} a multiplication defined on $\Phi(\Om^\bullet)$ by $[\w] \wed [\nu] \coloneqq [\w \wed \nu]$.  It follows from Theorem \ref{theorem:Takeuchi} that every element of $\Phi(\Om^k)$ is a sum of elements of the form $[\w_1] \wed \cdots \wed [\w_k]$, for $\w_i \in \Om^1$.  When working with quantum homogeneous spaces, we usually adopt the convenient notation $V^\bullet \coloneqq \Phi(\Om^\bullet)$.


\def\cprime{$'$}
\providecommand{\bysame}{\leavevmode\hbox to3em{\hrulefill}\thinspace}
\providecommand{\MR}{\relax\ifhmode\unskip\space\fi MR }
\providecommand{\MRhref}[2]{%
  \href{http://www.ams.org/mathscinet-getitem?mr=#1}{#2}
}
\providecommand{\href}[2]{#2}

\end{document}